\documentclass[a4paper,10pt]{article}

\usepackage{ifthen}
\usepackage{amssymb}
\usepackage{amsmath}
\usepackage{mathrsfs}
\usepackage{theorem}
\usepackage[all]{xypic}
\usepackage{color}
	{
		\theorembodyfont{\itshape}

		\newtheorem{theorem}{Theorem}[section]
		\newtheorem{lemma}[theorem]{Lemma}
		\newtheorem{proposition}[theorem]{Proposition}
		\newtheorem{corollary}[theorem]{Corollary}
	}

	{
		\theorembodyfont{\rmfamily}
		\newtheorem{definition}[theorem]{Definition}
		\newtheorem{postulate}[theorem]{Postulate}
		\newtheorem{remark}[theorem]{Remark}
	}

	\newenvironment{proof}{
		\goodbreak\par
		\textit{Proof.}%
	}{%
		\nopagebreak
		\hfill{\vrule width 1ex height 1ex depth 0ex}
		\medskip
		\goodbreak
	}
	\newcommand{\sizedescriptor}[2]
	{
		\ifthenelse{\equal{#1}{0}}{}{
		\ifthenelse{\equal{#1}{1}}{\big}{
		\ifthenelse{\equal{#1}{2}}{\Big}{
		\ifthenelse{\equal{#1}{3}}{\bigg}{
		\ifthenelse{\equal{#1}{4}}{\Bigg}{
		#2}}}}}
	}
	
	\newcommand{\nc}[1][\land]{~{#1}~}
	\newcommand{\nl}[1][\land]{\nc[#1]\\&\nc[#1]}
	\newcommand{\proven}[1]{\underline{#1}\vspace{0.2em}\\}

	\newcommand{\ie}[1][~]{i.e.{#1}}
	
	\newcommand{\eg}[1][~]{e.g.{#1}}
	
	\newcommand{\etc}[1][~]{etc.{#1}}

	\newcommand{\intermission}{\bigskip\medskip}
	\newcommand{\df}[1]{\emph{#1}}
	\newcommand{\ism}{\cong}
	\newcommand{\equ}{\sim}
	\newcommand{\dfeq}{:=}

	\newcommand{\sepdfeq}{\quad\dfeq\quad}
	\newcommand{\apart}{\mathrel{\#}}
	\newcommand{\id}[1][]{\textrm{Id}_{#1}}
	\newcommand{\impl}{\Rightarrow}
	
	\newcommand{\rstr}[1]{\left.{#1}\right|}
	\newcommand{\cnct}{{{:}{:}}}
	\newcommand{\im}{\text{im}}
	\newcommand{\ball}[3][]{B_{#1}\left(#2, #3\right)}
	\newcommand{\parto}{\mathrel{\rightharpoonup}}
	\newcommand{\insarg}{\text{\textrm{---}}}

	\newcommand{\all}[3]{\forall\, #1 \,{\in}\, #2\,.\left(#3\right)}

	\newcommand{\xall}[3]{\forall\, #1 \,{\in}\, #2\,.\,#3}
	\newcommand{\xsome}[3]{\exists\, #1 \,{\in}\, #2\,.\,#3}
	\newcommand{\xexactlyone}[3]{\exists!\, #1 \,{\in}\, #2\,.\,#3}

	\newcommand{\st}[3][auto]{\sizedescriptor{#1}{\left}\{#2\;\sizedescriptor{#1}{\middle}|\;#3\sizedescriptor{#1}{\right}\}}
	\newcommand{\mlst}[5][\Big]{\begin{align*} #2 #1\{ #3 \ #1|& \ #4 #1\}#5 \end{align*}}
	
	\newcommand{\pst}{\mathscr{P}}
	\newcommand{\finseq}[1]{{#1}^*}
	\newcommand{\C}{\mathscr{C}}

	\newcommand{\NN}{\mathbb{N}}
	\newcommand{\ZZ}{\mathbb{Z}}
	\newcommand{\QQ}{\mathbb{Q}}
	\newcommand{\RR}{\mathbb{R}}
	\newcommand{\intoo}[3][\RR]{{#1}_{(#2, #3)}}
	\newcommand{\intcc}[3][\RR]{{#1}_{[#2, #3]}}
	
	\newcommand{\intco}[3][\RR]{{#1}_{[#2, #3)}}

	\newcommand{\one}{\mathbf{1}}
	\newcommand{\unit}{*}
	\newcommand{\od}{\mathbb{D}}
	\newcommand{\nnr}{{\RR_{\geq 0}}}
	\newcommand{\strat}[3][]{{}_{#1}{#2}_{#3}}
	\newcommand{\prt}{\mathbf{W}}
	\newcommand{\psd}{\mathbf{V}}
	\newcommand{\U}{\mathbf{U}}
	\newcommand{\mtr}[1]{\mathbf{#1}}
	\newcommand{\cmpl}[1]{\widehat{#1}}
	\newcommand{\cmtr}[1]{\cmpl{\mtr{#1}}}

	\newcommand{\locd}{\mathscr{L}}
	\newcommand{\Cauchy}{\mathrm{Cauchy}}
	\newcommand{\cs}{\mathrm{cs}}
	\newcommand{\tuple}[6][auto]{\strat[#2]{\sizedescriptor{#1}({#4}_{#6}, {#5}_{#6}\sizedescriptor{#1})}{{#6} \in \NN_{< #3}}}
	\newcommand{\norm}[1][\insarg]{\|{#1}\|}
	\newcommand{\fdis}{\rightarrow}
	\newcommand{\dis}{\mathop{\leftrightarrow}}
	\newcommand{\proto}[1][]{\strat[#1]{\prt}{\od}}
	\newcommand{\protoaux}[1][]{\strat[#1]{\mathbf{A}}{\od}}
	\newcommand{\pseudo}[1][]{\strat[#1]{\psd}{\od}}
	\newcommand{\pseudoaux}[1][]{\strat[#1]{\mathbf{B}}{\od}}
	\newcommand{\Ury}[1][]{\strat[#1]{\U}{\od}}
	\newcommand{\gentuple}[1][n]{\tuple{#1}{l}{a}{\alpha}{i}}
	\newcommand{\age}[1]{\text{age}(#1)}
	\newcommand{\lnth}[1]{\text{lnth}(#1)}
	\newcommand{\atuple}[3]{\tuple{\age{#1}}{\lnth{#1}}{#1}{#2}{#3}}
	\newcommand{\et}[1][0]{\strat[#1]{()}{}}
	\newcommand{\fpi}[1][]{\mathrm{FPIsom}_{#1}(\mtr{X}, \U)}
	\newcommand{\isom}{\mathrm{Isom}(\mtr{X}, \U)}
	\newcommand{\aut}{\mathrm{Aut}(\U)}
	\newcommand{\ext}[1][]{\mathscr{E}_{#1}}
	\newcommand{\prms}[1][]{\mathscr{P}_{#1}}
	\newcommand{\extm}[1][]{\mathrm{ext}_{#1}}

\title{Continuity and Algebraic Structure of the Urysohn space}

\author{
Davorin Le\v{s}nik\vspace{1ex}\\
Department of Mathematics\\
Darmstadt University of Technology, Germany\\
\texttt{lesnik@mathematik.tu-darmstadt.de}\vspace{1ex}\\
Department of Mathematics and Physics\\
University of Ljubljana, Slovenia\\
\texttt{davorin.lesnik@fmf.uni-lj.si}
}
\begin{document}

	\maketitle
	
	\begin{abstract}
		The Urysohn space is a complete separable metric space, universal among separable metric spaces for extending finite partial isometries into it. We present an alternative construction of the Urysohn space which enables us to show that extending isometries can be done in a canonical and continuous way, and allows us to equip the Urysohn space with algebraic structure. This is achieved in a constructive setting without assuming any choice principles.
	\end{abstract}
	
%	\begin{keywords}
%		Urysohn space, metric spaces, disgroups, disrings, constructive mathematics
%	\end{keywords}

	%\tableofcontents

	\section{Introduction}\label{Section: introduction}
		In~\cite{Urysohn_PS_1927:_sur_un_espace_metrique_universel} the Urysohn space was defined as a complete separable metric space with the property that any partial isometry with a finite domain (which we call a \df{finite partial isometry}) from a separable metric space into the Urysohn space has an extension to the whole space. A model of a Urysohn space was constructed, and shown that it is unique up to isometric isomorphism.
		
		The purpose of this paper is threefold.
		\begin{itemize}
			\item\proven{Show that extending isometries can be done in a canonical and continuous way.}
				To this end we choose to model the Urysohn space in an alternative way. The additional structure we get enables us to construct an explicit mapping from finite partial isometries to total isometries of a separable metric space into the Urysohn space. Furthermore, we suitably topologize the spaces of isometries, and show that the given extension of isometries is continuous.
			\item\proven{Equip the Urysohn space with algebraic structure.}
				We identify a structure (called a \df{disring}) which helps us with the construction of the Urysohn space, and then we verify that the Urysohn space is a ``complete normed module'' over it.
			\item\proven{Prove these theorems in a weak constructive setting.}
				Constructivism in mathematics is, roughly speaking, proving existence theorems by explicitly constructing the object in question, as opposed to, say, assuming its non-existence, and deriving a contradiction~\cite{Troelstra_AS_Dalen_D_1988:_constructivism_in_mathematics_volume_1}. Reasons for doing mathematics in a constructive (instead of ``classical'') way can be philosophical (as was the case with Brouwer who essentially started the field), or practical (a physicist might not be impressed by a theorem stating that a solution to his equations exists in principle, without actually providing the solution). For us (for the purposes of this paper) it is merely proving statements in a more general setting than the standard ZFC axioms. Constructive mathematics can be formalized; basically it is classical mathematics without assuming the law of excluded middle (``every statement is either true or false'') and the axiom of choice, with perhaps some further axioms omitted, or some alternative ones added. The exact details vary with the version of constructivism; there are several~\cite{Bridges_DS_Richman_F_1987:_varieties_of_constructive_mathematics}. We prove our theorems in predicative IZF (intuitionistic Zermelo-Fraenkel) theory with natural numbers which is a common core of mostly considered varieties of constructivism (as well as classical mathematics).
		\end{itemize}
		
		This paper can be seen as a continuation of my paper~\cite{Lesnik_D_2009:_constructive_urysohn_universal_metric_space}, filling in some details and expanding the constructive development of the Urysohn space by showing its uniqueness, continuity of extensions, and developing the algebraic structure. However, it is written as a standalone paper, and aimed at more general than just constructive audience.
		
		\subsection{Outline of the Paper}
		
			\begin{itemize}
				\item \textbf{Section~\ref{Section: introduction}:} Introduction\\
					We explain the theme, contents and notation of the paper, as well as provide a soft introduction to constructive mathematics for a classical mathematician.
				\item \textbf{Section~\ref{Section: dis_structures}:} Disgroups and Disrings\\
					We introduce \df{disgroups} and related structures, for which it is meaningful to say, what an internal distance is (the prime example are non-negative real numbers with the absolute value of a difference for the distance), as well as metric spaces (and their generalizations) with distances in them. This structure is later used for the construction and algebra of the Urysohn space.
				\item \textbf{Section~\ref{Section: countable_Urysohn}:} Countable Urysohn Space\\
					As usual, we first construct a ``countable version'' of the Urysohn space (the completion of which is then the actual Urysohn space). It is given as an inductive structure, where the inductive step is essentially the extension of an isometry by one point. For technical reasons it is split into three stages: first, elements representing all combinations of distances to the new point are added, then these are cut down to the ones which make sense (read: satisfy triangle inequality), then points at zero distance are identified, to obtain a metric space. A model of such a space is explicitly given, and proven that it is unique up to isometric isomorphism.
				\item \textbf{Section~\ref{Section: reals_and_completions}:} Real Numbers and Metric Completion\\
					An intermezzo to discuss complete metric spaces in constructive settings. In particular, we recall the completion by \df{locations} which is much better suited for the construction of the Urysohn space (than the classical one with Cauchy sequences).
				\item \textbf{Section~\ref{Section: complete_Urysohn}:} Complete Urysohn Space\\
					We verify that the completion of the space from Section~\ref{Section: countable_Urysohn} satisfies the Urysohn properties, and that it is unique such up to isometric isomorphism.
				\item \textbf{Section~\ref{Section: continuity}:} Continuity of Extensions\\
					We topologize domains and codomains of extension mappings and prove their continuity (in fact, Lipschitz continuity with constant $1$ in a suitable sense).
				\item \textbf{Section~\ref{Section: algebra_of_Urysohn}:} Algebraic Structure of the Urysohn Space\\
					We show that the operation $\dis$ on the base disgroup induces an associative disgroup structure on the Urysohn space, and if we start with a disring, we obtain what is essentially a ``Banach space over the disring'' (in particular, the operations are continuous).
				\item \textbf{Section~\ref{Section: applications}:} Applications\\
					A few basic applications of continuity and algebra of the Urysohn space.
				\item \textbf{Section~\ref{Section: concluding_remarks}:} Concluding Remarks\\
					The concluding section contains remarks and some additional results, to put the paper in a wider perspective. Several further questions are posed.
			\end{itemize}
		
		\subsection{Notation and Style}
			
			\begin{itemize}
				\item
					Most of the paper (when the writer is the subject) is written in first person plural (as is usual in mathematical texts), but occasionally I use first person singular. The former is used more formally (and when I expect the reader to go along with and except what is written), whereas the latter is used when I want to express my personal style, opinion, preference or suggestion.
				\item
					Number sets are denoted by $\NN$ (natural numbers), $\ZZ$ (integers), $\QQ$ (rationals), and $\RR$ (reals). Zero is considered a natural number (so $\NN = \{0, 1, 2, 3,\ldots\}$).
				\item
					Subsets of number sets, obtained by comparison with a certain number, are denoted by the suitable order sign and that number in the index. For example, $\NN_{< 42}$ denotes the set $\st{n \in \NN}{n < 42} = \{0, 1, \ldots, 41\}$ of all natural numbers smaller than $42$, and $\RR_{\geq 0}$ denotes the set $\st{x \in \RR}{x \geq 0}$ of non-negative real numbers.
				\item
					Intervals between two numbers are denoted by these two numbers in brackets and in the index. Round, or open, brackets $(\ )$ denote the absence of the boundary in the set, and square, or closed, brackets $[\ ]$ its presence; for example $\intco[\NN]{5}{10} = \st{n \in \NN}{5 \leq n < 10} = \{5, 6, 7, 8, 9\}$ and $\intcc{0}{1} = \st{x \in \RR}{0 \leq x \leq 1}$.
				\item
					Given a map $a\colon N \to A$ where $N$ is a subset of natural numbers, we often write simply $a_k$ instead of $a(k)$ for the value of $a$ at $k \in N$.
				\item
					The set of maps from $A$ to $B$ is written as the exponential $B^A$.
				\item
					The set of finite sequences of elements in $A$ is denoted by $\finseq{A}$.
				\item
					Concatenation of sequences $a$ and $b$ is $a \cnct b$.
				\item
					Given sets $A \subseteq X$, $B \subseteq Y$ and a map $f\colon X \to Y$ with the image $\im(f) \subseteq B$, the restriction of $f$ to $A$ and $B$ is denoted by $\rstr{f}_A^B$. When we restrict only the domain or only the codomain, we write $\rstr{f}_A$ and $\rstr{f}^B$, respectively.
				\item
					A one-element set (a singleton) is denoted by $\one$ (and its sole element by $*$).
				\item
					The onto maps are called surjective, and the one-to-one maps injective.
				\item
					The quotient of a set $X$ by an equivalence relation $\equ$ is denoted by $X/_\equ$. Its elements --- the equivalence classes --- are denoted by $[x]$ where $x \in X$ (\ie if $q\colon X \to X/_\equ$ is the quotient map, then $[x] \dfeq q(x)$).
				\item
					The coproduct (disjoint union) is denoted by $+$ in the binary case, and by $\coprod$ in the general case.
			\end{itemize}
		
		\subsection{Constructive vs.~Classical}
			
			We mentioned that results in this paper will be proved in constructive setting. The main reason\footnote{A different subtler reason is given in Remark~\ref{Remark: synthetic_topology} at the end of the paper.} for this is simply that these proofs work in more general than classical setting, so I feel it is reasonable to present them in such way. Greater generality has, of course, a wider scope of applications; for example, constructive results can be implemented on a computer~\cite{Bauer_A_2000:_the_realizability_approach_to_computable_analysis_and_topology, Lietz_P_2004:_from_constructive_mathematics_to_computable_analysis_via_the_realizability_interpretation}.
			
			In this subsection we recall a few constructive definitions that we require in this paper.
			
			Since we actually need to construct an element to prove its existence, it is not the same to say for a set, that it is non-empty or that it possesses an element; thus we say that a set is \df{inhabited} when there exists an element in it.
			
			The main formal difference between classical and constructive mathematics is that in the latter, the \df{law of excluded middle}, stating that every proposition is either true or false, is not assumed. To put it differently, call a proposition $p$ \df{decidable} when $p \lor \lnot{p}$ holds (that is, we may decide whether $p$ is true or not). The law of excluded middle says that every proposition is decidable while constructively only some might be, for which this needs to be specifically proven. For example, it turns out that relations $=$, $\leq$, $<$ on $\NN$, $\ZZ$, $\QQ$ are decidable (that is, for every pair of elements it is decidable, whether they are in relation), but they are in general not decidable on $\RR$ (a fact which programmers who implement exact real arithmetic are well familiar with --- there exists no algorithm which returns a correct answer for every two real numbers whether they match).
			
			The negation is, as usual, given as $\lnot{p} = (p \Rightarrow \bot)$ where $\bot$ denotes falsehood. Thus $\lnot{p}$ is proven by assuming $p$ and deriving a contradiction. That said, a \df{proof by contradiction} of $p$, where we assume $\lnot{p}$ and derive a contradiction, is constructively not generally valid; indeed with it we merely prove $\lnot\lnot{p}$. The propositions for which it is valid, that is, those for which $\lnot\lnot{p} \implies p$ holds, are called \df{$\lnot\lnot$-stable} (or simply \df{stable}). Every decidable proposition is stable (but not vice versa in general).
		
		We call a set $X$ \df{finite} when there exists a surjective map $\NN_{< n} \to X$ for some $n \in \NN$, \ie we can enumerate the elements of $X$ with the first few natural numbers. Note that the empty set $\emptyset$ is finite by this definition since we can take $n = 0$. In fact, any finite set is either empty or inhabited; consider any surjection $\NN_{< n} \to X$, and decide whether $n$ equals or is greater than $0$.
         
      If we fix a surjection $a\colon \NN_{< n} \to X$, we can write a finite set as $X = \{a_0, a_1, \ldots, a_{n-1}\}$. However, in this list some elements can potentially repeat since we only require $a$ to be a surjection, not a bijection. Therefore, contrary to the classical intuition, for a finite set $X$ there need not exist $n \in \NN$ such that $X$ would have exactly $n$ elements (in the sense that there is a bijection between $X$ and $\NN_{< n}$). In fact, this happens precisely when $X$ has decidable equality (since in that case we can remove the repetitions of elements in the list).\footnote{Some authors reserve the word `finite' only for sets in bijection with $\NN_{< n}$ while what we call finite they term \df{finitely enumerated}. Our definition of finiteness is equivalent to Kuratowski finiteness.}
      
      Next, we call a set $X$ \df{infinite} when there exists an injective map $\NN \to X$. Clearly if a set is infinite, it is also inhabited, and if it contains an infinite subset, it is itself infinite.
      
      The dual notion, the existence of a surjective map $\NN \to X$, is something like countability, but we want the empty set to be considered countable as well, so we allow enumerations containing ``placeholder elements'' which don't actually represent elements of $X$. Thus we define that $X$ is \df{countable} when there exists a map $s\colon \NN \to X + \one$ such that $X$ is contained in the image of $s$. Of course, by this definition every finite set $\{a_0, a_1, \ldots, a_{n-1}\}$ is countable, witnessed by the map
      $$i \mapsto \begin{cases} a_i & \text{if } i < n,\\ \unit & \text{otherwise}. \end{cases}$$
      However, if $X$ is inhabited, then the placeholder element can be replaced by a particular element from $X$, so we see that there is a surjective map $\NN \to X$ if and only if $X$ is both countable and inhabited. That does \emph{not} mean that we can treat $\emptyset$ as a special case; inhabitedness is not a decidable property for general countable sets.
      
      The \df{axiom of choice} (more precisely, the axiom schema) states that every total relation contains a graph of some map:
      $$\xall{x}{X}\xsome{y}{Y}{R(x, y)} \implies \xsome{f}{Y^X}\xall{x}{X}{R(x, f(x))}.$$
      Classically the axiom of choice is usually assumed, while constructively only some of its weaker versions might be. When $X$ is in bijection with $\NN_{< n}$ for some $n \in \NN$, this is called \df{finite choice}, and can be proven by induction, so it is always accepted. Many constructivists also accept the case when $X \ism \NN$, called \df{countable choice}, but we will not do so in this paper. (See however Remark~\ref{Remark: unique_choice} at the end.)

	\section{Disgroups and Disrings}\label{Section: dis_structures}
		
		In this section we define some structures involving an operation $\dis$ which will play a role in the construction and algebra of the Urysohn space.
		
		\begin{definition}\label{Definition: disgroup}
			Let $(X, +, 0)$ be a commutative monoid (= semigroup with the neutral element $0$) and $\dis\colon X \times X \to X$ a binary operation. Declare the relation $\leq$ on $X$ by
			$$a \leq b \sepdfeq a + (a \dis b) = b$$
			for $a, b \in X$. Suppose that the following holds for all $a, b, x \in X$.
			\begin{itemize}
				\item
					$a \dis b = b \dis a$ \quad (commutativity, or symmetry)
				\item
					$a \dis 0 = a$ \quad (unit)
				\item
					$a \dis b = 0$ if and only if $a = b$
				\item
					$(a + x) \dis (b + x) = a \dis b$ \quad (additivity)
				\item
					$(a + a) \dis (b + b) = (a \dis b) + (a \dis b)$
				\item
					if $a + a \leq b + b$, then $a \leq b$
				\item
					$a \dis b \leq a \dis x + b \dis x$ \quad (triangle inequality)
			\end{itemize}
			Then we call $(X, +, 0, \dis)$ a \df{disgroup}. Regarding the order of operations, we declare that $\dis$ is evaluated before $+$.\footnote{I still occasionaly use unnecessary brackets if I feel that they make the calculation clearer.}
		\end{definition}
		
		Intuitively, the operation $\dis$ can be viewed as an ``internal distance'' on $X$; for this reason I suggest the name `disgroup', and $\dis$ to be read as `dis' (the way $+$ is read `plus', and $\cdot$ `times').
		
		Our leading example of a disgroup is the set of non-negative real numbers (the possible distances in metric spaces) $\nnr$ where the addition is the usual one, and $\dis$ is defined for $x, y \in \nnr$ as $x \dis y \dfeq |x - y|$. This example and the fact that we view a ``distance'' to be something ``non-negative'', as well as the axiom $a \dis 0 = a$, suggest that disgroups lack ``negative'' elements. Indeed:
		\begin{proposition}\label{Proposition: elements_nonnegative_in_disgroup}
			In a disgroup $(X, +, 0, \dis)$ all elements are ``non-negative'', that is, the statement $\xall{x}{X}{0 \leq x}$ holds.
		\end{proposition}
		\begin{proof}
			Simple: $0 + (0 \dis x) = 0 + x = x$.
		\end{proof}
      
      Consequently we expect that a disgroup is seldom a (commutative) group. The exception is when its elements are opposite to themselves.
		
		\begin{proposition}\label{Proposition: when_disgroups_are_groups}
			The following statements are equivalent for a disgroup $(X, +, 0, \dis)$.
			\begin{enumerate}
				\item
					The operations $+$ and $\dis$ match.
				\item
					The operation $\dis$ is associative.
				\item
					The relation $\leq$ is full: $\xall{a, b}{X}{a \leq b}$.
				\item
					We have $x + x = 0$ for all $x \in X$.
				\item
					All elements of $X$ are invertible with respect to addition (making $X$ a group).
			\end{enumerate}
		\end{proposition}
		\begin{proof}
			\begin{itemize}
				\item\proven{$(1 \impl 2)$}
					Of course.
				\item\proven{$(2 \impl 1)$}
					For $a, b \in X$ calculate
					$$(a + b) \dis (a \dis b) = \big((a + b) \dis a\big) \dis b = (b \dis 0) \dis b = b \dis b = 0.$$
					Hence $a + b = a \dis b$.
				\item\proven{$(1 \land 2 \impl 3)$}
					We have $a + (a \dis b) = a \dis (a \dis b) = (a \dis a) \dis b = 0 \dis b = b$, so $a \leq b$.
				\item\proven{$(3 \impl 4)$}
					The assuption tells us in particular that $x \leq 0$ which means $0 = x + (x \dis 0) = x + x$.
				\item\proven{$(4 \impl 5)$}
					By the assumption every element $x \in X$ is opposite to itself, \ie $-x = x$.
				\item\proven{$(5 \impl 1)$}
					For $a, b \in X$ we have
					$$a \dis b = (a - b) \dis 0 = a - b,$$
					so $-$ and $\dis$ match. Furthermore,
					$$-a = (-a) \dis 0 = 0 \dis a = a,$$
					so every element of $X$ is opposite to itself which implies that $+$ and $-$ match. We conclude that $+$ and $\dis$ match as well.
			\end{itemize}
		\end{proof}
		
		In this proposition we assumed that $X$ already is a disgroup, but the disgroup conditions actually follow from the above properties. In particular, groups of order $2$ are also examples of disgroups, $\dis$ being simply the addition.
		
		\begin{definition}
			If in a disgroup $\dis$ is associative, we call it an \df{associative disgroup}.
		\end{definition}
		
		\begin{proposition}\label{Proposition: order_two_group_is_a_disgroup}
			Any group $(G, +, 0, -)$ with the property $\xall{x}{G}{x + x = 0}$ is an associative disgroup (where $a \dis b$ is given as $a + b$). (The converse holds by Proposition~\ref{Proposition: when_disgroups_are_groups}.)
		\end{proposition}
		\begin{proof}
			The condition $x + x = 0$ is equivalent to $x = -x$. Recall that any such group $G$ is commutative since $a + b = (-a) + (-b) = -(b + a) = b + a$. Taking $\dis$ to equal $+$, we verify that the relation $\leq$ is full the same way as in the previous proposition. From here the disgroup conditions easily follow.
		\end{proof}
		
		We now examine the properties of general disgroups, particularly of the relation $\leq$.
		
		\begin{proposition}\label{Proposition: cancellation_property_of_disgroups}
			In a disgroup $(X, +, 0, \dis)$ the \df{cancellation property} holds for $=$ and $\leq$:
			$$a + x = b + x \iff a = b \qquad \text{and} \qquad a + x \leq b + x \iff a \leq b$$
			for all $a, b, x \in X$.
		\end{proposition}
		\begin{proof}
			Calculate
			$$a + x = b + x \iff (a + x) \dis (b + x) = 0 \iff a \dis b = 0 \iff a = b$$
			whence we see that $a + x + (a + x) \dis (b + x) = b + x$ and $a + a \dis b = b$ are equivalent as well.
		\end{proof}
		
		\begin{proposition}
			In a disgroup $(X, +, 0, \dis)$ we have
			$$a \leq b \land c \leq d \implies a + c \leq b + d$$
			for all $a, b, c, d \in X$.
		\end{proposition}
		\begin{proof}
			Suppose $a \leq b$ and $c \leq d$, \ie $a + a \dis b = b$, $c + c \dis d = d$. Then
			$$a + c + (a + c) \dis (b + d) = a + c + (a + c) \dis (a + a \dis b + c + c \dis d) =$$
			$$= a + c + a \dis b + c \dis d = b + d.$$
		\end{proof}
		
		\begin{lemma}\label{Lemma: dis_is_difference_between_comparable_elements}
			Let $(X, +, 0, \dis)$ be a disgroup. The following statements are equivalent for $a, b \in X$.
			\begin{enumerate}
				\item
					$a \leq b$
				\item
					$\xall{x}{X}{a \dis (b + x) = (a \dis b) + x}$
				\item
					$\all{x}{X}{x \leq a \dis b \iff x + a \leq b}$
				\item
					$\xsome{x}{X}{a + x = b}$
			\end{enumerate}
		\end{lemma}
		\begin{proof}
			\begin{itemize}
				\item\proven{$(1 \impl 2)$}
					Suppose $a \leq b$, \ie $a + a \dis b = b$, and take $x \in X$. Then
					$$a \dis (b + x) = (a + 0) \dis (a + a \dis b + x) = 0 \dis (a \dis b + x) = a \dis b + x.$$
				\item\proven{$(2 \impl 1)$}
					Take $x = a$:
					$$a + (a \dis b) = (a \dis b) + a = a \dis (b + a) = (0 + a) \dis (b + a) = 0 \dis b = b.$$
				\item\proven{$(1 \impl 3)$}
					If $a \leq b$, then
					$$x + a \leq b \iff x + a \leq a + a \dis b \iff x \leq a \dis b.$$
				\item\proven{$(3 \impl 1)$}
					Take $x = 0$, and recall that all elements, including $a \dis b$, are non-negative.
				\item\proven{$(1 \impl 4)$}
					Take $x = a \dis b$.
				\item\proven{$(4 \impl 1)$}
					Since $x \geq 0$, we have $b = a + x \geq a$.
			\end{itemize}
		\end{proof}
		
		\begin{proposition}
			The relation $\leq$ is a preorder (reflexive and transitive).
		\end{proposition}
		\begin{proof}
			Reflexivity is easy: $a + (a \dis a) = a + 0 = a$.
			
			For transitivity suppose $a \leq b$ and $b \leq c$, \ie $a + a \dis b = b$, $b + b \dis c = c$. Then
			$$a + (a \dis c) = a + (a \dis (b + b \dis c)) = a + ((a \dis b) + (b \dis c)) = b + (b \dis c) = c$$
			(the second equality holds by the previous lemma).
		\end{proof}
		
		We already know (use Proposition~\ref{Proposition: when_disgroups_are_groups} for a group of order $2$) that $\leq$ need not be a partial order (\ie also antisymmetric) in general. However we will mostly be interested in the case when it is. Here is the characterization.
		\begin{proposition}\label{Proposition: when_disgroup_is_partially_ordered}
			The following is equivalent for a disgroup $(X, +, 0, \dis)$.
			\begin{enumerate}
				\item
					$\leq$ is antisymmetric (thus a partial order).
				\item
					$\all{x}{X}{x + x = 0 \implies x = 0}$
				\item
					The map $X \to X$, $x \mapsto x + x$, is injective.
			\end{enumerate}
		\end{proposition}
		\begin{proof}
			\begin{enumerate}
				\item\proven{$(1 \impl 2)$}
					Suppose $\leq$ is antisymmetric and $x + x = 0$ for $x \in X$. Then $x + (x \dis 0) = x + x = 0$, so $x \leq 0$. We also have $0 \leq x$ by Proposition~\ref{Proposition: elements_nonnegative_in_disgroup}. Thus $x = 0$ by antisymmetry.
				\item\proven{$(2 \impl 1)$}
					Take $a, b \in X$, and suppose $a \leq b$ and $b \leq a$. Summing $a + (a \dis b) = b$ and $b + (b \dis a) = a$, we obtain $a + b + (a \dis b) + (a \dis b) = a + b$. Use the cancellation property to obtain $(a \dis b) + (a \dis b) = 0$ whence $a \dis b = 0$ by assumption. Thus $a = b$.
				\item\proven{$(2 \impl 3)$}
					$$x + x = y + y \implies (x + x) \dis (y + y) = 0 \implies$$
					$$\implies (x \dis y) + (x \dis y) = 0 \implies x \dis y = 0 \implies x = y$$
				\item\proven{$(3 \impl 2)$}
					Of course since $0 + 0 = 0$.
			\end{enumerate}
		\end{proof}
		Note that this proposition and Proposition~\ref{Proposition: when_disgroups_are_groups} imply that the only partially ordered disgroup which is also a group is the trivial group.
		
		\begin{lemma}\label{Lemma: special_cases_of_triangle_inequality}
			The following special cases of the triangle inequality hold in a disgroup:
			\begin{enumerate}
				\item\label{Lemma: special_cases_of_triangle_inequality: bound}
					$a \leq a \dis b + b$,
				\item\label{Lemma: special_cases_of_triangle_inequality: comparison_between_dis_and_plus}
					$a \dis b \leq a + b$.
			\end{enumerate}
		\end{lemma}
		\begin{proof}
			These are special cases of the triangle inequality when one of the elements is $0$.
		\end{proof}
		
		\begin{lemma}\label{Lemma: dis_upper_bounds}
			In a disgroup $(X, +, 0, \dis)$ we have
			$$y + a \dis b \leq x \hspace{1em} \iff \hspace{1em} y + a \leq x + b \hspace{1ex} \land \hspace{1ex} y + b \leq x + a$$
			for all $a, b, x, y \in X$ (in particular $a \dis b \leq x \iff a \leq x + b \land b \leq x + a$).
		\end{lemma}
		\begin{proof}
			Suppose $y + a \leq x + b$ and $y + b \leq x + a$ which imply $y + a + a \leq x + a + b$ and $y + b + b \leq x + a + b$. Calculate
			$$y + y + a + a + b + b + a \dis b + a \dis b =$$
			$$= y + y + a + a + b + b + (a + a) \dis (b + b) =$$
			$$= y + a + a + y + b + b + (y + a + a) \dis (y + b + b) \leq$$
			$$\leq (y + a + a) + (y + b + b) + (y + a + a) \dis (x + a + b) + (y + b + b) \dis (x + a + b) =$$
			$$= x + a + b + x + a + b$$
			whence $y + a \dis b + y + a \dis b \leq x + x$ by the cancellation property. Hence $y + a \dis b \leq x$.
			
			Conversely, suppose $y + a \dis b \leq x$ which means $(y + a \dis b) + (y + a \dis b) \dis x = x$. Use Lemma~\ref{Lemma: special_cases_of_triangle_inequality}(\ref{Lemma: special_cases_of_triangle_inequality: bound}) and the non-negativity of elements to calculate
			$$x + b = b + (y + a \dis b) + (y + a \dis b) \dis x \geq y + a + (y + a \dis b) \dis x \geq y + a;$$
			similarly $x + a \geq y + b$.
		\end{proof}
		
		\begin{lemma}
			In a disgroup the following variants of triangle inequality also hold:
			\begin{enumerate}
				\item
					$(a \dis x) \dis (b \dis x) \leq a \dis b$,
				\item
					$a \dis b \leq c \implies a \dis x \leq b \dis x + c$.
			\end{enumerate}
		\end{lemma}
		\begin{proof}
			The first statement follows from the triangle inequality by Lemma~\ref{Lemma: dis_upper_bounds}. The second one is obvious by the additivity and transitivity of $\leq$.
		\end{proof}
		
		\begin{remark}\label{Remark: nondegeneracy_of_dis}
			In the definition of a disgroup we wrote the condition $a \dis b = 0 \iff a = b$ as an equivalence since that is what we are used to from metric spaces, but the implication $a \dis b = 0 \implies a = b$ actually follows from other axioms. Recall from Lemma~\ref{Lemma: special_cases_of_triangle_inequality}(\ref{Lemma: special_cases_of_triangle_inequality: bound}) (which uses only the triangle inequality and the fact $a \dis 0 = a$) that $a \leq a \dis b + b$ which in the case $a \dis b = 0$ means $a \leq b$. By definition this is $a + a \dis b = b$; using $a \dis b = 0$ again, we infer $a = b$.
		\end{remark}
		
		To get a sense what disgroups are, we classify them as the ``non-negative parts'' of certain groups.
		
		\begin{definition}
			Let $(G, +, 0, -)$ be a commutative group and $|\insarg|\colon G \to G$ an operation (the \df{absolute value}) on $G$. Define the relation $\leq$ by
			$$a \leq b \sepdfeq a + |b - a| = b$$
			for $a, b \in G$; in particular $0 \leq x \iff |x| = x$. Suppose that the following holds for all $x, y \in G$.
			\begin{itemize}
				\item
					$\big| |x| \big| = |x|$ \quad (idempotence)
				\item
					$|{-x}| = |x|$
				\item
					$|x| = 0 \iff x = 0$
				\item
					$x \leq |x|$
				\item
					$|x + x| = |x| + |x|$
				\item
					if $x + x \leq y + y$, then $x \leq y$
				\item
					if $0 \leq x$ and $0 \leq y$, then $0 \leq x + y$
				\item
					$|x + y| \leq |x| + |y|$ \quad (triangle inequality)
			\end{itemize}
			Then we call $(G, +, 0, -, |\insarg|)$ a \df{commutative group with absolute value}.
		\end{definition}
		
		\begin{proposition}
			The following holds for a commutative group with absolute value $(G, +, 0, -, |\insarg|)$ and all $a, b, c, d, x \in G$.
			\begin{enumerate}
				\item
					$0 \leq |x|$.
				\item
					$|a - b| + |b - c| \geq |a - c|$.
				\item
					$a + x \leq b + x \iff a \leq b$.
				\item
					$a \leq b \land c \leq d \implies a + c \leq b + d$.
				\item
					$\leq$ is a preorder on $G$.
				\item
					{\hspace{0pt}}\phantom{$\iff$} $\leq$ is a partial order on $G$\\
					$\iff$ $\all{x}{G}{x + x = 0 \implies x = 0}$\\
					$\iff$ $x \mapsto x + x$ is injective.
				%\item
				%	$\big||a| + |b|\big| = |a| + |b|$, or equivalently, $0 \leq a \land 0 \leq b \implies 0 \leq a + b$.
				\item
					$|a| \leq x \iff a \leq x \land -a \leq x$.
			\end{enumerate}
		\end{proposition}
		\begin{proof}
			\begin{enumerate}
				\item
					This is precisely the idempotence of the absolute value.
				\item
					Standard: $|a - c| = |(a - b) + (b - c)| \leq |a - b| + |b - c|$.
				\item
					Because in a group $a + x + |(b + x) - (a + x)| = b + x$ if and only if $a + |b - a| = b$.
				\item
					Let $a \leq b$, $c \leq d$ which means $0 \leq b - a$, $0 \leq d - c$. Then also $0 \leq b - a + d - c$, so $a + c \leq b + d$.
				\item
					Reflexivity: $x + |x - x| = x + 0 = x$.
					
					Transitivity: take $a, b, c \in G$, and suppose $a \leq b$, $b \leq c$. By the previous two items we may sum the two inequalities, and cancel $b$, thus obtaining $a \leq c$.
				\item
					Suppose $\leq$ is a partial order, in particular antisymmetric, and let $x \in G$, $x + x = 0$. Then $|x| + |0 - |x|| = |x| + |x| = |x + x| = 0$, so $|x| \leq 0$. We also know $|x| \geq 0$, so $|x| = 0$, hence $x = 0$. Conversely, take $a, b \in G$ such that $a \leq b$, $b \leq a$ which means $b - a = |b - a| = |a - b| = a - b$, so $(b - a) + (b - a) = b - a + a - b = 0$, therefore $b - a = 0$ by assumption; conclude $a = b$.
					
					Clearly injectivity of $x \mapsto x + x$ implies $x + x = 0 \implies x = 0$ (since $0 + 0 = 0$). Conversely, suppose $x + x = y + y$. Then $(x - y) + (x - y) = 0$ which by assumption means $x - y = 0$, \ie $x = y$.
				%\item
				%	We have $|a| + |b| \geq |a + b| \geq 0$, so by transitivity $|a| + |b| \geq 0$, \ie $\big||a| + |b|\big| = |a| + |b|$.
				\item
					Because $a \leq |a|$ and $-a \leq |-a| = |a|$, one implication holds by transitivity of $\leq$. For the other assume $a \leq x$ and $-a \leq x$, meaning $a + |x - a| = x = -a + |x + a|$. Then
					$$|a| + |a| = |a + a| \leq |a - x| + |a + x| = x - a + x + a = x + x$$
					whence $|a| \leq x$.
			\end{enumerate}
		\end{proof}
		
		\begin{proposition}\label{Proposition: classification_of_disgroups}
			Disgroups are (up to isomorphism) the ``non-negative'' parts of commutative groups with absolute value. More precisely:
			\begin{enumerate}
				\item
					If $(G, +, 0, -, |\insarg|)$ is a commutative group with absolute value, then $(G_{\geq 0}, +, 0, \dis)$ is a disgroup with $\dis$ for $a, b \in G_{\geq 0}$ given as $a \dis b = |b - a|$.
				\item
					Every disgroup embeds (in a way which preserves all operations) into its group of formal differences $G$ (with image $G_{\geq 0}$) which is a commutative group with absolute value.
				\item
					These two processes are mutually inverse up to isomorphism.
			\end{enumerate}
		\end{proposition}
		\begin{proof}
			The proof is a simple exercise. We mention only that for a disgroup $X$ we construct its group of formal differences as the quotient $X \times X/_\equ$ where $(a, b) \equ (c, d) \iff a + d = b + c$, and then the operations on the equivalence classes are given by $[a, b] + [c, d] = [a + c, b + d]$, $-[a, b] = [b, a]$ (as usual) while the absolute value is $|[a, b]| = [a \dis b, 0]$ which is well defined since if $a + d = b + c$, then $a \dis b = (a + d) \dis (b + d) = (b + c) \dis (b + d) = c \dis d$.
		\end{proof}
		
		Recall that $\leq$ being a partial order in a disgroup is equivalent to the injectivity of $x \mapsto x + x$. We will be particularly interested in the case when this map is not only injective, but also split.
		
		\begin{definition}
			We say that $(X, +, 0, \dis, \frac{\insarg}{2})$ is a \df{halved disgroup} when $(X, +, 0, \dis)$ is a disgroup and the operation (the \df{halving map}) $\frac{\insarg}{2}\colon X \to X$ has the following properties for all $a, b \in X$.
			\begin{itemize}
				\item
					$\frac{a}{2} + \frac{a}{2} = a$
				\item
					$\frac{a + b}{2} = \frac{a}{2} + \frac{b}{2}$
			\end{itemize}
		\end{definition}
		
		Here is the characterization of the existence of the halving map.
		\begin{proposition}\label{Proposition: characterization_of_halving_map}
			Let $(X, +, 0, \dis)$ be a disgroup and $g\colon X \to X$ a map, given by $g(x) \dfeq x + x$. The following is equivalent.
			\begin{enumerate}
				\item
					A halving map on the given disgroup exists.
				\item
					The map $g$ is bijective.
			\end{enumerate}
			When these conditions are satisfied, the halving map is the inverse of $g$. In particular, a disgroup can have at most one halving map.
		\end{proposition}
		\begin{proof}
			A halving map is the inverse of $x \mapsto x + x$ because $x = \frac{x}{2} + \frac{x}{2} = \frac{x + x}{2}$, and is therefore unique. Conversely, suppose $f\colon X \to X$ is an inverse of $g$. Then clearly $f(x) + f(x) = x$, but also
			$$f(a + b) = f(f(a) + f(a) + f(b) + f(b)) = f(g(f(a) + f(b))) = f(a) + f(b).$$
		\end{proof}
		
		\begin{proposition}
			A halving map on a disgroup $(X, +, 0, \dis)$ preserves $\dis$ and $\leq$ as well, \ie
			$$\frac{a \dis b}{2} = \frac{a}{2} \dis \frac{b}{2} \qquad \text{and} \qquad a \leq b \iff \frac{a}{2} \leq \frac{b}{2}$$
			for all $a, b \in X$. Moreover, in the presence of the halving map $\leq$ is a partial order.
		\end{proposition}
		\begin{proof}
			We have
			$$\frac{a}{2} \dis \frac{b}{2} = \frac{\frac{a}{2} \dis \frac{b}{2} + \frac{a}{2} \dis \frac{b}{2}}{2} = \frac{(\frac{a}{2} + \frac{a}{2}) \dis (\frac{b}{2} + \frac{b}{2})}{2} = \frac{a \dis b}{2}$$
			and
			$$\frac{a}{2} \leq \frac{b}{2} \iff \frac{a}{2} + \frac{a}{2} \leq \frac{b}{2} + \frac{b}{2} \iff a \leq b.$$
			Furthermore, since the halving map is the inverse of the map $x \mapsto x + x$, the latter must be bijective (in particular injective), so $\leq$ is a partial order by Proposition~\ref{Proposition: when_disgroup_is_partially_ordered}.
		\end{proof}
		
		Recall that given any partial order $(X, \leq)$ and its subset $A \subseteq X$, an element $u \in X$ is defined to be the \df{supremum} of $A$ (denoted by $\sup A$) when
		$$\all{x}{X}{u \leq x \iff \xall{a}{A}{a \leq x}},$$
		and analogously, $l \in X$ is the \df{infimum} of $A$ ($l = \inf A$) when
		$$\all{x}{X}{x \leq l \iff \xall{a}{A}{x \leq a}}.$$
		Suprema and infima are unique in a partial order, though they do not always exist. We show that halved disgroups have finite suprema and binary infima.
		
		\begin{proposition}
			A halved disgroup $(X, +, 0, \dis, \frac{\insarg}{2})$ is a lattice with a bottom element; that is, it has suprema of finite subsets and infima of inhabited finite subsets. Binary suprema and infima are given as
			$$\sup\{a, b\} = \frac{a + b + a \dis b}{2}, \qquad \inf\{a, b\} = \frac{(a + b) \dis (a \dis b)}{2}.$$
			Furthermore, the following holds for all $a, b, x \in X$.
			\begin{enumerate}
				\item
					$\sup\{a + x, b + x\} = \sup\{a, b\} + x, \quad \inf\{a + x, b + x\} = \inf\{a, b\} + x$ \quad (additivity)
				\item
					$a \dis b \leq \sup\{a, b\} \leq a + b$
				\item
					$a + b = \sup\{a, b\} + \inf\{a, b\}$
				\item
					$a \dis b = \sup\{a, b\} \dis \inf\{a, b\}$
			\end{enumerate}
		\end{proposition}
		\begin{proof}
			Throughout this proof keep in mind Lemma~\ref{Lemma: special_cases_of_triangle_inequality}(\ref{Lemma: special_cases_of_triangle_inequality: comparison_between_dis_and_plus}), namely that $a \dis b \leq a + b$.
			
			As always, the nullary supremum $\sup\emptyset$ is the smallest element in the partial order which in our case exists, namely $0$. To show that binary suprema and infima are as prescribed, take arbitrary $a, b, x \in X$. Then
			$$\frac{a + b + a \dis b}{2} \leq x \iff a + b + a \dis b \leq x + x \iff$$
			$$\iff a + b + a \leq x + x + b \land a + b + b \leq x + x + a \iff$$
			$$\iff a + a \leq x + x \land b + b \leq x + x \iff a \leq x \land b \leq x$$
			where the second equivalence holds by Lemma~\ref{Lemma: dis_upper_bounds}, and
			$$x \leq \frac{(a + b) \dis (a \dis b)}{2} \iff x + x \leq (a + b) \dis (a \dis b) \iff$$
			$$\iff x + x + a \dis b \leq a + b \iff x + x + a \leq a + b + b \land x + x + b \leq a + b + a \iff$$
			$$\iff x + x \leq b + b \land x + x \leq a + a \iff x \leq a \land x \leq b$$
			where the second equivalence holds by Lemma~\ref{Lemma: dis_is_difference_between_comparable_elements}, and the third by Lemma~\ref{Lemma: dis_upper_bounds}.
			
			\begin{enumerate}
				\item
					Calculate
					$$\sup\{a + x, b + x\} = \frac{a + x + b + x + (a + x) \dis (b + x)}{2} = \frac{a + b + a \dis b}{2} + x,$$
					$$\inf\{a + x, b + x\} = \frac{(a + x + b + x) \dis ((a + x) \dis (b + x))}{2} =$$
					$$= \big(x + \frac{a + b}{2}\big) \dis \frac{a \dis b}{2} = x + \frac{(a + b) \dis (a \dis b)}{2}$$
					where the last equality holds by Lemma~\ref{Lemma: dis_is_difference_between_comparable_elements}.
				\item
					Because
					$$a \dis b = \frac{a \dis b}{2} + \frac{a \dis b}{2} \leq \underbrace{\frac{a + b}{2} + \frac{a \dis b}{2}}_{\sup\{a, b\}} \leq \frac{a + b}{2} + \frac{a + b}{2} = a + b.$$
				\item
					$$\sup\{a, b\} + \inf\{a, b\} = \frac{a + b + a \dis b}{2} + \frac{(a + b) \dis (a \dis b)}{2} =$$
					$$= \frac{a + b + a \dis b + (a + b) \dis (a \dis b)}{2} = \frac{a + b + (a + b)}{2} = a + b$$
				\item
					Since $\inf\{a, b\} \leq \sup\{a, b\}$, the statement $a \dis b = \sup\{a, b\} \dis \inf\{a, b\}$ is equivalent to $a \dis b + \inf\{a, b\} = \sup\{a, b\}$ by Lemma~\ref{Lemma: dis_is_difference_between_comparable_elements}.
					$$a \dis b + \inf\{a, b\} = \frac{a \dis b}{2} + \frac{a \dis b}{2} + \frac{(a + b) \dis (a \dis b)}{2} =$$
					$$= \frac{a \dis b + a \dis b + (a + b) \dis (a \dis b)}{2} = \frac{a \dis b + (a + b)}{2} = \sup\{a, b\}$$
			\end{enumerate}
		\end{proof}
		
		Halved disgroups can also be described as the non-negative parts of certain groups.
		
		\begin{proposition}\label{Proposition: absolute_value_and_lattice_groups}
			Let $\frac{\insarg}{2}\colon G \times G \to G$ be a halving map (\ie the inverse of the map $x \mapsto x + x$) on a commutative group $(G, +, 0, -)$. The following is equivalent.\footnote{This statement can be made more precise: one can define the category of halved commutative groups with absolute value, and the category of halved commutative lattice groups, and show that they are isomorphic (with the isomorphism preserving the underlying sets and the partial order).}
			\begin{enumerate}
				\item
					$G$ has an absolute value map.
				\item
					$G$ is a \df{lattice group} in the sense that there is a partial order $\leq$ which makes $(G, \leq)$ a lattice, and the following properties additionally hold for $a, b, c, d \in G$.
					\begin{itemize}
						\item
							$a \leq b \land c \leq d \implies a + c \leq b + d$
						\item
							$\sup\{a + a, b + b\} = \sup\{a, b\} + \sup\{a, b\}$
					\end{itemize}
			\end{enumerate}
		\end{proposition}
		\begin{proof}
			Exercise. Show that given the absolute value, binary suprema and infima are calculated as
			$$\sup\{a, b\} = \frac{a + b + |b - a|}{2}, \qquad \inf\{a, b\} = \frac{a + b - |b - a|}{2}.$$
			Conversely, the absolute value is expressed in terms of a supremum as
			$$|a| = \sup\{a, -a\}.$$
			To get the proof in this direction going, start by showing $\sup\{a + x, b + x\} = \sup\{a, b\} + x$, and consequently, $\sup\{a, b\} + \sup\{c, d\} = \sup\{a + c, a + d, b + c, b + d\}$.
			
			Note also that suprema can be expressed in terms of infima, and vice versa, in at least two ways:
			$$\sup\{a, b\} = a + b - \inf\{a, b\} = -\inf\{-a, -b\},$$
			$$\inf\{a, b\} = a + b - \sup\{a, b\} = -\sup\{-a, -b\}.$$
		\end{proof}
		
		\begin{proposition}\label{Proposition: classification_of_halved_disgroups}
			Halved disgroups are (in the sense of Proposition~\ref{Proposition: classification_of_disgroups}) precisely the ``non-negative'' parts of halved commutative groups with absolute value, or equivalently, of halved commutative lattice groups.
		\end{proposition}
		\begin{proof}
			Much the same as the proof of Proposition~\ref{Proposition: classification_of_disgroups}.
		\end{proof}
		
		\intermission
		
		Adding the multiplicative structure to a commutative (semi)group turns it into a (semi)ring. We can do this for disgroups as well. Recall that
		\begin{itemize}
			\item
				$(X, +, 0, \cdot)$ is a \df{semiring} when $(X, +, 0)$ is a commutative monoid, $(X, \cdot)$ is a semigroup, the multiplication $\cdot$ is distributive over addition, and the condition $0 \cdot x = x \cdot 0 = 0$ holds,
			\item
				$(X, +, 0, \cdot, 1)$ is a \df{unital} semiring when $(X, +, 0, \cdot)$ is a semiring and $(X, \cdot, 1)$ a monoid,
			\item
				a semiring is \df{commutative} when its multiplication is commutative.
		\end{itemize}
		
		\begin{definition}
			A structure $(X, +, 0, \dis, \cdot, 1)$ is a \df{disring} when $(X, +, 0, \dis)$ is a disgroup, $(X, +, 0, \cdot, 1)$ a unital commutative semiring, and for $\dis$ the additional property
			\begin{itemize}
				\item
					$(a \dis b) \cdot x = (a \cdot x) \dis (b \cdot x)$ \quad (distributivity, or homogeneity)
			\end{itemize}
			holds for all $a, b, x \in X$. The order of operations is to first evaluate $\cdot$, then $\dis$, then $+$.
		\end{definition}
		
		\begin{remark}
			Informally, a semiring is a ``ring without subtraction''. In a semiring it is necessary to postulate the condition $0 \cdot x = x \cdot 0 = 0$, unlike in the case of a ring where it is implied by $0 \cdot x = (0 + 0) \cdot x = 0 \cdot x + 0 \cdot x$. For the same reason we would not need to assume this condition in the case of a disring, as $0 \cdot x = (0 \dis 0) \cdot x = 0 \cdot x \dis 0 \cdot x = 0$.
		\end{remark}
		
		The set $\nnr$ is a disring. Proposition~\ref{Proposition: when_disgroups_are_groups} still applies to see when a disring is a ring, meaning that rings of characteristic $2$ are examples of disrings ($\dis$ being the addition).
		
		\begin{proposition}
			Let $(X, +, 0, \dis, \cdot, 1)$ be a disring. Then
			$$a \leq b \implies a \cdot x \leq b \cdot x$$
			for all $a, b, x \in X$.
		\end{proposition}
		\begin{proof}
			Straightforward. Suppose $a \leq b$, \ie $a + a \dis b = b$. Then
			$$a \cdot x + a \cdot x \dis b \cdot x = (a + a \dis b) \cdot x = b \cdot x.$$
		\end{proof}
		
		Naturally, we can equip a disring with a halving map as well, obtaining a \df{halved disring}. Again, our example is $\nnr$.
		
		\begin{proposition}
			Let $(X, +, 0, \dis, \cdot, 1, \frac{\insarg}{2})$ be a halved disring. Then $\frac{a}{2} = \frac{1}{2} \cdot a$ for all $a \in X$.
		\end{proposition}
		\begin{proof}
			Note that $a \mapsto \frac{1}{2} \cdot a$ is a halving map on $X$ since by distributivity
			$$\frac{1}{2} \cdot a + \frac{1}{2} \cdot a = \big(\frac{1}{2} + \frac{1}{2}\big) \cdot a = 1 \cdot a = a$$
			and
			$$\frac{1}{2} \cdot (a + b) = \frac{1}{2} \cdot a + \frac{1}{2} \cdot b.$$
			By Proposition~\ref{Proposition: characterization_of_halving_map} a halving map is unique, so $\frac{a}{2} = \frac{1}{2} \cdot a$.
		\end{proof}
		
		As usual, we denote a product of $n \in \NN$ factors $a \in X$ by $a^n$, \ie $a^0 = 1$ and inductively $a^{n+1} = a^n \cdot a$. We'll often deal with powers of one half in a halved disring, and we denote them by $2^{-n} = \big(\frac{1}{2}\big)^n$.
		
		Following the by now usual line, we interpret (halved) disrings as non-negative parts of certain rings.
		\begin{proposition}\label{Proposition: classification_of_halved_dis}
			(Halved) disrings are (in the sense of Proposition~\ref{Proposition: classification_of_disgroups}) the ``non-negative'' parts of (halved) commutative unital rings with absolute value.
		\end{proposition}
		\begin{proof}
			As before.
		\end{proof}
		
		A halving map implies many properties, including some in the definition of a disgroup. As such, we can give halved disgroups/disrings with fewer properties.
		\begin{proposition}
			Let $(X, +, 0)$ be a commutative monoid, $\frac{\insarg}{2}\colon X \to X$ a halving map on it, $\dis\colon X \times X \to X$ a binary operation, and $\leq$ defined as usual: $a \leq b \iff a + a \dis b = b$ for all $a, b \in X$. Suppose that the following holds for all $a, b, x \in X$.
			\begin{itemize}
				\item
					$a \dis b = b \dis a$
				\item
					$a \dis 0 = a$
				\item
					$a \dis a = 0$
				\item
					$(a + x) \dis (b + x) = a \dis b$
				\item
					$a \dis b \leq a \dis x + b \dis x$
			\end{itemize}
			Then $(X, +, 0, \dis, \frac{\insarg}{2})$ is a halved disgroup. If moreover we have $1 \in X$ and a a binary operation $\cdot\colon X \times X \to X$ such that $(X, \cdot, 1)$ is a commutative monoid and $\cdot$ distributes over $+$ and $\dis$, then $(X, +, 0, \dis, \cdot, 1, \frac{\insarg}{2})$ is a halved disring.
		\end{proposition}
		\begin{proof}
			Recall from Remark~\ref{Remark: nondegeneracy_of_dis} that $a \dis b = 0 \implies a = b$ is implied by other conditions, so together with $a \dis a = 0$ we obtain $a \dis b = 0 \iff a = b$. The other conditions follow easily from the properties of the halving map.
		\end{proof}
		
		\begin{remark}\label{Remark: dis_algebraic_theories}
			Notice that all the conditions in the previous proposition are given as \emph{equations}, including the ones not explicitly written (for a monoid \etc[]) as well as the triangle inequality (since $\leq$ is given by an equation). Thus halved disgroups and halved disrings form \df{finitary algebraic theories}. Several consequences immediately follow, such as that products of halved disgroups/disrings are again halved disgroups/disrings.
		\end{remark}
		
		\begin{remark}\label{Remark: examples_of_disgroups}
			So far we've only mentioned $\nnr$ (and groups/rings of order/ characteristic $2$) as an example of (halved) disgroups/disrings. There are of course plenty others, including:
			\begin{itemize}
				\item
					subdisrings of $\nnr$ (\eg natural numbers, non-negative dyadic rationals, non-negative rationals and non-negative real algebraic numbers, the last three being halved) and their arbitrary products (by Remark~\ref{Remark: dis_algebraic_theories});
				\item
					symmetric (or Hermitian in the complex case) matrices form a halved commutative group with absolute value, and therefore their non-negative part --- the positive-semidefinite matrices --- forms a halved disgroup (by Proposition~\ref{Proposition: classification_of_halved_disgroups});
				\item
					Lebesgue integrable maps (more precisely, their equivalence classes, \ie spaces $\mathscr{L}^1$) on closed bounded intervals, continuous maps \etc form halved unital commutative rings, so their non-negative parts (maps which take only non-negative values) form a halved disring (by Proposition~\ref{Proposition: classification_of_halved_dis}).
			\end{itemize}
			See also Proposition~\ref{Proposition: Boolean_lattices_disrings}.
		\end{remark}
		
		\intermission
		
		The purpose of disgroups is to serve us as possible distances of metric spaces. Hereafter let $\od$ denote an arbitrary disgroup (as well as, with a slight abuse of notation, its underlying set). As we prove new results, we will progressively impose further requirements on $\od$:
		\begin{itemize}
			\item
				that its relation $\leq$ is a partial order hereafter,
			\item
				that it has finite suprema (\eg it is halved) in Section~\ref{Section: countable_Urysohn} and onwards,
			\item
				that it is a halved subdisgroup of $\nnr$ containing $1$ in Section~\ref{Section: reals_and_completions} and onwards,
			\item
				that it is a halved disring in Section~\ref{Section: algebra_of_Urysohn}.
		\end{itemize}
		
		We will require not just metric spaces, but also their generalizations. Pseudometric is a common generalization of a metric; what we term `protometric', less so.
		\begin{definition}\label{Definition: od_metric_spaces}
			Let $d\colon X \times X \to \od$ be a map. We call $\mtr{X} = (X, d)$
			\begin{itemize}
				\item
					a \df{$\od$-protometric space} when
					$$d(a, b) = d(b, a), \qquad \text{(symmetry)}$$
					$$d(a, b) + d(b, c) \geq d(a, c) \qquad \text{(triangle inequality)}$$
					hold for all $a, b, c \in X$,
				\item
					a \df{$\od$-pseudometric space} when additionally
					$$d(a, a) = 0$$
					for all $a \in X$,
				\item
					a \df{$\od$-metric space} when furthermore
					$$d(a, b) = 0 \implies a = b$$
					holds for all $a, b \in X$.
			\end{itemize}
			The latter two conditions can be summarized as $d(a, b) = 0 \iff a = b$.
			
			The map $d$ is called the (\df{proto}-, \df{pseudo}-)\df{metric}, or more informally, the \df{distance} on $X$.
		\end{definition}
		
		Of course, the usual definition of (pseudo)metric spaces matches $\nnr$-(pseudo)metric spaces in the sense of Definition~\ref{Definition: od_metric_spaces}.
		
		\medskip
		
		Recall we mentioned that $\dis$ in a disgroup can be viewed as an ``internal distance''. We can now make this precise.
		\begin{proposition}
			Every disgroup $\od$ is a $\od$-metric space with $\dis$ as the distance.
		\end{proposition}
		\begin{proof}
			Follows immediately from definitions.
		\end{proof}
		
		The morphisms of metric spaces we will mostly work with are the following.
		\begin{definition}\label{Definition: maps}
			The map $f\colon X \to Y$ map between $\od$-(proto-, pseudo-)metric spaces $\mtr{X} = (X, d_\mtr{X})$ and $\mtr{Y} = (Y, d_\mtr{Y})$ is
			\begin{itemize}
				\item
					\df{non-expansive} when $d_\mtr{Y}\big(f(x), f(y)\big) \leq d_\mtr{X}(x, y)$ for all $x, y \in X$,
				\item
					an \df{isometry} when $d_\mtr{Y}\big(f(x), f(y)\big) = d_\mtr{X}(x, y)$ for all $x, y \in X$,
				\item
					an \df{isometric embedding} when it is an injective isometry,
				\item
					an \df{isometric isomorphism} when it is a bijective isometry (and hence its inverse is an isometry as well).
			\end{itemize}
		\end{definition}
		
		\begin{remark}
			Non-expansive maps are often taken as morphisms of a category of metric spaces, and are called \df{metric maps} in that context. The isomorphisms of this category are isometric isomorphisms.
		\end{remark}
		
		The following definition is also useful when talking about the Urysohn space.
		\begin{definition}
			A \df{finite partial isometry} $X \parto Y$ is an isometry which maps from a finite subset of $X$ to $Y$.
		\end{definition}
		
		Next, we recall what (binary) products of (proto-, pseudo-)metric spaces are. For spaces $\mtr{X} = (X, d_\mtr{X})$, $\mtr{Y} = (Y, d_\mtr{Y})$ there are many reasonable choices for the metric on the product $X \times Y$, all yielding the same topology. The two product metrics we will use in this paper are the $\infty$-metric (also called the $\sup$-metric), given by
		$$d_\infty\big((x, y), (x', y')\big) \dfeq \sup\{d_\mtr{X}(x, x'), d_\mtr{Y}(y, y')\},$$
		and the $1$-metric (also know as the ``taxicab'' metric)
		$$d_1\big((x, y), (x', y')\big) \dfeq d_\mtr{X}(x, x') + d_\mtr{Y}(y, y').$$
		These definitions extend to general finite products. Note that projections are non-expansive maps in either case.
		
		\medskip
		
		We defined metric spaces in stages, going through protometric and pseudometric spaces first. These will also be the stages of the construction of the Urysohn space. We observe here that there are natural passages between the notions.
		
		Given a $\od$-protometric space $\mtr{X} = (X, d)$, define its \df{kernel} by
		$$\ker(\mtr{X}) \dfeq \st{x \in X}{d(x, x) = 0}.$$
		It is immediate that the kernel of a $\od$-protometric space is a $\od$-pseudometric space. The inclusion $\ker(\mtr{X}) \hookrightarrow X$ is of course an isometric embedding.\footnote{In categorical terms, the subcategory of $\od$-pseudometric spaces is coreflective in the category of $\od$-protometric spaces, the kernel is the coreflector, and its inclusion is the counit of the adjunction.}
		
		In any $\od$-pseudometric space $\mtr{X} = (X, d)$ we can define the relation $\equ$ for $a, b \in X$ by
		$$a \equ b \dfeq \big(d(a, b) = 0\big).$$
		It is easy to see that this is an equivalence relation. The quotient it induces is called the \df{Kolmogorov quotient}\footnote{In classical general topology the Kolmogorov quotient of a topological space is constructed by identifying points which have the same neighbourhoods, thus obtaining a $T_0$ space. In pseudometric spaces the points with the same neighbourhoods are precisely those at zero distance.}, and it is a $\od$-metric space for the metric (which we'll by a slight abuse of notation denote by the same symbol) $d([a], [b]) \dfeq d(a, b)$ (properties of a pseudometric imply that this is well-defined). From this it is clear that the Kolmogorov quotient map is a surjective isometry.\footnote{The categorical interpretation is that the subcategory of $\od$-metric spaces is reflective in the category of $\od$-pseudometric spaces, the Kolmogorov quotient is the reflector, and the quotient map is the unit of the adjunction.}
		
		\begin{remark}
			Actually, the converse also holds: any surjective isometry from a pseudometric to a metric space is, up to isometric isomorphism, the Kolmogorov quotient map. As for the injectivity of isometries, while they need not be injective in general, recall that they perforce are if their domain is a metric space.
		\end{remark}

	\section{Countable Urysohn Space}\label{Section: countable_Urysohn}
	
		A typical construction of the Urysohn space involves constructing a rational version, and then the actual Urysohn space is its completion~\cite{Urysohn_PS_1927:_sur_un_espace_metrique_universel, Holmes_MR_1992:_the_universal_separable_metric_space_of_urysohn_and_isometric_embeddings_thereof_in_banach_spaces}. In this section we generalize this approach, refining the construction we presented in~\cite{Lesnik_D_2009:_constructive_urysohn_universal_metric_space}. The idea is to construct a metric space which has the extension property much like the Urysohn space, but has distances limited to a disgroup $\od$, and is not yet complete. We do so in three steps: we construct a $\od$-protometric candidate, refine it to a $\od$-pseudometric space of which we take the Kolmogorov quotient to obtain the desired $\od$-metric space.
		
		Define two sequences of sets inductively as follows: let $\proto[-1] \dfeq \emptyset$, and then
		$$\protoaux[n] \dfeq \finseq{(\proto[n-1] \times \od)} \qquad \text{and} \qquad \proto[n] \dfeq \proto[n-1] + \protoaux[n] \qquad \text{for $n \in \NN$}.$$
		In words, we start with the empty set, and then repeatedly make finite tuples of elements that we already have, together with some elements from $\od$. Note that $\proto[0]$ is a singleton, as we can make only the empty tuple from zero elements. Consequently $\protoaux[1] \ism \finseq{\od}$ and $\proto[1] \ism \finseq{\od} + \one$. The later sets get more complicated quickly.
		
		We adopt the following notation. Taking the elements $a_0, \ldots, a_{l-1} \in \proto[n-1]$ and $\alpha_0, \ldots, \alpha_{l-1} \in \od$, we denote the tuple, constructed from these elements, by $\gentuple \in \protoaux[n] \subseteq \proto[n]$. We call $n$ the \df{age}, and $l$ the \df{length} of the tuple while $a_i$s are its \df{predecessors}. Note that ``the same'' tuple appears at all the later ages as well, but (in view of using disjoint unions) we consider these tuples to be different elements: $\gentuple \neq \gentuple[m]$ if $n \neq m$. In particular, the empty tuple appears at all the ages (from $0$ onwards).
		
		The idea behind this construction is that a tuple $\gentuple$ should represent a point which is at distances $\alpha_i$ from $a_i$s, and so by inductively adding these tuples, the space we obtain in the end should satisfy the extension property (roughly speaking, as not all choices of distances are valid; we deal with this below). We define $\proto$ to be the set of all such tuples, \ie $\proto \dfeq \coprod_{n \in \NN} \protoaux[n]$; equivalently, $\proto$ is the colimit (or the direct limit, or the union if you will) of all $\proto[n]$s.
		
		For now, these definitions should be considered on the formal level; we cannot just make a sequence of sets like this in our setting. We prove that these definitions are valid by constructing an explicit model of $\proto$, $\protoaux[n]$s and $\proto[n]$s.
		
		The idea is that a tuple should be encoded by its age, the encodings of its predecessors, and prescribed distances (we'll be able to infer the length) which can be done with a combination of natural numbers and the elements of $\od$. Let $(\overline{x})$ denote the encoding of a tuple $x$; then inductively, for $a = \gentuple$,
		$$(\overline{a}) \dfeq \big(n, \alpha_0, \overline{a_0}, n, \alpha_1, \overline{a_1}, n, \ldots, n, \alpha_{l-1}, \overline{a_{l-1}}, n\big).$$
		Thus the age of $a$ is the first term of the sequence, and the length is the number of times the age $n$ appears, minus one. The encodings of predecessors are unambigously separated by $n$s, as predecessors necessarily have lesser ages. In conclusion, $\proto$ is the subset of $\finseq{(\NN + \od)}$ containing those sequences which start with a natural number (say $n$), end by $n$ as well, all natural numbers appearing in the sequence are $\leq n$, the first term after every $n$ (except the last one) is in $\od$, and recursively, for every two consecutive $n$s the sequence between them, minus the first term, is a valid encoding of a tuple of an age $< n$. Finally, define $\protoaux[n]$ to be the set of those elements from thusly represented $\proto$ which have the first term $n$, and $\proto[n]$ to contain the elements with the first term $\leq n$.
		
		This representation notwithstanding, we still prefer to write the elements of $\proto$ as tuples of the form $\gentuple$, as these are easier to deal with than the encodings.
		
		Our next task is to equip $\proto$ with the distance. As mentioned, the intuition is that the distances of $\gentuple$ from $a_i$s should be $\alpha_i$s, but defining the distance between different tuples is a non-trivial task, as we need to satisfy all triangle inequalities (or in our case, the condition for a $\od$-protometric) at once. Generally there are many solutions, as a triangle inequality bounds a distance to an interval, not a single point. It turns out, however, that we want the ``minimal'' solution which is to say that the distance between tuples with similar terms is small (see Proposition~\ref{Proposition: similar_distances_and_points_yield_similar_extensions_in_Urysohn_space} below) which we rely upon in the next section to show that the completion is the actual Urysohn space.
		
		Let $\age{a}$ and $\lnth{a}$ denote the age and the length of a tuple $a$, respectively. We define the map $d\colon \proto \times \proto \to \od$ for $a = \atuple{a}{\alpha}{i}, b = \atuple{b}{\beta}{j} \in \proto$ inductively on $\age{a} + \age{b}$ as
		$$d(a, b) \dfeq \sup\Big(\st{d(a_i, b) \dis \alpha_i}{i \in \NN_{< \lnth{a}}} \cup \st{d(a, b_j) \dis \beta_j}{j \in \NN_{< \lnth{b}}}\Big).$$
		This inductive definition is convenient, as we do not need the base case since eventually we end up calculating (the distance of the empty tuple to itself which is) the supremum of the empty set (namely $0$).
		
		\begin{proposition}
			$(\proto, d)$ is a $\od$-protometric space.
		\end{proposition}
		\begin{proof}
			Let $a = \atuple{a}{\alpha}{i}$, $b = \atuple{b}{\beta}{j}$, $c = \atuple{c}{\gamma}{k} \in \proto$ denote three general elements of $\proto$.
			\begin{itemize}
				\item\proven{symmetry}
					By induction on $\age{a} + \age{b}$:
					$$d(b, a) =$$
					$$= \sup\Big(\st{d(b_j, a) \dis \beta_j}{j \in \NN_{< \lnth{b}}} \cup \st{d(b, a_i) \dis \alpha_i}{i \in \NN_{< \lnth{a}}}\Big) =$$
					$$= \sup\Big(\st{d(a, b_j) \dis \beta_j}{j \in \NN_{< \lnth{b}}} \cup \st{d(a_i, b) \dis \alpha_i}{i \in \NN_{< \lnth{a}}}\Big) =$$
					$$= d(a, b).$$
				\item\proven{triangle inequality}
					By induction on $\age{a} + \age{b} + \age{c}$; it is sufficient to verify that every value of the set, of which supremum is $d(a, c)$, is at most $d(a, b) + d(b, c)$.
					$$d(a_i, c) \dis \alpha_i \leq d(a_i, b) \dis \alpha_i + d(b, c) \leq d(a, b) + d(b, c)$$
					$$d(a, c_k) \dis \gamma_k \leq d(a, b) + d(b, c_k) \dis \gamma_k \leq d(a, b) + d(b, c)$$
			\end{itemize}
		\end{proof}
		
		Let $\lambda, \mu \in \od$ and $a \dfeq \strat[1]{\big(\et, \lambda, \et, \mu\big)}{}$. Note that $d(a, a) = \lambda \dis \mu$ which is in general not zero, so $(\proto, d)$ is not a $\od$-pseudometric space. This shouldn't be too surprising, as $a$ is supposed to be a point which is at distances both $\lambda$ and $\mu$ from $\et$ which makes sense only when $\lambda = \mu$ (which is when $\lambda \dis \mu = 0$). More generally, distances should respect triangle inequalities, and we see that we need to trim $\proto$ to tuples that do.
		
		Define inductively $\pseudo[-1] \dfeq \emptyset$,
		\mlst{\pseudoaux[n] \dfeq}{\gentuple \in \protoaux[n]}
			{\xall{i}{\NN_{< n}}{a_i \in \pseudo[n-1]} \nl
			\all{i, j}{\NN_{< n}}{d(a_i, a_j) \dis \alpha_i \leq \alpha_j}}
		{,}
		$$\pseudo[n] \dfeq \pseudo[n-1] + \pseudoaux[n] \qquad \text{for $n \in \NN$}.$$
		Finally, let $\pseudo \dfeq \coprod_{n \in \NN} \pseudoaux[n]$, or equivalently, $\pseudo$ is the colimit od $\pseudo[n]$s. To see that we have models of these sets, just consider them as subsets $\pseudo[n] \subseteq \proto[n]$, $\pseudo \subseteq \proto$. We say that a tuple $a \in \proto$ is \df{permissible} when $a \in \pseudo$.
		
		\begin{theorem}\label{Theorem: distances_as_described_in_Urysohn_tuples}
			For every permissible tuple $a = \atuple{a}{\alpha}{i} \in \pseudo$ we have $d(a, a) = 0$, or equivalently, $d(a, a_i) = \alpha_i$ for every $i \in \NN_{< \lnth{a}}$.
		\end{theorem}
		\begin{proof}
			It is easy to see that the statements in the theorem are equivalent; we focus on actually proving them. We do so by induction on $\age{a}$.
			
			In the base case $\age{a} = 0$, \ie $a = \et$, there is nothing to prove. For a general $a$ describe all its predecessors inductively as follows:
			$$a_{j_0, j_1, \ldots, j_r} = (a_{j_0, j_1, \ldots, j_r, j_{r+1}}, \alpha_{j_0, j_1, \ldots, j_r, j_{r+1}})_{j_{r+1} \in \NN_{< \lnth{a_{j_0, j_1, \ldots, j_r}}}}.$$
			
			The heart of the proof is in the following \textbf{Claim}:
			\begin{itemize}
				\item
					{\it Let $l \in \NN$ and $j_i \in \NN_{< \lnth{a_{j_0, j_1, \ldots, j_{i-1}}}}$ for all $i \in \NN_{\leq l}$. Assume that for all $j_{l+1} \in \NN_{< \lnth{a_{j_0, j_1, \ldots, j_l}}}$ we have
					$$d(a, a_{j_0, j_1, \ldots, j_{l+1}}) \leq \alpha_{j_0} +  \alpha_{j_0, j_1} + \alpha_{j_0, j_1, j_2} + \ldots + \alpha_{j_0, j_1, \ldots, j_l} + \alpha_{j_0, j_1, \ldots, j_{l+1}}.$$
					Then $d(a, a_{j_0, j_1, \ldots, j_l}) \leq \alpha_{j_0} + \alpha_{j_0, j_1} + \alpha_{j_0, j_1, j_2} + \ldots + \alpha_{j_0, j_1, \ldots, j_l}$.}
					\bigskip
					\begin{proof}
						We have
						$$d(a, a_{j_0, j_1, \ldots, j_l}) = \sup\Big\{\st[1]{d(a_i, a_{j_0, j_1, \ldots, j_l}) \dis \alpha_i}{i \in \NN_{< \lnth{a}}} \cup$$
						$$\cup \st[1]{d(a, a_{j_0, j_1, \ldots, j_{l+1}}) \dis \alpha_{j_0, j_1, \ldots, j_{l+1}}}{j_{l+1} \in \NN_{< \lnth{a_{j_0, j_1, \ldots, j_l}}}}\Big\}.$$
						For $i \in \NN_{< \lnth{a}}$ recall permissibility and the original induction hypothesis.
						$$d(a_i, a_{j_0, j_1, \ldots, j_l}) \dis \alpha_i \leq d(a_i, a_{j_0}) \dis \alpha_i + d(a_{j_0}, a_{j_0, j_1, \ldots, j_l}) \leq$$
						$$\leq d(a_i, a_{j_0}) \dis \alpha_i + d(a_{j_0}, a_{j_0, j_1}) + d(a_{j_0, j_1}, a_{j_0, j_1, j_2}) + \ldots +$$
						$$+ d(a_{j_0, j_1, \ldots, j_{l-1}}, a_{j_0, j_1, \ldots, j_l}) \leq \alpha_{j_0} + \alpha_{j_0, j_1} + \alpha_{j_0, j_1, j_2} + \ldots + \alpha_{j_0, j_1, \ldots, j_l}$$
						
						As for the second part, take any $j_{l+1} \in \NN_{< \lnth{a_{j_0, j_1, \ldots, j_l}}}$.
						$$\alpha_{j_0, j_1, \ldots, j_{l+1}} = d(a_{j_0, j_1, \ldots, j_l}, a_{j_0, j_1, \ldots, j_{l+1}}) \leq$$
						$$\leq d(a, a_{j_0, j_1, \ldots, j_{l+1}}) + d(a, a_{j_0, j_1, \ldots, j_l}) \leq$$
						$$\leq d(a, a_{j_0, j_1, \ldots, j_{l+1}}) + \alpha_{j_0} + d(a_{j_0}, a_{j_0, j_1, \ldots, j_l}) \leq d(a, a_{j_0, j_1, \ldots, j_{l+1}}) + \alpha_{j_0} +$$
						$$+ d(a_{j_0}, a_{j_0, j_1}) + d(a_{j_0, j_1}, a_{j_0, j_1, j_2}) + \ldots + d(a_{j_0, j_1, \ldots, j_{l-1}}, a_{j_0, j_1, \ldots, j_l}) =$$
						$$= d(a, a_{j_0, j_1, \ldots, j_{l+1}}) + \alpha_{j_0} + \alpha_{j_0, j_1} + \alpha_{j_0, j_1, j_2} + \ldots + \alpha_{j_0, j_1, \ldots, j_l}$$
						The last inequality to prove,
						$$d(a, a_{j_0, j_1, \ldots, j_{l+1}}) \leq \alpha_{j_0, j_1, \ldots, j_{l+1}} + \alpha_{j_0} + \alpha_{j_0, j_1} + \alpha_{j_0, j_1, j_2} + \ldots + \alpha_{j_0, j_1, \ldots, j_l},$$
						holds by assumption.
					\end{proof}
			\end{itemize}
			
			Notice that this {Claim} serves not only as the inductive step, but also as the base of induction since when we reach the empty tuple (which is after at most $\age{a}$ steps), the condition is vacuous. In the end we obtain $d(a, a_{j_0}) \leq \alpha_{j_0}$.
			
			The reverse inequality is easier:
			$$d(a, a_{j_0}) \geq d(a_{j_0}, a_{j_0}) \dis \alpha_{j_0} = \alpha_{j_0}$$
			since $d(a_{j_0}, a_{j_0}) = 0$ by the induction hypothesis.
		\end{proof}
		
		This proves that $\pseudo$ is a $\od$-pseudometric space. In fact, we claim it contains precisely the tuples from $\proto$ such that they, and their predecessors, and the predecessors' predecessors \etc are at distance $0$ to themselves.
		
		\begin{theorem}\label{Theorem: classification_of_permissible_tuples}
			\
			\begin{enumerate}
				\item
					Define the map $r\colon \proto \to \pseudo$ for $a = \atuple{a}{\alpha}{i} \in \proto$ inductively on $\age{a}$ by
					$$r(a) \dfeq {}_{\age{a}}\Big(r(a_i), d\big(a, r(a_i)\big)\Big)_{i \in \NN_{< \lnth{a}}}.$$
					The map $r$ is well defined (we need to verify $r(\protoaux[n]) \subseteq \pseudoaux[n]$) and a retraction of $\proto$ onto $\pseudo$ (\ie $\rstr{r}_{\pseudo} = \id[\pseudo]$).
				\item
					Define inductively $\strat[-1]{\psd'}{\od} \dfeq \emptyset$,
					$$\strat[n]{\psd'}{\od} \dfeq \st{a = \atuple{a}{\alpha}{i} \in \proto[n]}{d(a, a) = 0 \land \xall{i}{\NN_{< \lnth{a}}}{a_i \in \strat[n-1]{\psd'}{\od}}}.$$
					Then $\pseudo[n] = \strat[n]{\psd'}{\od}$ for all $n \in \NN$.
			\end{enumerate}
		\end{theorem}
		\begin{proof}
			\begin{enumerate}
				\item
					\begin{itemize}
						\item\proven{$r(\protoaux[n]) \subseteq \pseudoaux[n]$ for all $n \in \NN$}
							By induction on $n$. Clearly the age of the tuple is preserved by $r$ if the ages of predecessors are. To see that $r$ maps tuples to permissible ones, note that $r$-images of predecessors are permissible by the induction hypothesis, and the condition
							$$d(r(a_i), r(a_j)) \dis d(a, r(a_i)) \leq d(a, r(a_j))$$
							holds by triangle inequality.
						\item\proven{$\rstr{r}_{\pseudo[n]} = {\id[]}_{\pseudo[n]}$}
							Take $a = \atuple{a}{\alpha}{i} \in \pseudo[n]$; we verify $r(a) = a$. By induction on $n$ we have $r(a_i) = a_i$ for all $i \in \NN_{< \lnth{a}}$, and the condition $d(a, a_i) = \alpha_i$ holds by Theorem~\ref{Theorem: distances_as_described_in_Urysohn_tuples}.
					\end{itemize}
				\item
					Induction on $n$, together with Theorem~\ref{Theorem: distances_as_described_in_Urysohn_tuples}, tells us that $\pseudo[n] \subseteq \strat[n]{\psd'}{\od}$ for all $n$. For the converse it is sufficient to verify $r(a) = a$ for all $a = \atuple{a}{\alpha}{i} \in \strat[n]{\psd'}{\od}$. By the induction hypothesis $r(a_i) = a_i$ for all $i$, so we only still need to see $d(a, a_i) = \alpha_i$, but this follows from the assumption $d(a, a) = 0$.
			\end{enumerate}
		\end{proof}
		
		However, $\pseudo$ is not a metric space --- to obtain different tuples at distance $0$, try for example changing the order of terms in the tuple, repeat the terms, or simply consider any $a \in \pseudo[n]$ and ${}_{n+1}(a, 0)$.
		
		We define $\Ury$ to be the Kolmogorov quotient of $\pseudo$. As such, it is a $\od$-metric space.
		
		\begin{proposition}\label{Proposition: similar_distances_and_points_yield_similar_extensions_in_Urysohn_space}
			Let $a = {}_{\age{a}}(a_i, \alpha_i)_{i \in \NN_{< l}}, b = {}_{\age{b}}(b_j, \beta_j)_{j \in \NN_{< l}} \in \pseudo$ be tuples of the same length $l \in \NN$. For any $\epsilon, \epsilon' \in \od$ if $\alpha_i \dis \beta_i \leq \epsilon$ and $d(a_i, b_i) \leq \epsilon'$ for all $i \in \NN_{< l}$, then $d(a, b) \leq \epsilon + \epsilon'$.
		\end{proposition}
		\begin{proof}
		   $$d(a_i, b) \dis \alpha_i \leq d(b_i, b) \dis \alpha_i + d(a_i, b_i) = \beta_i \dis \alpha_i + d(a_i, b_i) \leq \epsilon + \epsilon'$$
		   $$d(a, b_i) \dis \beta_i \leq d(a, a_i) \dis \beta_i + d(a_i, b_i) = \alpha_i \dis \beta_i + d(a_i, b_i) \leq \epsilon + \epsilon'$$
		\end{proof}
		
		\begin{corollary}\label{Corollary: Urysohn_tuples_canonical_on_quotient}
			Let $l \in \NN$, $\omega_0, \ldots, \omega_{l-1} \in \od$ and $a = {}_{\age{a}}(a_i, \omega_i)_{i \in \NN_{< l}}, b = {}_{\age{b}}(b_j, \omega_j)_{j \in \NN_{< l}} \in \pseudo$ such that $d(a_i, b_i) = 0$ for all $i \in \NN_{< n}$. Then $d(a, b) = 0$.
		\end{corollary}
		\begin{proof}
			Take $\epsilon = \epsilon' = 0$ in Proposition~\ref{Proposition: similar_distances_and_points_yield_similar_extensions_in_Urysohn_space}.
		\end{proof}
		
		Denote
		$$\prms[\od] \dfeq \st{(x_i, \chi_i)_{i \in \NN_{< l}} \in \finseq{(\Ury \times \od)}}{\all{i, j}{\NN_{< l}}{d(x_i, x_j) \dis \chi_i \leq \chi_j}},$$
		and let $\ext[\od]\colon \prms[\od] \to \Ury$ be given as
		$$\ext[\od]\Big(\big([a_i], \chi_i\big)_{i \in \NN_{< l}}\Big) \dfeq \big[{}_{\sup\st{\age{a_i}}{i \in \NN_{< l}}+1}(a_i, \chi_i)_{i \in \NN_{< l}}\big].$$
		Note that the map $\ext[\od]$ is well defined by Corollary~\ref{Corollary: Urysohn_tuples_canonical_on_quotient}.
		
		\begin{lemma}\label{Lemma: extend_isometry_by_one_point_into_uncompleted_Urysohn}
			We have $d\big(\ext[\od]((x_i, \chi_i)_{i \in \NN_{< l}}), x_k\big) = \chi_k$ for all $k \in \NN_{< l}$.
		\end{lemma}
		\begin{proof}
			By Theorem~\ref{Theorem: distances_as_described_in_Urysohn_tuples}. 
		\end{proof}
		
		We claim that the existence of such a map ensures an extension property, similar to the one that the Urysohn space has. For this reason we introduce the following definition.
		\begin{definition}\label{Definition: od-Urysohn}
			Let
			\begin{itemize}
				\item
					$(U', d')$ be a $\od$-metric space,
				\item
					$\prms[\od]' \dfeq \st{(x_i, \chi_i)_{i \in \NN_{< l}} \in \finseq{(U' \times \od)}}{\all{i, j}{\NN_{< l}}{d'(x_i, x_j) \dis \chi_i \leq \chi_j}}$, and
				\item
					$\ext[\od]'\colon \prms[\od]' \to U'$ a map with the property
					$$\xall{(x_i, \chi_i)_{i \in \NN_{< l}}}{\prms[\od]'}\xall{k}{\NN_{< l}}{d'\big(\ext[\od]'((x_i, \chi_i)_{i \in \NN_{< l}}), x_k\big) = \chi_k}.$$
			\end{itemize}
			Then $(U', d', \ext[\od]')$ is called a \df{$\od$-Urysohn space}.
		\end{definition}
		
		\begin{theorem}\label{Theorem: countable_Urysohn_space}
			Let $(U', d', \ext[\od]')$ be a $\od$-Urysohn space and let
			\begin{itemize}
				\item
					$(X, d_\mtr{X}, s\colon \NN \to \one + X)$ be a countable $\od$-metric space (with $s$ the enumeration of its elements),
				\item
					$F \subseteq X$ a finite subset with enumeration $F = \{y_0, \ldots, y_{k-1}\}$,
				\item
					$e\colon F \to U'$ an isometry.
			\end{itemize}
			Then there exists a canonical choice of an isometry $f\colon X \to U'$ such that $\rstr{f}_F = e$.
		\end{theorem}
		\begin{proof}
			Notice that $\ext[\od]'$ allows us to extend the isometry for one point, so the idea is to use it inductively, first for $F$, then adding more and more terms of the sequence $s$. Explicitly, if $s_n \in X$, define $f(s_n)$ inductively on $n \in \NN$ as
			$$f(s_n) \dfeq \ext[\od]'\Big(\big(e(y_i), d_\mtr{X}(s_n, y_i)\big)_{i \in \NN_{< k}} \cnct \big(f(s_j), d_\mtr{X}(s_n, s_j)\big)_{j \in \NN_{< n} \cap s^{-1}(X)}\Big).$$
			The map $f$ is well defined --- if $s_n$ equals some $x \in F \cup \big(s(\NN_{< n}) \cap X\big)$, then $d'\big(f(s_n), f(x)\big) = 0$, so $f(s_n) = f(x)$. For the same reason $f$ is an extension of $e$.
		\end{proof}
		
		Thus $\Ury$, and more generally any $\od$-Urysohn space, satisfies the extension property for finite partial isometries from countable $\od$-metric spaces. The converse of course also holds: if we have a canonical choice of extending finite partial isometries, then $\ext[\od]'$ can be defined as its special case. To see this, take $\big(x_i, \chi_i\big)_{i \in \NN_{< l}} \in \prms[\od]'$ and let $F \dfeq \st{x_i}{i \in \NN_{< l}}$ be the metric subspace of $U'$. Declare $X'$ to be $F$, together with another point $\unit$ which is at distance $\chi_i$ to $x_i$ for all $i \in \NN_{< l}$. Let $X$ be the Kolmogorov quotient of $X'$, to ensure that it is metric ($X'$ might not have been, as some $\chi_i$s could potentially be zero). Extend the isometric embedding $F \hookrightarrow U'$ and declare that $\ext[\od]'((x_i, \chi_i)_{i \in \NN_{< l}})$ is the image of $\unit$.
		
		\begin{corollary}\label{Corollary: countable_Urysohn_embedding_property}
			Any countable $\od$-metric space can be isometrically embedded into a $\od$-Urysohn space.
		\end{corollary}
		\begin{proof}
			Extend the finite partial isometry with the empty domain.
		\end{proof}
		
		How many different $\od$-Urysohn spaces are there, though? In the remainder of this section we verify, that when $\od$ is countable, $\Ury$ is the only countable $\od$-Urysohn space up to isometric isomorphism.
		
		\begin{lemma}\label{Lemma: finite_sequences_over_countable_alphabet_are_countable}
			For any set $A$ there exists a mapping which takes a surjection $f\colon\NN \to A$ to a surjection $\NN \to \finseq{A}$.
		\end{lemma}
		\begin{proof}
			Fix a bijection $\NN \ism \finseq{\NN}$ and compose it with $\coprod_{n \in \NN} f^n$.
		\end{proof}
		
		\begin{lemma}\label{Lemma: countable_od_implies_countable_protoUrysohn}
			Suppose $\od$ is countable. Then there exists a sequence of sequences $s\colon \NN \times \NN \to \proto$ such that for every $n \in \NN$ the image of $s_n$ is $\protoaux[n]$.
		\end{lemma}
		\begin{proof}
			Recall that there exists a sequence of bijections $b_n\colon \NN \to \NN^n$, $n \in \NN_{\geq 1}$. Also, let $c\colon \NN \to \od$ be a surjection (it exists because $\od$ is inhabited (\eg $0 \in \od$) and countable).
			
			We define the sequences $s_n\colon \NN \to \protoaux[n]$ inductively on $n \in \NN$. Let $s_0$ be the only possible map $\NN \to \protoaux[0]$, \ie the constant sequence with terms ${}_{0}()$. Suppose now that $n \in \NN_{\geq 1}$, and that we already defined $s_0, \ldots, s_{n-1}$. Note that the map $t\colon \NN \to \proto[n-1]$, defined by $t(k) \dfeq (s_0, \ldots, s_{n-1}) \circ b_n(k)$, is surjective. Thus the map $(t \times c) \circ b_2\colon \NN \to \proto[n-1] \times \od$ is surjective as well. Use Lemma~\ref{Lemma: finite_sequences_over_countable_alphabet_are_countable} to obtain $s_n$.
		\end{proof}
		
		\begin{lemma}\label{Lemma: countability_of_od_and_Urysohn}
			The following statements are equivalent.\footnote{Classically the implications $2 \impl 3 \impl 4$ are trivial since $\pseudo \subseteq \proto \subseteq \finseq{(\NN + \od)}$. Constructively some work is required, because a subset of a countable set need not be countable.}
			\begin{enumerate}
				\item
					$\od$ is countable.
				\item
					$\finseq{(\NN + \od)}$ is countable.
				\item
					$\proto$ is countable.
				\item
					$\pseudo$ is countable.
				\item
					$\Ury$ is countable.
			\end{enumerate}
		\end{lemma}
		\begin{proof}
			\begin{itemize}
				\item\proven{$(1 \impl 2)$}
					If $\od$ is countable, so is $\NN + \od$. Now use Lemma~\ref{Lemma: finite_sequences_over_countable_alphabet_are_countable}.
				\item\proven{$(2 \impl 1)$}
					The image of a countable set is countable, and $\od$ is the image of $\finseq{(\NN + \od)}$ via the map which takes the empty list and lists which start with a natural number to $0$, and a list which starts with $\lambda \in \od$ to $\lambda$. Obviously this map is surjective.
				\item\proven{$(1 \impl 3)$}
					Lemma~\ref{Lemma: countable_od_implies_countable_protoUrysohn} gives a surjection $\NN \times \NN \to \proto$. Precompose it with a bijection $\NN \ism \NN \times \NN$.
				\item\proven{$(3 \impl 4)$}
					Because $\pseudo$ is an image (even a retract) of $\proto$ by Theorem~\ref{Theorem: classification_of_permissible_tuples}.
				\item\proven{$(4 \impl 5)$}
					Because $\Ury$ is an image of $\pseudo$ via the Kolmogorov quotient map.
				\item\proven{$(5 \impl 1)$}
					Because $\od$ is an image of $\Ury$ via the map $a \mapsto d(a, [{}_{0}()])$. This is indeed a surjective mapping, as for every $\lambda \in \od$ we have $d\big([{}_{1}({}_{0}(), \lambda)], [{}_{0}()]\big) = \lambda$.
			\end{itemize}
		\end{proof}
		
		\begin{theorem}\label{Theorem: uniqueness_of_countable_Urysohn}
			Suppose $(U', d', \ext[\od]'\colon \prms[\od]' \to U')$ and $(U'', d'', \ext[\od]''\colon \prms[\od]'' \to U'')$ are $\od$-Urysohn spaces. Then there exists a mapping which takes any surjections $s'\colon \NN \to U'$ and $s''\colon \NN \to U''$ to an isometric isomorphism $U' \to U''$.\footnote{A classical mathematician writing this theorem would likely also add the assumption that some surjections $\NN \to U'$, $\NN \to U''$ actually exist. But the theorem is still true even if $U'$ and/or $U''$ aren't countable; we just get a mapping with an empty domain.}
		\end{theorem}
		\begin{proof}
			The standard proof using the so-called \df{back-and-forth method} goes as follows: inductively construct mutually inverse isometries between $U'$ and $U''$ by extending one isometry (using $\ext[\od]''$) over the first element in $s'$ on which it is not yet defined (extending also the other one to be inverse to it), then extending the other isometry (using $\ext[\od]'$) similarly. Continue this ad infinitum. Since $s'$ and $s''$ are surjective, we exhaust all elements in $U'$ and $U''$, thus obtaining surjective isometries between metric spaces, hence isometric isomorphisms.
			
			Hidden in this proof is the implicit assumption that $U'$ and $U''$ have decidable equality. We adopt the proof to work constructively as well.
			
			We inductively on $n \in \NN$ define $t'_n\colon \NN_{2 n} \to U'$, $t''_n\colon \NN_{2 n} \to U''$ and isometries $f_n\colon \im(t'_n) \to U''$ and $g_n\colon \im(t''_n) \to U'$ as follows. Let $t'_0$, $t''_0$, $f_0$ and $g_0$ be the empty maps (the only possibility, as they have the empty domain). Now suppose $t'_k$, $t''_k$, $f_k$, $g_k$ have been defined for all $k \in \NN_{\leq n}$, and denote
			$$a \dfeq \ext[\od]''\Big(\Big(f_n\big(t'_n(i)\big), d'\big(s'(n), t'_n(i)\big)\Big)_{i \in \NN_{< 2n}}\Big),$$
			$$b \dfeq \ext[\od]'\Big(\Big(g_n\big(t''_n(i)\big), d''\big(s''(n), t''_n(i)\big)\Big)_{i \in \NN_{< 2n}} \cnct \big(a, d''(s''(n), a)\big)\Big).$$
			We used $\ext[\od]'$ and $\ext[\od]''$ on elements of $\prms[\od]'$ and $\prms[\od]''$ because we took distances from metric spaces. Define:
			$$\rstr{t'_{n+1}}_{\NN_{< 2n}} \dfeq t'_n, \quad t'_{n+1}(2n) \dfeq s'(n), \quad t'_{n+1}(2n+1) \dfeq b,$$
			$$\rstr{t''_{n+1}}_{\NN_{< 2n}} \dfeq t''_n, \quad t''_{n+1}(2n) \dfeq a, \quad t''_{n+1}(2n+1) \dfeq s''(n),$$
			$$\rstr{f_{n+1}}_{\im(t'_n)} \dfeq f_n, \quad f(t'_{n+1}(2n)) \dfeq a, \quad f(t'_{n+1}(2n+1)) \dfeq s''(n),$$
			$$\rstr{g_{n+1}}_{\im(t''_n)} \dfeq g_n, \quad g(t''_{n+1}(2n)) \dfeq s'(n), \quad g(t''_{n+1}(2n+1)) \dfeq s''(n).$$
			The defining property of $\ext[\od]'$ and $\ext[\od]''$ implies that $f_{n+1}$ and $g_{n+1}$ are well defined (for example, $t'_{n+1}(2n)$ might equal some previous term, but then their distance is zero, as is the distance of their $f$-images which then match), and that they are isometries.
			
			Let $f$ and $g$ be colimits of $f_n$s and $g_n$s, respectively. We see that they are total on $U'$, $U''$ since $s'$, $s''$ are surjective and $s'(\NN_{< n}) \subseteq \im(t'_n)$ and $s''(\NN_{< n}) \subseteq \im(t''_n)$. By construction they are mutually inverse isometries between $U'$ and $U''$.
		\end{proof}
		
		\begin{corollary}\label{Corollary: uniqueness_of_countable_Urysohn}
			Up to isometric isomorphism there exists at most one countable $\od$-Urysohn space. Thus if $\od$ is countable, then $\Ury$ is (up to isometric isomorphism) the sole countable $\od$-Urysohn space.
		\end{corollary}
		\begin{proof}
			Any $\od$-Urysohn space has to be inhabited (as every countable $\od$-metric space, including $\one$, can be embedded into it by Corollary~\ref{Corollary: countable_Urysohn_embedding_property}), thus for a countable one there exists a surjection from $\NN$ onto it. The first part of the corollary now follows from the preceding theorem. For the second one use the fact that $\Ury$ is indeed a $\od$-Urysohn space, and moreover countable by Lemma~\ref{Lemma: countability_of_od_and_Urysohn} if $\od$ is.
		\end{proof}

	\section{Real Numbers and Metric Completion}\label{Section: reals_and_completions}
	
		We want to construct the Urysohn space as the completion of $\Ury$ for a suitable $\od$ which leads us to the question what is a completion of a metric space. Classically one constructs a completion as the set of equivalence classes of Cauchy sequences; call this the \df{Cauchy completion}, and call a metric space in which every Cauchy sequence converges \df{Cauchy complete}. This construction is problematic in our case for two reasons. First of all, constructively (when not assuming countable choice) this theory does not work well since the Cauchy completion need not be Cauchy complete~\cite{MALQ:MALQ200710007}. Second, even if we are not concerned about constructivism, there is a method of completing a space which lends itself far better to our construction of the Urysohn space (and is arguably simpler, in particular no quotients are involved). Before we can present it however, we need to say something about real numbers.
		
		Normally one does not bother with how the reals are explicitly constructed; one merely uses the fact that they are a field with all the rest of the structure. We will not have this luxury; we will in some cases need to explicitly prove that something is/determines a real number. However, in the spirit of proving our theorems in as general setting as we can, we prefer not to choose a specific model of reals, as different varieties of constructivism use different ones. Therefore, instead of choosing a construction of reals, we make some postulates about them.
		
		\begin{postulate}
			The set of real numbers $\RR$ is a halved lattice ring. Moreover, it is equipped with a relation $<$ (the strict order) which satisfies the following conditions for all $a, b, x \in \RR$.
			\begin{itemize}
				\item
					$\lnot(a < b) \iff b \leq a$
				\item
					$\lnot(a < b \land b < a)$ \quad (asymmetry)
				\item
					$a < b \implies a < x \lor x < b$ \quad (cotransitivity)
				\item
					$a < b \iff a + x < b + x$ \quad (additivity)
				\item
					$0 < x \implies (a < b \iff a \cdot x < b \cdot x)$
				\item
					$x < 0 \lor 0 < x \iff \xsome{y}{\RR}{x \cdot y = 1}$
				\item
					$0 < x \iff \xsome{n}{\NN}{2^{-n} \leq x}$
			\end{itemize}
		\end{postulate}
		The last condition is (in the presence of others) actually the Archimedean axiom in disguise.
		
		\begin{corollary}
			\
			\begin{enumerate}
				\item
					$\RR$ is a commutative group with absolute value, and therefore also a metric space with the Euclidean metric.
				\item
					$\nnr$ is a halved disring.
			\end{enumerate}
		\end{corollary}
		\begin{proof}
			By Propositions~\ref{Proposition: absolute_value_and_lattice_groups} and~\ref{Proposition: classification_of_halved_dis}.
		\end{proof}
		
		The second postulate describes the property that a real number can be given in terms of its (arbitrarily good) lower and upper approximations. Essentially we are saying that $\RR$ is Dedekind complete.
		\begin{postulate}\label{Postulate: reals_Dedekind_complete}
			Let $L, U \subseteq \RR$ have the properties
			\begin{itemize}
				\item
					$\xall{q}{L}\xall{r}{U}{q \leq r}$,
				\item
					$\xall{\epsilon}{\RR_{> 0}}\xsome{q}{L}\xsome{r}{U}{r \leq q + 2 \epsilon}$.
			\end{itemize}
			Then there exists a unique $x \in \RR$ such that $\sup{L} = x = \inf{U}$.\footnote{Obviously multiplying $\epsilon$ by $2$ in the second condition doesn't change the content of the statement, but this form is more useful since in practice we usually determine both $r$ and $q$ up to $\epsilon$ away from $x$, and then they differ by as much as twice this amount.}
		\end{postulate}
		
		Let us now return to metric spaces. The presence of the relation $<$ on the reals (something which we didn't have in a general disring, but see Subsection~\ref{Subsection: completion_of_disgroups}) lets us no longer defer the standard metric definitions which use it.
		\begin{definition}
			Let $\mtr{X} = (X, d_\mtr{X})$ be a protometric space.
			\begin{itemize}
				\item
					The subset
					$$\ball[\mtr{X}]{x}{r} \dfeq \st{y \in X}{d_\mtr{X}(x, y) < r}$$
					is called the (\df{open}) \df{ball with the center $x \in X$ and the radius $r \in \RR$}.
				\item
					A subset $A \subseteq X$ is \df{dense} in the space $\mtr{X}$ when every ball with a positive radius intersects it, \ie when
					$$\xall{x}{X}\xall{r}{\RR_{> 0}}\xsome{a}{A}{d_\mtr{X}(x, a) < r}$$
					holds.
			\end{itemize}
		\end{definition}
		
		\begin{lemma}\label{Lemma: density_of_od}
			Let $\od$ be a halved subdisgroup of $\nnr$ with $1 \in \od$ (\eg $\od$ is a halved subdisring).
			\begin{enumerate}
				\item
					$\od$ contains all non-negative diadic rational numbers, \ie $\xall{m, n}{\NN}{\frac{m}{2^n} \in \od}$.
				\item
					$\od$ is dense in $\nnr$.
			\end{enumerate}
		\end{lemma}
		\begin{proof}
			\begin{enumerate}
				\item
					Since $\od$ contains $0$ and $1$ and is closed for addition, it contains all natural numbers. Due to the halving map it then contains all non-negative diadic rationals.
				\item
					Because non-negative diadic rationals are dense in $\nnr$.
			\end{enumerate}
		\end{proof}
		
		In addition to the maps in Definition~\ref{Definition: maps}, the following classes of maps will also be relevant to us.
		\begin{definition}\label{Definition: more_maps}
			The map $f\colon X \to Y$ between protometric spaces $\mtr{X} = (X, d_\mtr{X})$, $\mtr{Y} = (Y, d_\mtr{Y})$ is:
			\begin{itemize}
				\item
					a \df{dense isometry} when it is an isometry with a dense image in $\mtr{Y}$,
				\item
					a \df{Lipschitz map} when there exists a \df{Lipschitz coefficient} $C \in \RR_{> 0}$, such that for every $x, y \in X$ we have $d_\mtr{Y}\big(f(x), f(y)\big) \leq C \cdot d_\mtr{X}(x, y)$,
				\item
					an \df{area Lipschitz map} when there exists $R \in \RR_{> 0}$ such that $f$ is Lipschitz on all balls of radius $R$ in $X$, that is,
			 $$\xsome{R}{\RR_{> 0}}\xall{x}{X}\xsome{C}{\RR_{> 0}}\all{y}{\ball[\mtr{X}]{x}{r}}{d_\mtr{Y}(f(x), f(y)) \leq C \cdot d_\mtr{X}(x, y)},$$
				\item
					\df{continuous} when it satisfies the usual $\epsilon$-$\delta$ condition
					$$\xall{x}{X}\xall{\epsilon}{\RR_{> 0}}\xsome{\delta}{\RR_{> 0}}\xall{y}{\ball[\mtr{X}]{x}{\delta}}{d_\mtr{Y}(f(x), f(y)) < \epsilon}.$$
			\end{itemize}
		\end{definition}
		It should be clear that these properties imply the later ones.
		
		\begin{remark}
			In the definition of a Lipschitz map we purposefully restrict the coefficient to be a positive number, because we often divide by it. It changes nothing, as the Lipschitz coefficient can always be increased. The definition of an area Lipschitz map is new. Since we are saying that a map is Lipschitz on some balls, one might also consider the name \df{locally Lipschitz}, but this would be misleading, I think. The point of a local property is that it holds on arbitrarily small balls, but here the purpose is quite different: we want the Lipschitz property on sufficiently \emph{large} balls.
		\end{remark}
		
		As is well known, a continuous map between metric spaces is determined already by its values on a dense subset. We recall the proof just so that we notice that the domain of the map can more general.
		\begin{lemma}\label{Lemma: uniqueness_of_continuous_extensions}
			Let $f\colon X \to Y$ be a continuous map from a protometric space $\mtr{X} = (X, d_\mtr{X})$ to a metric space $\mtr{Y} = (Y, d_\mtr{Y})$. Let $i\colon X \to X'$ be a dense isometry between protometric spaces $\mtr{X}$, $\mtr{X'} = (X', d_\mtr{X'})$. Then there exists at most one continuous map $X' \to Y$ which extends $f$, \ie for all continuous maps $g, h\colon X' \to Y$ the statement $g \circ i = f = h \circ i$ implies $g = h$.
		\end{lemma}
		\begin{proof}
			Take any $a \in X'$ and suppose $d_\mtr{Y}(g(a), h(a)) > 0$. Let $\epsilon \dfeq \frac{d_\mtr{Y}(g(a), h(a))}{2}$. By continuity of $g$ and $h$ there exists $\delta \in \RR_{> 0}$, so that $d_\mtr{Y}(g(a), g(b)) < \epsilon$ and $d_\mtr{Y}(h(a), h(b)) < \epsilon$ for all $b \in X'$ less that $\delta$ away from $a$. Let $x \in X$ be such, that $d_\mtr{X'}(i(x), a) < \delta$. Then
			$$d_\mtr{Y}(g(a), h(a)) \leq d_\mtr{Y}(g(a), g(i(x))) + d_\mtr{Y}(g(i(x)), h(i(x))) + d_\mtr{Y}(h(i(x)), h(a)) <$$
			$$< \epsilon + 0 + \epsilon = d_\mtr{Y}(g(a), h(a)),$$
			a contradiction, so $d_\mtr{Y}(g(a), h(a)) = 0$. Since $\mtr{Y}$ is metric, $g(a) = h(a)$.
		\end{proof}
		
		In a similar vein we can test the relation $\leq$ between continuous maps just on dense subsets of their domain.
		\begin{corollary}\label{Corollary: comparison_on_dense_subsets}
			Let $\mtr{X} = (X, d_\mtr{X})$ be a protometric space and $A \subseteq X$ its dense subset.
			\begin{enumerate}
				\item\label{Corollary: comparison_on_dense_subsets: leq}
					Let $f, g\colon X \to \RR$ be continuous maps, such that $f(a) \leq g(a)$ for all $a \in A$. Then $f(x) \leq g(x)$ for all $x \in X$.
				\item\label{Corollary: comparison_on_dense_subsets: suprema}
					Let $f\colon X \to \RR$ be a continuous map. If the supremum $s_A \dfeq \sup\st{f(a)}{a \in A}$ exists (as a real number), then so does $s_X \dfeq \sup\st{f(x)}{x \in X}$, and they are equal.
			 \end{enumerate}
		\end{corollary}
		\begin{proof}
			\begin{enumerate}
				\item
					The maps $x \mapsto f(x) + f(x) \dis g(x)$ and $x \mapsto g(x)$ match on $A$ (and are continuous since $f$ and $g$ as well as $\dis$ on $\RR$ are), and therefore on the whole $X$ by the previous lemma.
				\item
					Suppose $s_A$ exists; by the previous item $f(x) \leq s_A$ for all $x \in X$ which is sufficient for the existence of $s_X$ and the equality $s_X = s_A$.
			\end{enumerate}
		\end{proof}
		
		\begin{lemma}\label{Lemma: properties_of_continuous_maps_inferred_from_their_dense_restriction}
			Let $\mtr{X} = (X, d_\mtr{X})$, $\mtr{Y} = (Y, d_\mtr{Y})$, $\mtr{Z} = (Z, d_\mtr{Z})$ be protometric spaces, $i\colon X \to Y$ a dense isometry and $f\colon Y \to Z$ a continuous map. Then $f$ is a dense isometry/an isometry/non-expansive/Lipschitz/area Lipschitz if and only if its restriction $f \circ i$ is (with the same parameters, such as the Lipschitz coefficient).
		\end{lemma}
		\begin{proof}
			It is straightforward that restrictions of such maps also have these same properties. For the converse note that by Corollary~\ref{Corollary: comparison_on_dense_subsets}(\ref{Corollary: comparison_on_dense_subsets: leq}), if $d_\mtr{Z}(f(i(x)), f(i(y))) \leq C \cdot d_\mtr{X}(x, y)$ (resp.~$d_\mtr{Z}(f(i(x)), f(i(y))) \geq C \cdot d_\mtr{X}(x, y)$) holds on some $A \subseteq X$, then $d_\mtr{Z}(g(x)), g(y))) \leq C \cdot d_\mtr{Y}(x, y)$ (resp.~$d_\mtr{Z}(g(x)), g(y))) \geq C \cdot d_\mtr{Y}(x, y)$) holds on any subset of $Y$ into which $A$ densely embeds via $i$. Also, since the image of $f \circ i$ is contained in the image of $f$, if $f \circ i$ is dense, so is $f$.
		\end{proof}
		
		In the remainder of the section we discuss the completness of (pseudo)metric spaces. It is useful to have a definition of completness which is independent of the model. The definition below is the formalization of the fact that the completion is the largest metric space into which a metric space can be densely isometrically embedded. Also, we generalize the notion to include protometric spaces.
		\begin{definition}\label{Definition: completion}
			The \df{completion} of a protometric space $\mtr{X}$ is a space $\cmtr{X}$, together with a dense isometry $i\colon \mtr{X} \to \cmtr{X}$, such that for every dense isometry $f\colon \mtr{X} \to \mtr{Y}$ there exists a unique dense isometry $g\colon \mtr{Y} \to \cmtr{X}$, for which $i = g \circ f$.
		\end{definition}
		This can be succinctly put in categorical terms. Let $\mathcal{M}$ be the category of (proto)metric spaces and dense isometries. Then the completion of $\mtr{X}$ is the terminal object in the coslice category $\mtr{X}/\mathcal{M}$. Since is it given by a universal property, it is determined up to (in this case isometric) isomorphism.
		
		We say that a space $\mtr{X}$ is \df{complete} when its identity (equivalently, any isometric isomorphism with domain $\mtr{X}$) is its completion.
		
		When the dense isometry $i\colon \mtr{X} \to \cmtr{X}$ is understood, we often simply say that the completion of $\mtr{X}$ is just the space $\cmtr{X}$. Clearly, if $f\colon \mtr{X} \to \mtr{Y}$ is a dense isometry, then $\mtr{X}$ and $\mtr{Y}$ have the ``same'' completion, in the sense that if $j\colon \mtr{Y} \to \mtr{Z}$ is the completion of $\mtr{Y}$, then $j \circ f\colon \mtr{X} \to \mtr{Z}$ is the completion of $\mtr{X}$, and if $i\colon \mtr{X} \to \mtr{Z}$ is the completion of $\mtr{X}$ and $j\colon \mtr{Y} \to \mtr{Z}$ the unique dense isometry for which $i = j \circ f$, then $j$ is the completion of $\mtr{Y}$. In particular, a completion of a space is complete.
		\begin{remark}
			Note also, that a completion is always a metric space since the Kolmogorov quotient map is a surjective, hence dense, isometry.
		\end{remark}
		
		We recall two models of completion: the one with locations, and (assuming countable choice) the one with Cauchy sequences. The idea for the first is that points $a \in X$ in a metric space are in bijective correspondence with maps $d(a, \insarg)\colon X \to \nnr$ --- the inverse correspondence is taking the unique zero. It turns out that maps of the form $d(a, \insarg)$ are precisely characterized as maps $f\colon X \to \nnr$ which satisfy the triangle inequality $|f(x) - d(x, y)| \leq f(y)$ and have a zero (its uniqueness follows from the previous condition). However, note that these maps are non-expansive, and as such are determined by its values on a dense subset. Restricting to a dense subset, ``having a zero'' becomes ``attaining arbitrarily small positive values''.
		\begin{definition}
			Let $\mtr{X} = (X, d_\mtr{X})$ be a (pseudo)metric space. A map $f\colon X \to \nnr$ is called a \df{location}~\cite{Richman_F_2000:_the_fundamental_theorem_of_algebra_a_constructive_development_without_choice} on $\mtr{X}$ when
			\begin{itemize}
				\item
					$d_\mtr{X}(x, y) \dis f(x) \leq f(y)$ for all $x, y \in X$, and
				\item
					$\xall{\epsilon}{\RR_{> 0}}\xsome{x}{X}{f(x) \leq \epsilon}$.
			\end{itemize}
			We denote the set of locations on $\mtr{X}$ by $\locd(\mtr{X})$.
		\end{definition}
		From the above discussion we see that locations on $\mtr{X}$ ought to represent distance maps from points in the completion of $\mtr{X}$, and hence represent points of completion themselves. To obtain the completion of $\mtr{X}$ we thus need to equip $\locd(\mtr{X})$ with a metric and provide a dense isometry from $\mtr{X}$ into it.
		
		For locations $f, g\colon X \to \nnr$ define
		$$d_{\locd(\mtr{X})}(f, g) \dfeq \sup\st{f(x) \dis g(x)}{x \in X} = \inf\st{f(x) + g(x)}{x \in X}.$$
		To see that this supremum and infimum indeed exist and are equal, use Postulate~\ref{Postulate: reals_Dedekind_complete}. Take any $x, y \in X$. Then
		$$f(x) \dis g(x) \leq f(x) \dis d(x, y) + d(x, y) \dis g(y) \leq f(y) + g(y),$$
		so the first condition from the postulate is satisfied. For the second, take any $\epsilon \in \RR_{> 0}$. Then there exists $x \in X$ such that $f(x) \leq \epsilon$. Hence
		$$f(x) + g(x) = g(x) - f(x) + 2 f(x) \leq f(x) \dis g(x) + 2 \epsilon.$$
		
		\begin{proposition}\label{Proposition: locations_are_completion}
			\
			\begin{enumerate}
				\item
					The map $d_{\locd(\mtr{X})}\colon \locd(\mtr{X}) \times \locd(\mtr{X}) \to \nnr$ is a metric on $\locd(\mtr{X})$.
				\item
					The map $c_{\mtr{X}}\colon X \to \locd(\mtr{X})$, given by $c_{\mtr{X}}(x) \dfeq d_\mtr{X}(x, \insarg)$, is a dense isometry.
				\item
					Let $g\colon X \to Y$ be a dense isometry between pseudometric spaces $\mtr{X}$ and $\mtr{Y} = (Y, d_\mtr{Y})$. Then there exists a unique continuous map (necessarily a dense isometry) $h\colon Y \to \locd(\mtr{X})$ such that $h \circ g = c_{\mtr{X}}$. That is, $\big(\locd(\mtr{X}), d_{\locd(\mtr{X})}\big)$, together with $c_{\mtr{X}}$, is a model of completion of $\mtr{X}$ in the sense of Definition~\ref{Definition: completion}.
			\end{enumerate}
		\end{proposition}
		\begin{proof}
			\begin{enumerate}
				\item
					Standard.
				\item
					Take $x, y \in X$; then
					$$d_{\locd(\mtr{X})}\big(c_{\mtr{X}}(x), c_{\mtr{X}}(y)\big) = \inf\st{d(x, z) + d(y, z)}{z \in X} \leq d(x, y) + d(y, y) = d(x, y),$$
					$$d_{\locd(\mtr{X})}\big(c_{\mtr{X}}(x), c_{\mtr{X}}(y)\big) = \sup\st{d(x, z) \dis d(y, z)}{z \in X} \geq d(x, y) \dis d(y, y) = d(x, y).$$
					As for density, take any $f \in \locd(\mtr{X})$ and $\epsilon \in \RR_{> 0}$. There is $x \in X$ such that $f(x) \leq \epsilon$, and then
			$$d_{\locd(\mtr{X})}\big(c_{\mtr{X}}(x), f\big) = \inf\st{d(x, y) + f(y)}{y \in X} \leq d(x, x) + f(x) \leq \epsilon.$$
				\item
					Define $h(y)(x) \dfeq d_\mtr{Y}(g(x), y)$ for all $y \in Y$, $x \in X$. It is easy to see that this works. It is unique by Lemma~\ref{Lemma: uniqueness_of_continuous_extensions} and a dense isometry by Lemma~\ref{Lemma: properties_of_continuous_maps_inferred_from_their_dense_restriction}.
			\end{enumerate}
		\end{proof}
		
		We'll use the construction of completion by locations in this paper, but we want to say something about Cauchy sequences as well. Let $\Cauchy(\mtr{X})$ be the set of Cauchy sequences of a (pseudo)metric space $\mtr{X} = (X, d_\mtr{X})$, and equip it with its standard pseudometric, that is,
		$$d_{\Cauchy(\mtr{X})}\big((a_n)_{n \in \NN}, (b_n)_{n \in \NN}\big) \dfeq \lim_{n \to \infty} d_\mtr{X}(a_n, b_n).$$
		Further, let $\cs\colon X \to \Cauchy(\mtr{X})$ map a point $x \in X$ to a constant sequence with terms $x$. Clearly, $\cs$ is a dense isometry. Finally, let $q\colon \Cauchy(\mtr{X}) \to \Cauchy(\mtr{X})/_\equ$ be the Kolmogorov quotient map of the pseudometric space $\big(\Cauchy(\mtr{X}), d_{\Cauchy(\mtr{X})}\big)$. Then by definition the quotient space, together with the dense isometry $q \circ \cs$, is the \df{Cauchy completion} of $\mtr{X}$. We say that a space is \df{Cauchy complete} when it is isometrically isomorphic to its Cauchy completion.
		
		The universal property of the completion ensures that the Cauchy completion isometrically embeds into it (that is, the Cauchy completion can be regarded as a subspace of a completion), but this embedding need not be surjective in general. The proposition below recalls a sufficient condition for when it is.
		
		\begin{proposition}
			Assuming countable choice, the Cauchy completion is the completion in the sense of Definition~\ref{Definition: completion}.
		\end{proposition}
		\begin{proof}
			Let $f\colon X \to Y$ be a dense isometry between (pseudo)metric spaces $\mtr{X} = (X, d_\mtr{X})$ and $\mtr{Y} = (Y, d_\mtr{Y})$. Take an arbitrary $y \in Y$. By countable choice there exists a sequence $x_n \in X$, such that $d_{\mtr{Y}}(f(x_n), y) \leq 2^{-n-1}$. Note that any two such sequences are equivalent, so $g(y) = [(x_n)_{n \in \NN}]$ determines a well-defined map $g\colon Y \to \Cauchy(\mtr{X})/_\equ$. Observe that it is a dense isometry satisfying $g \circ f = q \circ \cs$, and is the only one such by Lemma~\ref{Lemma: uniqueness_of_continuous_extensions}.
		\end{proof}
		
		\begin{remark}
			In constructive and computational practice often not all Cauchy sequences are taken for the completion, but only those with some prescribed rate of convergence; for example, $(a_n)_{n \in \NN}$ is called a \df{rapid Cauchy sequence} when it satisfies the condition, that the distance between $a_n$ and $a_{n+1}$ is $\leq 2^{-n}$ for all $n \in \NN$. The theory still works under this restriction (in fact, the sequence we produced in the proof of the previous proposition is rapid Cauchy).
		\end{remark}
		
		Even if in general a Cauchy completion need not be complete, the converse does hold.
		\begin{proposition}\label{Proposition: complete_implies_Cauchy_complete}
			A complete space is also Cauchy complete.
		\end{proposition}
		\begin{proof}
			Let $\mtr{X} = (X, d_\mtr{X})$ be a complete metric space. Observe that the Cauchy completion $X \to \Cauchy(\mtr{X})/_\equ$ and the map $\Cauchy(\mtr{X})/_\equ \to X$, which exists by the definition of completion, are mutually inverse isometries.
		\end{proof}
		
		\begin{remark}
			We mentioned that the completion of a pseudometric space matches the completion of its Kolmogorov quotient (the same is true for the Cauchy completion). What about the protometric spaces? One can see that performing the completion by locations yields the same result as if we did it just for the kernel of the protometric space. Thus it would seem, that whatever a reasonable definition of a completion of protometric spaces is, it ought to match the completion of their kernels (but we won't need this in this paper).
		\end{remark}
		
		The universal property of completion tells us that dense isometries, defined on a dense subspace and mapping into a complete space, can be extended to the whole space. As is well known, this holds for more general maps.\footnote{See~\cite{Richman_F_2008:_real_numbers_and_other_completions} for the (constructive) proof for maps, uniformly continuous on bounded subsets.} For us, the relevant classes of maps will be non-expansive and area Lipschitz maps.
		
		\begin{proposition}\label{Proposition: extension_of_maps_into_complete_space}
			Let $\mtr{X} = (X, d_\mtr{X})$, $\mtr{Y} = (Y, d_\mtr{Y})$ be pseudometric spaces, $i\colon X \to Y$ a dense isometry between them, $\mtr{A} = (A, d_\mtr{A})$ a complete metric space, and $f\colon X \to A$ an area Lipschitz map. Then there exists a unique continuous map $g\colon Y \to A$ which extends $f$, \ie $g \circ i = f$. Moreover:
			\begin{itemize}
				\item
					$g$ is also area Lipschitz, for the same $R$,
				\item
					if $f$ is Lipschitz, so is $g$, with the same Lipschitz coefficient (in particular, if $f$ is non-expansive, so is $g$),
				\item
					if $f$ is an isometry, so is $g$.
			\end{itemize}
		\end{proposition}
		\begin{proof}
			The uniqueness of $g$ follows from Lemma~\ref{Lemma: uniqueness_of_continuous_extensions}. For its existence it is sufficient to construct an area Lipschitz extension $g'\colon Y \to \locd(\mtr{A})$; then $g$ is $g'$, composed with the isometric isomorphism $\locd(\mtr{A}) \ism A$.
			
			Let $R \in \RR_{> 0}$ witness that $f$ is area Lipschitz. Fix an arbitrary $y \in Y$ and $a \in A$, then let $z \in X$ be such, that $d_\mtr{Y}(i(z), y) < R$. Let $C \in \RR_{> 0}$ be a Lipschitz coefficient of $f$ on the ball $\ball[\mtr{X}]{z}{R}$. Declare:
			$$L \dfeq \st{d_\mtr{A}(a, f(x)) - C \cdot d_\mtr{Y}(i(x), y)}{x \in \ball[\mtr{X}]{z}{R}},$$
			$$U \dfeq \st{d_\mtr{A}(a, f(x)) + C \cdot d_\mtr{Y}(i(x), y)}{x \in \ball[\mtr{X}]{z}{R}}.$$
			Observe:
			\begin{itemize}
				\item
					for all $x, x' \in \ball[\mtr{X}]{z}{R}$
					$$d_\mtr{A}(a, f(x)) - C \cdot d_\mtr{Y}(i(x), y) \leq$$
					$$\leq d_\mtr{A}(a, f(x')) + d_\mtr{A}(f(x'), f(x)) - C \cdot d_\mtr{Y}(i(x), i(x')) + C \cdot d_\mtr{Y}(i(x'), y) \leq$$
					$$\leq d_\mtr{A}(a, f(x')) + C \cdot d_\mtr{X}(x', x) - C \cdot d_\mtr{X}(x, x') + C \cdot d_\mtr{Y}(i(x'), y) =$$
					$$= d_\mtr{A}(a, f(x')) + C \cdot d_\mtr{Y}(i(x'), y),$$
				\item
					for $\epsilon \in \RR_{> 0}$ we may find $x \in X$ such that $d_\mtr{Y}(i(x), y) \leq \inf\left\{\frac{\epsilon}{C}, R - d_\mtr{Y}(i(z), y)\right\}$, and then
					$$d_\mtr{A}(a, f(x)) + C \cdot d_\mtr{Y}(i(x), y) \leq d_\mtr{A}(a, f(x)) + \epsilon \leq d_\mtr{A}(a, f(x)) - C \cdot d_\mtr{Y}(i(x), y) + 2 \epsilon.$$
			\end{itemize}
			Let $s \dfeq \sup L = \inf U$ be the real number, determined by $L$, $U$ by Postulate~\ref{Postulate: reals_Dedekind_complete}.
			
			Define $g'(y)(a) \dfeq s$. We skip the technical verification that this works; do recall Lemma~\ref{Lemma: properties_of_continuous_maps_inferred_from_their_dense_restriction} however for the last part of the proposition.
		\end{proof}

	\section{Complete Urysohn Space}\label{Section: complete_Urysohn}
	
		Following the classical development, we now identify the Urysohn space as the completion of its ``countable version'', constructed in Section~\ref{Section: countable_Urysohn}. As such, we assume that $\od$ is a halved subdisgroup of $\nnr$ containing $1$ (and therefore all dyadic rationals by Lemma~\ref{Lemma: density_of_od}).
		
		In analogy with Definition~\ref{Definition: od-Urysohn} and discussion below it we provide the following definition.
		\begin{definition}\label{Definition: Urysohn}
			A \df{Urysohn space} is a tuple $(U', d', \ext'\colon \prms' \to U')$ where
			\begin{itemize}
				\item
					$(U', d')$ is a complete separable metric space,
				\item
					$\prms' \dfeq \st{(x_i, \chi_i)_{i \in \NN_{< l}} \in \finseq{(U' \times \nnr)}}{\all{i, j}{\NN_{< l}}{d'(x_i, x_j) \dis \chi_i \leq \chi_j}}$, and
				\item
					the map $\ext'\colon \prms[\od]' \to U'$ satisfies the property
					$$\xall{(x_i, \chi_i)_{i \in \NN_{< l}}}{\prms'}\xall{k}{\NN_{< l}}{d'\big(\ext'((x_i, \chi_i)_{i \in \NN_{< l}}), x_k\big) = \chi_k}.$$
			\end{itemize}
		\end{definition}
		Below (in Theorem~\ref{Theorem: Urysohn_space}) we show, that these properties imply the standard Urysohn extension property.
		
		Theorem~\ref{Theorem: countable_Urysohn_space} suggests that $\U_\nnr$ is a good candidate for the Urysohn space --- indeed, if it were complete, we could extend isometries into it from a dense countable subset to the whole of seperable metric space (as per Proposition~\ref{Proposition: extension_of_maps_into_complete_space}). However, it is not complete, in spite of the fact that $\nnr$ and the Kolmogorov quotients of $\pseudo[n]$s are.\footnote{Recall a similar situation: individual $\RR^n$s are complete, but their ``union'' (more precisely, the colimit of embeddings $\RR^n \ism \RR^n \times \{0\} \hookrightarrow \RR^{n+1}$) isn't.}
		
		\begin{lemma}\label{Lemma: sequence_without_limit_in_Urysohn_over_R}
			Let $s\colon \NN \to \pseudo$ be a sequence, inductively defined as
			$$s_n \dfeq \strat[n]{\big(s_k, 2^{-k}\big)}{k \in \NN_{< n}}.$$
			\begin{enumerate}
				\item
					The sequence $s$ is well defined, that is, the age of $s_n$ is indeed $n$ for all $n \in \NN$, and the terms are permissible tuples.
				\item
					The sequence $s$ is a rapid Cauchy sequence.
				\item\label{Lemma: sequence_without_limit_in_Urysohn_over_R: separated}
					For all $n \in \NN$ and all $a = \atuple{a}{\alpha}{i} \in \pseudo$ with $\age{a} < n$ the statements
					$$d(a, s_n) \geq 2^{-n} \qquad \text{and} \qquad d(a, s_n) = d(a, s_{n+1})$$
					hold.
			\end{enumerate}
		\end{lemma}
		\begin{proof}
			\begin{enumerate}
				\item
					By induction on $n \in \NN$. If $\age{s_k} = k$ for $k \in \NN_{< n}$, then we can take the age of $s_n$ to be $n$. To obtain permissibility first note that $s_k$s are permissible by the induction hypothesis for $k \in \NN_{< n}$, and that this in particular implies
					$$d(s_k, s_l) = \begin{cases} 2^{-k} & \text{ if } k < l,\\ 0 & \text{ if } k = l,\\ 2^{-l} & \text{ if } k > l \end{cases}$$
					for $k, l \in \NN_{< n}$. From here the inequalities, required for permissibility of $s_n$, easily follow.
				\item
					We have $d(s_n, s_{n+1}) = 2^{-n}$ (Theorem~\ref{Theorem: distances_as_described_in_Urysohn_tuples}) because $s_n$s are permissible.
				\item
					We prove the two statements simultaneously, using induction on $n + \age{a}$. First, write
					$$d(a, s_n) = \sup\big(\st{d(a_i, s_n) \dis \alpha_i}{i \in \NN_{< \lnth{a}}} \cup \st{d(a, s_k) \dis 2^{-k}}{k \in \NN_{< n}}\big).$$
					Suppose that $d(a, s_n) < 2^{-n}$, and thus in turn $d(a_i, s_n) \dis \alpha_i < 2^{-n}$ and $d(a, s_k) \dis 2^{-k} < 2^{-n}$ for all $i \in \NN_{< \lnth{a}}$ and $k \in \NN_{< n}$. If we actually have an $a$ to consider, then $n > \age{a} \geq 0$, so we can take $k = n-1$, obtaining $d(a, s_{n-1}) \dis 2^{-(n-1)} < 2^{-n}$ which implies $d(a, s_{n-1}) > 2^{-n}$. On the other hand
					$$d(a, s_{n-1}) = \sup\big(\st{d(a_i, s_{n-1}) \dis \alpha_i}{i \in \NN_{< \lnth{a}}} \cup \st{d(a, s_k) \dis 2^{-k}}{k \in \NN_{< n-1}}\big).$$
					We have $d(a_i, s_{n-1}) = d(a_i, s_n)$ by the induction hypothesis, and so all the terms in this supremum are $< 2^{-n}$, a contradiction to $d(a, s_{n-1}) > 2^{-n}$. Hence $d(a, s_n) \geq 2^{-n}$.
					
					For the second part calculate
					$$d(a, s_{n+1}) = \sup\big(\st{d(a_i, s_{n+1}) \dis \alpha_i}{i \in \NN_{< \lnth{a}}} \cup \st{d(a, s_k) \dis 2^{-k}}{k \in \NN_{< n+1}}\big) =$$
					$$= \sup\big(\st{d(a_i, s_n) \dis \alpha_i}{i \in \NN_{< \lnth{a}}} \cup \st{d(a, s_k) \dis 2^{-k}}{k \in \NN_{< n}} \cup \{d(a, s_n) \dis 2^{-n}\}\big) =$$
					$$= \sup\{d(a, s_n), d(a, s_n) \dis 2^{-n}\}.$$
					The second equality holds by the induction hypothesis. Furthermore, since $d(a, s_n) \geq 2^{-n}$, we have
					$$d(a, s_n) \dis 2^{-n} = d(a, s_n) - 2^{-n} < d(a, s_n),$$
					so we conclude $d(a, s_{n+1}) = d(a, s_n)$.
			\end{enumerate}
		\end{proof}
		
		\begin{proposition}\label{Proposition: uncompleted_Urysohn_not_complete}
			$\U_\nnr$ is not complete, in fact not even Cauchy complete.
		\end{proposition}
		\begin{proof}
			We claim that the Cauchy sequence $s$ from Lemma~\ref{Lemma: sequence_without_limit_in_Urysohn_over_R} (composed with the Kolmogorov quotient map) is not convergent. To see this, take an arbitrary $[a] \in \U_\nnr$. Putting together both statements of item~\ref{Lemma: sequence_without_limit_in_Urysohn_over_R: separated} of the aforementioned lemma, we infer $d([a], [s_n]) \geq 2^{-\age{a}-1}$ for all $n \in \NN_{> \age{a}}$, so $[a]$ cannot be the limit of the sequence.
		\end{proof}
		
		We (preliminarily) define $\U$ (its metric we again denote by $d$) to be the completion of $\U_{\nnr}$. However, we wish to show that $\U$ can be obtained by completing other $\Ury$s as well.
		
		Clearly if $\od'$, $\od''$ are disrings and $\od' \subseteq \od''$, then $\strat{\prt}{\od'} \subseteq \strat{\prt}{\od''}$, $\strat{\psd}{\od'} \subseteq \strat{\psd}{\od''}$ and $\strat{\U}{\od'} \subseteq \strat{\U}{\od''}$. In particular $\Ury \subseteq \strat{\U}{\nnr}$.
		
		The following lemma is essentially the inductive step for proving that $\pseudo$ is dense in $\U$ (but for the later course of proof it is more convenient to state it for a general subset $X \subseteq \U$). It is an exercise in choosing approximations in such a way that we obtain a permissible tuple.
		
		\begin{lemma}\label{Lemma: approximation_of_Urysohn_tuples}
			Let
			\begin{itemize}
				\item
					$X \subseteq \U$ such that $\xall{x}{X}\xall{r}{\RR_{> 0}}\xsome{a}{\pseudo}{d(x, [a]) \leq r}$,
				\item
					$l \in \NN$,
				\item
					$x_0, \ldots, x_{l-1} \in X$ and $\omega_0, \ldots, \omega_{l-1} \in \nnr$ such that $d(x_i, x_j) \dis \omega_i \leq \omega_j$ for all $i, j \in \NN_{< l}$,
				\item
					$\epsilon \in \RR_{> 0}$.
			\end{itemize}
			Then there exists $a = {}_{\age{a}}(a_i, \alpha_i)_{i \in \NN_{< l}} \in \pseudo$ such that $d(x_i, [a_i]) \leq \epsilon$ and $\omega_i \dis \alpha_i \leq \epsilon$ for all $i \in \NN_{< l}$.
		\end{lemma}
		\begin{proof}
			If $l = 0$, then ${}_{0}()$ works, and we are done. In the remainder of the proof assume $l \geq 1$. Also, we do not bother writing ages of the tuples we construct; just take them to be the supremum of ages of predecessors plus one.
			
			Let $\lambda \dfeq \frac{\epsilon}{4 l} > 0$. For each of finitely many $i \in \NN_{< l}$ choose $\alpha_i \in (\omega_i + \intoo{3 \lambda}{4 \lambda}) \cap \od$ (recall that $\od$ is dense in $\nnr$ by Lemma~\ref{Lemma: density_of_od}) and $a'_i \in \pseudo$ such that $d(x_i, [a'_i]) \leq \lambda$. Let
			$$\delta_{i,j} \dfeq \begin{cases} 1, & \text{if } i = j,\\ 0, & \text{if } i \neq j, \end{cases}$$
			denote the \df{Kronecker delta}, and $d_{i,j} \dfeq d(a'_i, a'_j) + 3 \lambda (1-\delta_{i,j})$. Define $a_0, \ldots, a_{l-1} \in \pseudo$ inductively by
			$$a_k \dfeq \big(a_i, d_{k,i}\big)_{i \in \NN_{< k}} \cnct \big(a'_k, \sup\st{d(a_j, a'_k) \dis d_{k,j}}{j \in \NN_{< k}}\big);$$
			in particular $a_0 = (a'_0, 0)$. The calculations (by induction on $k > i, j$) below confirm these are indeed permissible tuples:
			$$d(a_i, a_j) \dis d_{k,i} = d_{i,j} \dis d_{k,i} = \big(d(a'_i, a'_j) + 3 \lambda (1 - \delta_{i,j})\big) \dis \big(d(a'_i, a'_k) + 3 \lambda \delta_{i,j}\big) =$$
			$$= d(a'_i, a'_j) \dis \big(d(a'_i, a'_k) + 3 \lambda \delta_{i,j}\big) \leq d(a'_i, a'_j) \dis d(a'_i, a'_k) + 3 \lambda \delta_{i,j} \leq$$
			$$\leq d(a'_j, a'_k) + 3 \lambda = d_{k,j},$$
			
			$$d(a_i, a'_k) \dis d_{k,i} \leq \sup\st{d(a_j, a'_k) \dis d_{k,j}}{j \in \NN_{< k}},$$
			
			$$d(a_i, a'_k) \leq d(a_i, a'_k) \dis d_{k,i} + d_{k,i} \leq \sup\st{d(a_j, a'_k) \dis d_{k,j}}{j \in \NN_{< k}} + d_{k,i},$$
			(the next two lines prove $\sup\st{d(a_j, a'_k) \dis d_{k,j}}{j \in \NN_{< k}} \leq d(a_i, a'_k) + d_{k,i}$)
			$$d(a_j, a'_k) \leq d(a_i, a_j) + d(a_i, a'_k) = d_{i,j} + d(a_i, a'_k) = d(a'_i, a'_j) + 3 \lambda (1 - \delta_{i,j}) + d(a_i, a'_k) \leq$$
			$$\leq d(a'_i, a'_k) + d(a'_k, a'_j) + 6 \lambda + d(a_i, a'_k) = d_{k,i} + d_{k,j} + d(a_i, a'_k),$$
			
			$$d_{k,j} = d(a'_k, a'_j) + 3 \lambda \leq d(a'_i, a'_j) + d(a'_i, a'_k) + 3 \lambda \leq d_{i,j} + d(a'_i, a'_k) + 3 \lambda =$$
			$$= d(a_i, a_j) + d_{k,i} \leq d(a_j, a'_k) + d(a_i, a'_k) + d_{k,i}.$$
			
			\begin{itemize}
				\item\proven{$(a_i, \alpha_i)_{i \in \NN_{< l}} \in \pseudo$}
					Take any $i, j \in \NN_{< l}$. The condition $d(a_i, a_j) \dis \alpha_i \leq \alpha_j$ clearly holds for $i = j$, so assume $i \neq j$.
					
					$$\alpha_i + \alpha_j \geq (\omega_i + 3 \lambda) + (\omega_j + 3 \lambda) = \omega_i + \omega_j + 6 \lambda \geq d(x_i, x_j) + 6 \lambda \geq$$
					$$\geq d(a'_i, a'_j) - d(x_i, [a'_i]) - d(x_j, [a'_j]) + 6 \lambda \geq d(a'_i, a'_j) + 4 \lambda \geq d_{i,j} = d(a_i, a_j)$$
					
					$$\alpha_i \leq \omega_i + 4 \lambda \leq d(x_i, x_j) + \omega_j + 4 \lambda \leq$$
					$$\leq d(x_i, [a'_i]) + d(a'_i, a'_j) + d(x_j, [a'_j]) + \omega_j + 4 \lambda \leq d(a'_i, a'_j) + \omega_j + 6 \lambda =$$
					$$= d_{i,j} + \omega_j + 3 \lambda \leq d_{i,j} + \alpha_j = d(a_i, a_j) + \alpha_j$$
				
				\item\proven{$\xall{i}{\NN_{< l}}{d(x_i, [a_i]) \leq \epsilon}$}
					We claim that $d(a_k, a'_k) \leq 3 k \lambda$ for all $k \in \NN_{< l}$. This clearly holds for $k = 0$. By induction, for $k \geq 1$,
					$$d(a_k, a'_k) = \sup\st{d(a_j, a'_k) \dis d_{k,j}}{j \in \NN_{< k}} =$$
					$$= \sup\st{d(a_j, a'_k) \dis (d(a'_k, a'_j) + 3 \lambda)}{j \in \NN_{< k}} \leq$$
					$$\leq \sup\st{(d(a_j, a'_k) \dis d(a'_k, a'_j)) + 3 \lambda}{j \in \NN_{< k}} \leq$$
					$$\leq \sup\st{d(a_j, a'_j)}{j \in \NN_{< k}} + 3 \lambda \leq 3 (k-1) \lambda + 3 \lambda = 3 k \lambda.$$
					Therefore
					$$d(x_i, [a_i]) \leq d(x_i, [a'_i]) + d([a'_i], [a_i]) = d(x_i, [a'_i]) + d(a'_i, a_i) \leq \lambda + 3 i \lambda \leq 3 l \lambda \leq \epsilon.$$
				
				\item\proven{$\xall{i}{\NN_{< l}}{\omega_i \dis \alpha_i \leq \epsilon}$}
					$$\omega_i \dis \alpha_i < 4 \lambda \leq 4 l \lambda \leq \epsilon$$
			\end{itemize}
		\end{proof}
		
		\begin{proposition}\label{Proposition: density_preserved_by_Urysohn_construction}
			$\Ury$ is dense in $\U_\nnr$ (equivalently, $\pseudo$ is dense in $\psd_{\nnr}$).
		\end{proposition}
		\begin{proof}
			We prove by induction on age $n \in \NN$ that $\pseudo[n]$ is dense in $\strat[n]{\psd}{\nnr}$. The proposition clearly holds for $n = 0$. Assume $n \geq 1$, and fix $r \in \RR_{> 0}$. Take $b = {}_{\age{b}}(b_i, \beta_i)_{i \in \NN_{< l}} \in {}_{n}V_\RR$, and suppose the proposition holds for ages less than $n$, in particular for predecessors of $b$. This means that we can use Lemma~\ref{Lemma: approximation_of_Urysohn_tuples} for $X = {}_{n-1}V_\RR$, $x_i = [b_i]$, $\omega_i = \beta_i$ and $\epsilon = \frac{r}{2}$ to obtain $a = {}_{\age{a}}(a_i, \alpha_i)_{i \in \NN_{< l}} \in \pseudo$ so that $d(a_i, b_i) \leq \frac{r}{2}$ and $\alpha_i \dis \beta_i \leq \frac{r}{2}$ for all $i \in \NN_{< l}$. By Proposition~\ref{Proposition: similar_distances_and_points_yield_similar_extensions_in_Urysohn_space} $d(a, b) \leq r$.
		\end{proof}
		
		We preliminarily defined $\U$ to be the completion of $\U_\nnr$, but we now see that we could define it as the completion of $\Ury$ for \emph{any} halved subdisgroup $\od \subseteq \nnr$ containing $1$.
		\begin{corollary}\label{Corollary: completion_of_dense_is_Urysohn}
			$\U$ is the completion of $\Ury$.
		\end{corollary}
		\begin{proof}
			By Proposition~\ref{Proposition: density_preserved_by_Urysohn_construction}.
		\end{proof}
		
		Having constructed $\U$, we now turn our attention to proving its Urysohn properties. Let
		$$\prms \dfeq \st{(x_i, \omega_i)_{i \in \NN_{< l}} \in (\U \times \nnr)^*}{\all{i, j}{\NN_{< l}}{d(x_i, x_j) \dis \omega_i \leq \omega_j}}.$$
		
		\begin{lemma}\label{Lemma: completed_Urysohn_extension}
			For all $x = \tuple{}{l}{x}{\chi}{h} \in \prms$ declare the map $f_x\colon \pseudo \to \RR$ to be defined for $a = \atuple{a}{\alpha}{i} \in \pseudo$ inductively on $\age{a}$ as
			$$f_x(a) \dfeq \sup\big(\st{d(x_h, [a]) \dis \chi_h}{h \in \NN_{< l}} \cup \st{f_x(a_i) \dis \alpha_i}{i \in \NN_{< \lnth{a}}}\big).$$
			\begin{enumerate}
				\item
					Let $\epsilon, \epsilon' > 0$ and $c = \tuple{\age{c}}{l}{c}{\gamma}{k} \in \pseudo$ be such that $d(x_k, [c_k]) \leq \epsilon$ and $\chi_k \dis \gamma_k \leq \epsilon'$ for all $k \in \NN_{< l}$. Then
					$$f_x(a) \dis d(c, a) \leq \epsilon + \epsilon'$$
					for all $a = \atuple{a}{\alpha}{i} \in \pseudo$.
				\item
					The map $f_x$ is a location on $\pseudo$.
				\item
					We have $d_{\locd(\pseudo)}\big(f_x, d(x_h, [\insarg])\big) = \chi_h$ for all $h \in \NN_{< l}$.
			\end{enumerate}
		\end{lemma}
		\begin{proof}
			\begin{enumerate}
				\item
					By induction on $\age{a}$. It is equivalent to prove $f_x(a) \leq d(c, a) + \epsilon + \epsilon'$ and $d(c, a) \leq f_x(a) + \epsilon + \epsilon'$. We calculate
					$$d(x_k, [a]) \dis \chi_k \leq (d(x_k, [a]) \dis d(c_k, a)) + (d(c_k, a) \dis \gamma_k) + (\gamma_k \dis \chi_k) \leq$$
					$$\leq d(x_k, [c_k]) + (d(c_k, a) \dis \gamma_k) + (\gamma_k \dis \chi_k) \leq \epsilon + d(c, a) + \epsilon'$$
					for all $k \in \NN_{< l}$ and, using induction hypothesis,
					$$f_x(a_i) \dis \alpha_i \leq (f_x(a_i) \dis d(c, a_i)) + (d(c, a_i) \dis \alpha_i \leq (\epsilon + \epsilon') + d(c, a)$$
					for all $i \in \NN_{< \lnth{a}}$, proving the first claim. As for the second,
					$$d(c_k, a) \dis \gamma_k \leq (d(c_k, a) \dis d(x_k, [a])) + (d(x_k, [a]) \dis \chi_k) + (\chi_k \dis \gamma_k) \leq$$
					$$\leq d(x_k, [c_k])) + (d(x_k, [a]) \dis \chi_k) + (\chi_k \dis \gamma_k) \leq \epsilon + f_x(a) + \epsilon'$$
					and
					$$d(c, a_i) \dis \alpha_i \leq (d(c, a_i) \dis f_x(a_i)) + (f_x(a_i) \dis \alpha_i) \leq (\epsilon + \epsilon') + f_x(a).$$
				\item
					We need to prove $d(a, b) \dis f_x(a) \leq f_x(b)$ for all $a = \atuple{a}{\alpha}{i},$ $b = \atuple{b}{\beta}{j} \in \pseudo$. We do so by induction on $\age{a} + \age{b}$.
					
					Suppose $d(a, b) \dis f_x(a) > f_x(b)$, and let $\epsilon \dfeq \frac{d(a, b) \dis f_x(a) - f_x(b)}{5}$. Then $\epsilon > 0$, so using Lemma~\ref{Lemma: approximation_of_Urysohn_tuples} for $X = \U$, we may choose $c = \tuple{\age{c}}{l}{c}{\gamma}{k} \in \pseudo$ such that $d(x_k, c_k) \leq \epsilon$ and $\chi_k \dis \gamma_k \leq \epsilon$ for all $k \in \NN_{< l}$. Note that this implies
					$$d(x_k, [a]) \dis \chi_k \leq d(c_k, a) \dis \gamma_k + 2 \epsilon,$$
					$$d(c_k, b) \dis \gamma_k \leq d(x_k, [b]) \dis \chi_k + 2 \epsilon$$
					for all $k \in \NN_{< l}$. By the previous item we also have
					$$f_x(a_i) \dis \alpha_i \leq d(c, a_i) \dis \alpha_i + 2 \epsilon$$
					for all $i \in \NN_{< \lnth{a}}$. Thus
					$$d(a, b) \dis f_x(a) \leq d(c, a) \dis d(a, b) + 2 \epsilon \leq d(c, b) + 2 \epsilon \leq$$
					$$\leq f_x(b) + 4 \epsilon < f_x(b) + 5 \epsilon = d(a, b) \dis f_x(a),$$
					a contradiction.
				\item
					Recall that
					$$d_{\locd(\pseudo)}\big(f_x, d(x_h, [\insarg])\big) = \inf\st{f_x(a) + d(x_h, [a])}{a \in \pseudo} =$$
					$$= \sup\st{f_x(a) \dis d(x_h, [a])}{a \in \pseudo}.$$
					Thus it is sufficient to prove that $f_x(a) + d(x_h, [a]) \geq \chi_h$ and $f_x(a) \dis d(x_h, [a]) \leq \chi_h$ for all $a \in \pseudo$. We easily get
					$$f_x(a) + d(x_h, [a]) \geq d(x_h, [a]) \dis \chi_h + d(x_h, [a]) \geq \chi_h,$$
					$$\chi_h + f_x(a) \geq \chi_h + (d(x_h, [a]) \dis \chi_h) \geq d(x_h, [a]),$$
					but we still need to verify $f_x(a) \leq d(x_h, [a]) + \chi_h$. We prove $d(x_k, [a]) \dis \chi_k \leq d(x_h, [a]) + \chi_h$ by
					$$d(x_k, [a]) \leq d(x_h, [a]) + d(x_k, x_h) \leq d(x_h, [a]) + \chi_k + \chi_h,$$
					$$\chi_k - \chi_h \leq d(x_k, x_h) \leq d(x_k, [a]) + d(x_h, [a]).$$
					For $f_x(a_i) \dis \alpha_i \leq d(x_h, [a]) + \chi_h$ we use the induction hypothesis to calculate
					$$f_x(a_i) \leq d(x_h, [a_i]) + \chi_h \leq d(x_h, [a]) + d(a, a_i) + \chi_h = d(x_h, [a]) + \chi_h + \alpha_i,$$
					$$f_x(a_i) + d(x_h, [a]) + \chi_h \geq d(x_h, [a_i]) \dis \chi_h + \chi_h + d(x_h, [a]) \geq$$
					$$\geq d(x_h, [a_i]) + d(x_h, [a]) \geq d(a, a_i) = \alpha_i.$$
			\end{enumerate}
		\end{proof}
		
		\begin{theorem}\label{Theorem: Urysohn_space}
			The metric space $\Ury$ satisfies the properties of the Urysohn space. Explicitly, the following holds.
			\begin{enumerate}
				\item
					$\U$ is (an inhabited) complete separable metric space.
				\item\label{Theorem: Urysohn_space: extension_map}
					There is a map $\ext\colon \prms \to \U$ with the property $d(\ext((x_i, \omega_i)_{i \in \NN_{< l}}), x_k) = \omega_k$ for all $k \in \NN_{< l}$.
				\item\label{Theorem: Urysohn_space: extension_of_isometries}
					Let
					\begin{itemize}
						\item
							$\mtr{X} = (X, d_\mtr{X}, s)$ be a separable metric space,
						\item
							$F \subseteq X$ a finite subset with enumeration $F = \{y_0, \ldots, y_{k-1}\}$,
						\item
							$\mtr{F} = (F, d_{F})$ metric subspace of ${X}$, and
						\item
							$e\colon F \to \Ury$ an isometry.
					\end{itemize}
					Then there exists a canonical choice of an isometry $f\colon X \to \U$ such that $\rstr{f}_F = e$.
			\end{enumerate}
		\end{theorem}
		\begin{proof}
			\begin{enumerate}
				\item
					Let $\od$ be a countable, such as $\od = \QQ_{\geq 0}$. Using Lemma~\ref{Lemma: countability_of_od_and_Urysohn} and Corollary~\ref{Corollary: completion_of_dense_is_Urysohn} we see that $\U$ is the completion of an inhabited countable metric space $\Ury$.
				\item
					Lemma~\ref{Lemma: completed_Urysohn_extension} tells us that $\ext$, given as $x \mapsto f_x$, works.
				\item
					Define $f'\colon F \cup \big(s(\NN) \cap X\big) \to \Ury$ on $F$ by $f'(y_i) \dfeq e(y_i)$, and on $s(\NN) \cap X$ inductively on $n \in \NN$ as follows: if $s_n \in X$, then
					$$f'(s_n) \dfeq \ext\Big(\big(e(y_i), d_{X}(s_n, y_i)\big)_{i \in \NN_{< k}} \cnct \big(f(s_j), d_\mtr{X}(s_n, s_j)\big)_{j \in \NN_{< n} \cap s^{-1}(X)}\Big).$$
					The map $f'$ is well defined --- if $s_n$ equals some $x \in F \cup \big(s(\NN_{< n}) \cap X\big)$, then $d\big(f(s_n), f(x)\big) = 0$, so $f(s_n) = f(x)$ because $\U$ is a metric space. For the same reason $f'$ is an extension of $e$. It follows from the definition of $\ext$ that $f'$ is an isometry, and since $F \cup \big(s(\NN) \cap X\big)$ is metrically dense in $\mtr{X}$, it extends to the isometry $f\colon X \to \U$ by Proposition~\ref{Proposition: extension_of_maps_into_complete_space}.
			\end{enumerate}
		\end{proof}
		
		The corollary is that any metrically separable metric space isometrically embeds into $\U$ --- just take $F = \emptyset$ in the preceding theorem.
		
		To conclude the construction of the Urysohn space, we prove that it is unique up to isometric isomorphism.
		\begin{theorem}
			Let $(U, d_U, \ext[U])$ be a Urysohn space. Then there is an isometric isomorphism $U \ism \U$.\footnote{In fact, with more technical involvement one can show that the choice of the isomorphism is canonical (depending on the enumeration of a dense subset of $U$); compare with Theorem~\ref{Theorem: uniqueness_of_countable_Urysohn}.}
		\end{theorem}
		\begin{proof}
			Let $s_U\colon \NN \to U$ be an enumeration of a dense subset of $U$. Declare $\od$ to be the smallest halved subdisgroup of $\nnr$ which contains $1$ and the distances between terms of $s_U$. This $\od$ is countable since there are countably many pairs of natural numbers, hence countably many distances, hence countably many finite expressions involving these distances, $1$ and valid uses of disgroup operations and brackets, hence countably many their values. Let now $U_\od \subseteq U$ be the closure of the image of $s_U$ by the operation $\ext[U]$, restricted to the elements of the image of $s_U$ and the distances from $\od$. Again we are making and evaluating finite expressions over a countable alphabet, so $U_\od$ is countable. Moreover it is clearly a countable $\od$-Urysohn space (for restricted $d_U$ and $\ext[U]$), so there is an isometric isomorphism $U_\od \ism \Ury$ by Corollary~\ref{Corollary: uniqueness_of_countable_Urysohn}. The image of $s_U$, and therefore $U_\od$, is dense in $U$ while $\Ury$ is dense in $\U$ by Proposition~\ref{Proposition: density_preserved_by_Urysohn_construction} (and the definition of $\U$). Hence the isometric isomorphism $U_\od \ism \Ury$ extends to the one between $U$ and $\U$ (by Proposition~\ref{Proposition: extension_of_maps_into_complete_space}).
		\end{proof}

	\section{Continuity of Extensions}\label{Section: continuity}
	
		Theorem~\ref{Theorem: Urysohn_space} gives us a canonical way to extend finite partial isometries from a separable metric space $\mtr{X} = (X, d_\mtr{X})$ into the Urysohn space, \ie a mapping from the set of finite partial isometries $X \parto \U$ to the set of total ones $X \to \U$. In this section we show that this mapping is continuous.
		
		To do that, we need to topologize the domain and the codomain. We start with a simpler case, restricting to finite partial isometries of a given length. More precisely, for a metric space $\mtr{X} = (X, d_\mtr{X})$ let $\fpi[l]$ denote the set of isometries into $\U$, defined on lists of elements from $X$ of length $l \in \NN$.\footnote{We do not require that elements on a list all differ, so the domain of an isometry can have less than $l$ elements.} There is an obvious way how to represent such a set:
		$$\fpi[l] \dfeq$$
		$$= \st{\big((x_i)_{i \in \NN_{< l}}, (u_i)_{i \in \NN_{< l}}\big) \in X^l \times \U^l}{\xall{i, j}{\NN_{< l}}{d_\mtr{X}(x_i, x_j) = d(u_i, u_j)}}.$$
		Since $X$ and $\U$ are metric spaces, so is $X^l \times \U^l$. Thus $\fpi[l]$ is naturally topologized as its subspace. We obtain the same topology regardless of the product metric we choose on $X^l \times \U^l$, so we choose the one that is most convenient; for us this means the $\infty$-metric on $X^l$ and $\U^l$, and then the $1$-metric on their product. Explicitly, denoting this metric by $d_P$, this means
		$$d_P\big((x, u), (y, v)\big) \dfeq \sup\st{d_\mtr{X}(x_i, y_i)}{i \in \NN_{< l}} + \sup\st{d(u_i, v_i)}{i \in \NN_{< l}}$$
		for $(x, u) = \big((x_i)_{i \in \NN_{< l}}, (u_i)_{i \in \NN_{< l}}\big), (y, v) = \big((y_i)_{i \in \NN_{< l}}, (v_i)_{i \in \NN_{< l}}\big) \in \fpi[l]$.
		
		A hint why we opt for this product metric is the fact, that this combination naturally lends itself to proving that $\ext$ is non-expansive on tuples of a given length\footnote{Compare also with Proposition~\ref{Proposition: similar_distances_and_points_yield_similar_extensions_in_Urysohn_space}.} (the real reason, though, is Lemma~\ref{Lemma: estimation_of_supremum_metric_on_isometries}(\ref{Lemma: estimation_of_supremum_metric_on_isometries: dP_bound}) and its application in Theorem~\ref{Theorem: extension_nonexpansive} below).
		\begin{proposition}\label{Proposition: continuity_of_one_point_extension}
			Let ${\prms}_l$ denote the subset of $\prms$, containing the tuples of length $l \in \NN$. Then the restriction $\rstr{\ext}_{{\prms}_l}\colon {\prms}_l \to \U$ is continuous, in fact non-expansive if the metric on ${\prms}_l$ is given for as
			$$d_{{\prms}_l}\Big(\tuple{}{l}{x}{\chi}{i}, \tuple{}{l}{y}{\upsilon}{i}\Big) \dfeq \sup\st{d(x_i, y_i)}{i \in \NN_{< l}} + \sup\st{\chi_i \dis \upsilon_i}{i \in \NN_{< l}}.$$
		\end{proposition}
		\begin{proof}
			We wish to prove
			$$d\Big(\ext\big(\tuple{}{l}{x}{\chi}{i}\big), \ext\big(\tuple{}{l}{y}{\upsilon}{i}\big)\Big) \leq d_{{\prms}_l}\Big(\tuple{}{l}{x}{\chi}{i}, \tuple{}{l}{y}{\upsilon}{i}\Big).$$
			Let $\textrm{RHS}$ be the shorthand for the right-hand side. Recall from Theorem~\ref{Theorem: Urysohn_space} that $\ext\big(\tuple{}{l}{x}{\chi}{i}\big)$ and $\ext\big(\tuple{}{l}{y}{\upsilon}{i}\big)$ are given as locations; denote them by $f, g\colon \pseudo \to \RR$ respectively. We have $d_{\locd(\pseudo)}(f, g) = \sup\st{f(a) \dis g(a)}{a \in \pseudo}$, so it is sufficient to verify that $f(a) \dis g(a) \leq \textrm{RHS}$ for all $a = \atuple{a}{\alpha}{k} \in \pseudo$. We prove this by induction on $\age{a}$. Here is the proof that $f(a) \leq g(a) + \textrm{RHS}$; the inequality with $f$ and $g$ reversed is proved the same way. Recall from Lemma~\ref{Lemma: completed_Urysohn_extension} how $f$ and $g$ are given as certain suprema.
			$$d(x_i, [a]) \dis \chi_i \leq d(y_i, [a]) \dis \upsilon_i + d(x_i, y_i) + \chi_i \dis \upsilon_i \leq g(a) + \textrm{RHS}$$
			$$f(a_k) \dis \alpha_k \leq g(a_k) \dis \alpha_k + f(a_k) \dis g(a_k) \leq g(a) + \textrm{RHS}$$
		\end{proof}
		
		\begin{proposition}\label{Proposition: extension_of_Urysohn_properties}
			The following statement holds when we restrict\footnote{Alternatively, just define the topology on $\prms$ to be the coproduct topology on $\coprod_{l \in \NN} {\prms}_l$.} to tuples of length $l \in \NN$: $\prms[\od]$ is dense in $\prms$, and $\ext$ is the unique continuous extension of $\ext[\od]$.
		\end{proposition}
		\begin{proof}
			The proof of density is essentially the same as the proof of Proposition~\ref{Proposition: density_preserved_by_Urysohn_construction} (both $\pseudo$ and $\prms$ are supposed to contain those tuples which allow the extension). We know from the previous proposition that $\ext$ (and therefore also its restriction $\ext[\od]$) is continuous, so its uniqueness is given by Lemma~\ref{Lemma: uniqueness_of_continuous_extensions}.
		\end{proof}
		
		\begin{remark}
			Of course, since $\ext[\od]$ is non-expansive, we could have also defined $\ext$ as its unique continuous (and non-expansive) extension by Proposition~\ref{Proposition: extension_of_maps_into_complete_space}, but we prefered to provide an explicit formula for $\ext$.
		\end{remark}
		
		\begin{lemma}\label{Lemma: distances_in_completed_Urysohn}
			Let $(x_i, \chi_i)_{i \in \NN_{< l}}, (y_j, \upsilon_j)_{j \in \NN_{< k}} \in \prms$, and let
			$$x \dfeq \ext\big((x_i, \chi_i)_{i \in \NN_{< l}}\big), \qquad y \dfeq \ext\big((y_j, \upsilon_j)_{j \in \NN_{< k}}\big).$$
			Then
			$$d(x, y) = \sup\Big(\st{d(x_i, y) \dis \chi_i}{i \in \NN_{< l}} \cup \st{d(x, y_j) \dis \upsilon_j}{j \in \NN_{< k}}\Big).$$
		\end{lemma}
		\begin{proof}
			Observe that both sides of the equality restrict to the same map on $\prms[\od] \times \prms[\od]$ by the definition of the distance on the Urysohn space. Thus they match by Proposition~\ref{Proposition: extension_of_Urysohn_properties} and Lemma~\ref{Lemma: uniqueness_of_continuous_extensions}.
		\end{proof}
		
		Define $\isom$ to be the set of isometries from $X$ to $\U$. We wish to topologize it. Being a subset of the set of continuous maps between metric spaces, there are three standard candidates: topology of pointwise convergence, of uniform convergence on compact subsets (\ie compact-open topology), and of uniform convergence. We claim continuity of extensions for all of them; thus we choose the last option because it contains the other two, and a continuous map remains continuous if the topology on its codomain is replaced by a weaker one.
		
		Recall that the topology of uniform convergence is given by the basis
		$$\st[1]{U(f, r)}{f \in \isom \land r \in \RR_{> 0}}$$
		where
		$$U(f, r) \dfeq \st{g \in \isom}{\xsome{r'}{\intoo{0}{r}}\xall{x}{X}{d(f(x), g(x)) \leq r'}}.$$
		Here we require $r'$ to ensure that basic sets are open in the topology they generate (we would not need it we considered the topology of uniform convergence on compact subsets, or if $X$ were necessarily compact, as then writing $\xall{x}{X}{d(f(x), g(x)) < r}$ would imply the existence of a smaller bound $r'$).
		
		The sets $U(f, r)$ are reminiscent of metric balls, and indeed suitable subsets of $\isom$ are metrizable.
		\begin{lemma}\label{Lemma: sup_metric_induces_uniform_convergence}
			Let $A \subseteq \isom$ be a subset such that for every $f, g \in A$ the supremum
			$$d_{\sup}(f, g) \dfeq \sup\st{d(f(x), g(x))}{x \in X}$$
			exists (as a real number). Then $d_{\sup}$ is a metric on $A$, and the inclusion of $A$, topologized by $d_{\sup}$, into $\isom$ is a topological embedding. Moreover, for $f \in A$, the sets $U(f, r)$, restricted to $A$, are precisely the balls in $A$.
		\end{lemma}
		\begin{proof}
			The proof that $d_{\sup}$ is a metric is standard. The inclusion $A \hookrightarrow \isom$ is a topological embedding because every basic subset in $\isom$, restricted to $A$, is a union of balls in $A$:
			$$U(f, r) \cap A =$$
			$$= \bigcup\st[1]{\ball[\sup]{g}{\epsilon}}{g \in U(f, r) \cap A \land \epsilon \in \RR \land \all{x}{X}{d(f(x), g(x)) + \epsilon \leq r}}.$$
			Furthermore, if $f \in A$, then for every $g \in A$ and $r \in \RR$ the statements
			$$\xsome{r'}{\intoo{0}{r}}\xall{x}{X}{d(f(x), g(x)) \leq r'} \quad \text{and} \quad \sup\st{d(f(x), g(x))}{x \in X} < r$$
			are equivalent, so $U(f, r) \cap A = \ball[\sup]{f}{r}$.
		\end{proof}
		
		Define $\extm[l]\colon \fpi[l] \to \isom$ to be the extension map, given by Theorem~\ref{Theorem: Urysohn_space}(\ref{Theorem: Urysohn_space: extension_of_isometries}).
		
		The following very technical lemma establishes that the distance between two extensions of finite partial isometries can be approximated arbitrarily well\footnote{From above, to be precise, but from below it is a lot easier (see the following theorem).} by the distance between their finite parts.
		
		\begin{lemma}\label{Lemma: estimation_of_supremum_metric_on_isometries}
			Let
			\begin{itemize}
				\item
					$\mtr{X} = (X, d_\mtr{X}, s)$ be a separable metric space, with $s\colon \NN \to \one + X$ giving an enumeration of a countable dense subset of $X$,
				\item
					$l \in \NN$,
				\item
					$(x, u) = \big((x_i)_{i \in \NN_{< l}}, (u_i)_{i \in \NN_{< l}}\big), (y, v) = \big((y_i)_{i \in \NN_{< l}}, (v_i)_{i \in \NN_{< l}}\big) \in \fpi[l]$,
				\item
					sequences $w, z\colon \NN \to \one + \U$ given by
					$$w_n \dfeq
						\begin{cases}
							\unit & \text{ if } s_n \in \one,\\
							\ext\Big(\big(u_i, d_X(x_i, s_n)\big)_{i \in \NN_{< l}} \cnct \big(w_k, d_\mtr{X}(s_k, s_n)\big)_{k \in \NN_{< n} \cap s^{-1}(X)}\Big) & \text{ if } s_n \in X,
						\end{cases}
					$$
					$$z_n \dfeq
						\begin{cases}
							\unit & \text{ if } s_n \in \one,\\
							\ext\Big(\big(v_i, d_X(y_i, s_n)\big)_{i \in \NN_{< l}} \cnct \big(z_k, d_\mtr{X}(s_k, s_n)\big)_{k \in \NN_{< n} \cap s^{-1}(X)}\Big) & \text{ if } s_n \in X,
						\end{cases}
					$$
					\ie for $s_n \in X$ we have $w_n = \extm[l]\big((x_i)_{i \in \NN_{< l}}, (u_i)_{i \in \NN_{< l}}\big)(s_n)$ and $z_n = \extm[l]\big((y_i)_{i \in \NN_{< l}}, (v_i)_{i \in \NN_{< l}}\big)(s_n)$,
				\item
					$\epsilon \in \RR_{> 0}$,
				\item
					$M \in \NN$ large enough such that
					$$\xall{i}{\NN_{< l}}\xsome{k}{\NN_{\leq M} \cap s^{-1}(X)}{d_\mtr{X}(x_i, s_k) \leq \epsilon}$$
					and
					$$\xall{i}{\NN_{< l}}\xsome{k}{\NN_{\leq M} \cap s^{-1}(X)}{d_\mtr{X}(y_i, s_k) \leq \epsilon},$$
				\item
					$B \dfeq \sup\st{d(w_k, z_k)}{k \in \NN_{\leq M} \cap s^{-1}(X)}$.
			\end{itemize}
			Then
			\begin{enumerate}
				\item\label{Lemma: estimation_of_supremum_metric_on_isometries: finite_approximation}
					$d(w_m, z_n) \dis d_\mtr{X}(s_m, s_n) \leq B + 2\epsilon$,
				\item\label{Lemma: estimation_of_supremum_metric_on_isometries: dP_bound}
					$d(w_m, z_n) \dis d_\mtr{X}(s_m, s_n) \leq d_P\big((x, u), (y, v)\big)$
			\end{enumerate}
			for all $m, n \in s^{-1}(X)$.
		\end{lemma}
		\begin{proof}
			We prove both items by induction on $m+n$ (where $m, n \in s^{-1}(X)$), taking in account both times that
			\begin{align*}
				d(w_m, z_n) = \sup\Big(&\st{d(u_i, z_n) \dis d_\mtr{X}(x_i, s_m)}{i \in \NN_{< l}} \cup\\
				&\st{d(w_k, z_n) \dis d_\mtr{X}(s_k, s_m)}{k \in \NN_{< m} \cap s^{-1}(X)} \cup\\
				&\st{d(w_m, v_i) \dis d_\mtr{X}(y_i, s_n)}{i \in \NN_{< l}} \cup\\
				&\st{d(w_m, z_k) \dis d_\mtr{X}(s_k, s_n)}{k \in \NN_{< n} \cap s^{-1}(X)}
			\Big)
			\end{align*}
			by Lemma~\ref{Lemma: distances_in_completed_Urysohn}.
			\begin{enumerate}
				\item
					First we wish to prove $d_\mtr{X}(s_m, s_n) \leq B + 2\epsilon + d(w_m, z_n)$. This is obviously true if $m = n$. Suppose $m \neq n$; let $m < n$ (the case $m > n$ is proved analogously). Then $d(w_m, z_m) \dis d_\mtr{X}(s_m, s_n)$ appears in the last line of the above supremum, so
					$$d_\mtr{X}(s_m, s_n) \leq d(w_m, z_m) + d(w_m, z_m) \dis d_\mtr{X}(s_m, s_n) \leq$$
					$$\leq d(w_m, z_m) + d(w_m, z_n) = d(w_m, z_m) \dis d_\mtr{X}(s_m, s_m) + d(w_m, z_n) \leq B + 2\epsilon + d(w_m, z_n),$$
					the last inequality holding by the induction hypothesis since $m + m < m + n$.
					
					Second, we prove $d(w_m, z_n) \leq B + 2\epsilon + d_\mtr{X}(s_m, s_n)$ by showing that each individual line in the above supremum is at most the right-hand side. The second line is easy:
					$$d(w_k, z_n) \dis d_\mtr{X}(s_k, s_m) \leq d(w_k, z_n) \dis d_\mtr{X}(s_k, s_n) + d_\mtr{X}(s_m, s_n) \leq B + 2\epsilon + d_\mtr{X}(s_m, s_n),$$
					the last inequality holding by the induction hypothesis since $k + n < m + n$. As for the first line, for $i \in \NN_{< l}$ let $a \in \NN_{< B} \cap s^{-1}(X)$ be such that $d_\mtr{X}(x_i, s_a) \leq \epsilon$. Then
					$$d(u_i, z_n) \dis d_\mtr{X}(x_i, s_m) \leq d(w_a, z_n) \dis d_\mtr{X}(s_a, s_m) + d(u_i, w_a) + d_\mtr{X}(x_i, s_a) \leq$$
					$$\leq d(w_a, z_n) \dis d_\mtr{X}(s_a, s_m) + 2\epsilon \leq$$
					$$\leq d(z_a, z_n) \dis d_\mtr{X}(s_a, s_n) + d(w_a, z_a) + d_\mtr{X}(s_m, s_n) + 2\epsilon \leq B + 2\epsilon + d_\mtr{X}(s_m, s_n)$$
					since $d(u_i, w_a) = d_\mtr{X}(x_i, s_a) \leq \epsilon$ and $d(z_a, z_n) = d_\mtr{X}(s_a, s_n)$ because $\extm[l]$ maps into isometries. The statements $d(w_m, v_i) \dis d_\mtr{X}(y_i, s_n) \leq B + 2\epsilon + d_\mtr{X}(s_m, s_n)$ and $d(w_m, z_k) \dis d_\mtr{X}(s_k, s_n) \leq B + 2\epsilon + d_\mtr{X}(s_m, s_n)$ are proved analogously.
				\item
					Recall that
					$$d_P\big((x, u), (y, v)\big) = \sup\st{d_\mtr{X}(x_i, y_i)}{i \in \NN_{< l}} + \sup\st{d(u_i, v_i)}{i \in \NN_{< l}}.$$
					First we prove $d(w_m, z_n) \leq d_X(s_m, s_n) + d_P(x, y)$.
					$$d(u_i, z_n) \dis d_\mtr{X}(x_i, s_m) \leq d(z_n, v_i) \dis d_\mtr{X}(y_i, s_m) + d(u_i, v_i) + d_\mtr{X}(x_i, y_i) =$$
					$$= d_\mtr{X}(y_i, s_n) \dis d_\mtr{X}(y_i, s_m) + d(u_i, v_i) + d_\mtr{X}(x_i, y_i) \leq d_\mtr{X}(s_n, s_m) + d(u_i, v_i) + d_\mtr{X}(x_i, y_i)$$
					
					$$d(w_m, v_i) \dis d_\mtr{X}(y_i, s_n) \leq d(w_m, u_i) \dis d_\mtr{X}(x_i, s_n) + d(u_i, v_i) + d_\mtr{X}(x_i, y_i) =$$
					$$= d_\mtr{X}(x_i, s_m) \dis d_\mtr{X}(x_i, s_n) + d(u_i, v_i) + d_\mtr{X}(x_i, y_i) \leq d_\mtr{X}(s_m, s_n) + d(u_i, v_i) + d_\mtr{X}(x_i, y_i)$$
					For $d(w_k, z_n) \dis d_\mtr{X}(s_k, s_m)$ and $d(w_m, z_k) \dis d_\mtr{X}(s_k, s_n)$ use the induction hypothesis.
					
					Second, we wish to prove $d_\mtr{X}(s_m, s_n) \leq d(w_m, z_n) + d_P(x, y)$. Obviously this holds for $m = n$. Assume $m > n$ (the case $m < n$ is proved analogously); then
					$$d_\mtr{X}(s_m, s_n) = d(w_m, w_n) \leq d(w_m, z_n) + d(w_n, z_n) \leq d(w_m, z_n) + d_P(x, y)$$
					where the last inequality holds by the induction hypothesis since $n + n < m + n$.
			\end{enumerate}
		\end{proof}
		
		\begin{theorem}\label{Theorem: extension_nonexpansive}
			Let $\mtr{X} = (X, d_\mtr{X}, s)$ be a separable metric space.
			\begin{enumerate}
				\item
					For all $\big((x_i)_{i \in \NN_{< l}}, (u_i)_{i \in \NN_{< l}}\big), \big((y_i)_{i \in \NN_{< l}}, (v_i)_{i \in \NN_{< l}}\big) \in \fpi[l]$ the supremum
					$$\sup\st[2]{d\Big(\extm\big((x_i)_{i \in \NN_{< l}}, (u_i)_{i \in \NN_{< l}}\big)(x), \extm\big((y_i)_{i \in \NN_{< l}}, (v_i)_{i \in \NN_{< l}}\big)(x)\Big)}{x \in X}$$
					is a real number.
				\item
					The map $\extm[n]\colon \fpi[n] \to \isom$ is continuous, in fact non-expansive in the sense of Lemma~\ref{Lemma: sup_metric_induces_uniform_convergence}.
			\end{enumerate}
		\end{theorem}
		\begin{proof}
			\begin{enumerate}
				\item
					To simplify notation, shorten
					$$f(x) \dfeq d\Big(\extm\big((x_i)_{i \in \NN_{< l}}, (u_i)_{i \in \NN_{< l}}\big)(x), \extm\big((y_i)_{i \in \NN_{< l}}, (v_i)_{i \in \NN_{< l}}\big)(x)\Big)$$
					for $x \in X$. We prove that $\sup\st{f(x)}{x \in X}$ is a real number using Postulate~\ref{Postulate: reals_Dedekind_complete}. Define
					$$L \dfeq \st{\sup\st{f(s_k)}{k \in \NN_{\leq n} \cap s^{-1}(X)}}{n \in \NN},$$
					$$U \dfeq \st{a \in \RR}{\xall{n}{s^{-1}(X)}{a \geq f(s_n)}}.$$
					Take any $\epsilon \in \RR_{> 0}$. There exists a large enough $M \in \NN$ such that
					$$\xall{i}{\NN_{< l}}\xsome{k}{\NN_{\leq M} \cap s^{-1}(X)}{d_\mtr{X}(x_i, s_k) \leq \epsilon}$$
					and
					$$\xall{i}{\NN_{< l}}\xsome{k}{\NN_{\leq M} \cap s^{-1}(X)}{d_\mtr{X}(y_i, s_k) \leq \epsilon}.$$
					Let $B \dfeq \sup\st{f(s_k)}{k \in \NN_{\leq M} \cap s^{-1}(X)}$. Notice that this matches the definition of $B$ in Lemma~\ref{Lemma: estimation_of_supremum_metric_on_isometries} which tells us (by taking $m = n$ in its statement) that $f(s_n) \leq B + 2\epsilon$ for all $n \in s^{-1}(X)$. Clearly then $B \in L$ and $B + 2\epsilon \in U$, and it is obvious that $a \leq b$ for every $a \in L$ and $b \in U$. By Postulate~\ref{Postulate: reals_Dedekind_complete} $L$ and $U$ determine the real number equal to $\sup{L} = \sup\st{f(s_k)}{k \in s^{-1}(X)}$, but that equals $\sup\st{f(x)}{x \in X}$ by Corollary~\ref{Corollary: comparison_on_dense_subsets}(\ref{Corollary: comparison_on_dense_subsets: suprema}).
				\item
					Use the previous item with Lemma~\ref{Lemma: sup_metric_induces_uniform_convergence}, together with Lemma~\ref{Lemma: estimation_of_supremum_metric_on_isometries}(\ref{Lemma: estimation_of_supremum_metric_on_isometries: dP_bound}) for $m = n$.
			\end{enumerate}
		\end{proof}
		
		It is easy to extend the continuity result from finite lists of fixed length to all finite lists; let
		$$\fpi \dfeq \coprod_{l \in \NN} \fpi[n]$$
		be the topological coproduct (the disjoint union, with every summand open) of individual lists (we can realize it as a subset of $\finseq{X} \times \finseq{\U}$). The maps $\extm[n]$ induce the map $\extm\colon \fpi \to \isom$ which is continuous by the definition of a coproduct (it is defined by its continuous restrictions on members of an open cover).
		
		Finally, define two lists in $\fpi$ to be equivalent when they have the same image in $X$, on which they determine the same finite partial isometry. Here is the explicit definition. Let $\big((x_i)_{i \in \NN_{< l}}, (u_i)_{i \in \NN_{< l}}\big), \big((y_j)_{j \in \NN_{< k}}, (v_j)_{j \in \NN_{< k}}\big) \in \fpi$. Then
		$$\big((x_i)_{i \in \NN_{< l}}, (u_i)_{i \in \NN_{< l}}\big) \equ \big((y_j)_{j \in \NN_{< k}}, (v_j)_{j \in \NN_{< k}}\big) \dfeq$$
		$$\st{x_i}{i \in \NN_{< l}} = \st{y_j}{j \in \NN_{< k}} \land$$
		$$\land \xall{i}{\NN_{< l}}\all{j}{\NN_{< k}}{x_i = y_j \implies u_i = v_j}.$$
		Since we identify lists which represent the same isometries, the quotient set $\fpi/_\equ$ can genuinly be called a ``set of finite partial isometries from $\mtr{X}$ to $\U$''. Of course, it is not just a set, but a topological space, equipped with the quotient topology.
		
		Note that the result of $\extm$ does not depend on the actual lists, just on which elements appear on the list and where they get mapped (since we are calculating suprema of sets which depend only on this). Thus it induces the extension map on the quotient, $\widetilde{\extm}\colon \fpi/_\equ \to \isom$.
		
		\begin{theorem}
			The map $\widetilde{\extm}$, which maps a finite partial isometry into the Urysohn space to its total isometric extension, is continuous.
		\end{theorem}
		\begin{proof}
			A standard theorem from topology states, that a given a continuous map which respects an equivalence relation on its domain, the map it induces on the topological quotient of its domain is continuous.
		\end{proof}

	\section{Algebraic Structure of the Urysohn Space}\label{Section: algebra_of_Urysohn}
	
		In this section we equip the Urysohn space with the algebraic structure. Specifically, we show that it is a ``disring analogue of a Banach space over $\nnr$''.
		
		First define a map $\norm\colon \proto \to \od$ (call it a \df{norm}) for $a = \atuple{a}{\alpha}{i} \in \proto$ as $\norm[a] \dfeq d(a, \et)$; equivalently, it is inductively defined by
		$$\norm[a] \dfeq \sup\st{\norm[a_i] \dis \alpha_i}{i \in \NN_{< \lnth{a}}}.$$
		
		Second, define the operation $\dis\colon \proto \times \proto \to \proto$ for $a = \atuple{a}{\alpha}{i}, b = \atuple{b}{\beta}{j} \in \proto$ inductively on $\age{a} + \age{b}$ by
		$$a \dis b \dfeq \strat[\age{a} + \age{b}]{\big((a_i \dis b, \alpha_i)_{i \in \NN_{< \lnth{a}}} \cnct (a \dis b_j, \beta_j)_{j \in \NN_{< \lnth{b}}}\big)}{}.$$
		Clearly $\lnth{a \dis b} = \lnth{a} + \lnth{b}$. We declared $\age{a \dis b} = \age{a} + \age{b}$ which obviously works if $\age{a}, \age{b} \geq 1$, but it works also when one of $a, b$ is the empty tuple at age $0$ since it is easy to prove inductively $a \dis \et = \et \dis a = a$. Thus $\et$ is the neutral element for $\dis$.
		
		\begin{proposition}
			The operation $\dis$ is associative.
		\end{proposition}
		\begin{proof}
			Take $a = \atuple{a}{\alpha}{i}, b = \atuple{b}{\beta}{j}, c = \atuple{c}{\gamma}{k} \in \proto$. By induction on $\age{a} + \age{b} + \age{c}$
			$$(a \dis b) \dis c =$$
			{\small
			$$= \strat[(\age{a} + \age{b}) + \age{c}]{\big(((a_i \dis b) \dis c, \alpha_i)_{i \in \NN_{< \lnth{a}}} \cnct ((a \dis b_j) \dis c, \beta_j)_{j \in \NN_{< \lnth{b}}} \cnct ((a \dis b) \dis c_k, \gamma_k)_{k \in \NN_{< \lnth{c}}}\big)}{} =$$
			$$= \strat[\age{a} + (\age{b} + \age{c})]{\big((a_i \dis (b \dis c), \alpha_i)_{i \in \NN_{< \lnth{a}}} \cnct (a \dis (b_j \dis c), \beta_j)_{j \in \NN_{< \lnth{b}}} \cnct (a \dis (b \dis c_k), \gamma_k)_{k \in \NN_{< \lnth{c}}}\big)}{} =$$
			}
			$$= a \dis (b \dis c).$$
		\end{proof}
		
		We connect the two introduced operations.
		\begin{proposition}
			The following holds for all $a, b \in \proto$.
			\begin{enumerate}
				\item
					$\norm[a \dis b] = d(a, b)$
				\item
					$\norm[a] \dis \norm[b] \leq \norm[a \dis b]$ \quad (\df{triangle inequality})
			\end{enumerate}
		\end{proposition}
		\begin{proof}
			\begin{enumerate}
				\item
					By induction on $\age{a} + \age{b}$
					$$\norm[a \dis b] = \sup\big(\st{\norm[a_i \dis b] \dis \alpha_i}{i \in \NN_{< \lnth{a}}} \cup \st{\norm[a \dis b_j] \dis \beta_j}{j \in \NN_{< \lnth{b}}}\big) =$$
					$$= \sup\big(\st{d(a_i, b) \dis \alpha_i}{i \in \NN_{< \lnth{a}}} \cup \st{d(a, b_j) \dis \beta_j}{j \in \NN_{< \lnth{b}}}\big) = d(a, b).$$
				\item
					$\norm[a \dis b] = d(a, b) \geq d(a, \et) \dis d(b, \et) = \norm[a] \dis \norm[b]$
			\end{enumerate}
		\end{proof}
		
		\begin{lemma}\label{Lemma: scramble_terms_in_dis}
         Let $a, b, c, d \in \proto$. Then $d(a \dis b, c \dis d) = d(a \dis c, b \dis d)$.\footnote{The geometric meaning of this lemma is essentially that $\dis$ is a non-expansive map. See Proposition~\ref{Proposition: Urysohn_dis_nonexpansive} below.}
		\end{lemma}
		\begin{proof}
			Let $a = \atuple{a}{\alpha}{i},$ $b = \atuple{b}{\beta}{j},$ $c = \atuple{c}{\gamma}{k},$ $d = \atuple{d}{\delta}{l}$. We prove the statement by induction on $\age{a} + \age{b} + \age{c} + \age{d}$.
			$$d(a \dis b, c \dis d) =$$
			$$= \sup\Big(%
				\st{d(a_i \dis b, c \dis d) \dis \alpha_i}{i \in \NN_{< \lnth{a}}} \cup
				\st{d(a \dis b_j, c \dis d) \dis \beta_j}{j \in \NN_{< \lnth{b}}} \cup$$
				$$\cup
				\st{d(a \dis b, c_k \dis d) \dis \gamma_k}{k \in \NN_{< \lnth{c}}} \cup
				\st{d(a \dis b, c \dis d_l) \dis \delta_l}{l \in \NN_{< \lnth{d}}}%
			\Big) =$$
			$$= \sup\Big(%
				\st{d(a_i \dis c, b \dis d) \dis \alpha_i}{i \in \NN_{< \lnth{a}}} \cup
				\st{d(a \dis c, b_j \dis d) \dis \beta_j}{j \in \NN_{< \lnth{b}}} \cup$$
				$$\cup
				\st{d(a \dis c_k, b \dis d) \dis \gamma_k}{k \in \NN_{< \lnth{c}}} \cup
				\st{d(a \dis c, b \dis d_l) \dis \delta_l}{l \in \NN_{< \lnth{d}}}%
			\Big) =$$
			$$= d(a \dis c, b \dis d)$$
		\end{proof}
		
		\begin{corollary}\label{Corollary: dis_independent_of_representatives}
			Let $a, b, a', b' \in \proto$, and suppose $d(a, a') = d(b, b') = 0$. Then $d(a \dis b, a' \dis b') = 0$.
		\end{corollary}
		\begin{proof}
			By the previous lemma
			$$d(a \dis b, a' \dis b') = d(a \dis a', b \dis b') \leq \norm[a \dis a'] + \norm[b \dis b'] = d(a, a') + d(b, b') = 0.$$
		\end{proof}
		
		Now we consider the additional properties of $\dis$ when we restrict it to $\pseudo$.
		\begin{proposition}
			The following holds for all $a, b, x \in \pseudo$.
			\begin{enumerate}
				\item
					$\norm[a \dis a] = 0$
				\item
					$d(a \dis b, b \dis a) = 0$
				\item
					$d(a \dis x, b \dis x) = d(a, b) = d(x \dis a, x \dis b)$
				\item
					$a \dis b \in \pseudo$
			\end{enumerate}
		\end{proposition}
		\begin{proof}
			\begin{enumerate}
				\item
					$\norm[a \dis a] = d(a, a) = 0$
				\item
					Playing with Lemma~\ref{Lemma: scramble_terms_in_dis} we obtain
					$$d(a \dis b, b \dis a) = d(a \dis b \dis \et, b \dis a) = d(a \dis b \dis b, \et \dis a) =$$
					$$= d(a \dis b \dis b, a) = d(a \dis b \dis b, a \dis \et) = d(a \dis a, b \dis b \dis \et) =$$
					$$= d(a \dis a, b \dis b) \leq \norm[a \dis a] + \norm[b \dis b] = 0.$$
				\item
					Using Lemma~\ref{Lemma: scramble_terms_in_dis} we calculate
					$$d(a \dis x, b \dis x) = d(a \dis b, x \dis x) \leq \norm[a \dis b] + \norm[x \dis x] = d(a, b)$$
					whence also (by Corollary~\ref{Corollary: dis_independent_of_representatives} and associativity of $\dis$)
					$$d(a, b) = d(a \dis x \dis x, b \dis x \dis x) \leq d(a \dis x, b \dis x).$$
					Similarly for the other equality.
				\item
					Permissibility of $a \dis b$ is verified by the following calculations.
					$$d(a_i \dis b, a_j \dis b) \dis \alpha_i = d(a_i, a_j) \dis \alpha_i \leq \alpha_j$$
					$$d(a_i \dis b, a \dis b_j) = d(a_i \dis a, b \dis b_j) \leq \norm[a_i \dis a] + \norm[b \dis b_j] = d(a_i, a) + d(b, b_i) = \alpha_i + \beta_j$$
					$$\alpha_i \leq \alpha_i \dis \beta_j + \beta_j = d(a_i, a) \dis d(b, b_j) + \beta_j = \norm[a_i \dis a] \dis \norm[b, b_j] + \beta_j \leq d(a_i \dis a, b \dis b_j) + \beta_j$$
					Inequalities $d(a \dis b_i, a \dis b_j) \dis \beta_i \leq \beta_j$ and $d(a \dis b_j, a_i \dis b) \dis \beta_j \leq \alpha_i$ are verified analogously.
			\end{enumerate}
		\end{proof}
		
		Observe that by Corollary~\ref{Corollary: dis_independent_of_representatives} $\dis$ on $\pseudo$ induces $\dis$ on its Kolmogorov quotient $\Ury$.
		
		\begin{theorem}\label{Theorem: noncompleted_Urysohn_associative_disgroup}
			$\big(\Ury, \dis, [\et]\big)$ is an associative disgroup.
		\end{theorem}
		\begin{proof}
			By the discussion above $\dis$ is an associative commutative operation with neutral element $[\et]$ and it satisfies $x \dis x = [\et]$ for all $x \in \Ury$. Thus $\big(\Ury, \dis, [\et]\big)$ is an associative disgroup by Proposition~\ref{Proposition: order_two_group_is_a_disgroup}.
		\end{proof}
		
		Assume now that that $\od$ is not just a disgroup, but a disring, \ie that it also has the multiplicative structure. This enables us to make $\Ury$ a ``module'' over $\od$. For $\lambda \in \od$ and $a = \atuple{a}{\alpha}{i} \in \proto$ define inductively on $\age{a}$ the ``scalar multiplication''
		$$\lambda \cdot a \dfeq \strat[\age{a}]{(\lambda \cdot a_i, \lambda \alpha_i)}{i \in \NN_{< \lnth{a}}}.$$
		
		\begin{proposition}
			The following holds for $\lambda \in \od$, $a, b \in \proto$.
			\begin{enumerate}
				\item $1 \cdot a = a$
				\item $\norm[0 \cdot a] = 0$
				\item $\lambda \cdot (a \dis b) = (\lambda \cdot a) \dis (\lambda \cdot b)$
			\end{enumerate}
		\end{proposition}
		\begin{proof}
			Simple induction.
		\end{proof}
		
		\begin{definition}
			A structure $(X, +\colon X \times X \to X, \dis\colon X \times X \to X, 0, \cdot\colon \od \times X \to X$ is a \df{module} over a disring $\od$ when $(X, +, \dis, 0)$ is a disgroup and the following holds for all $x, y \in X$, $\lambda, \mu \in \od$:
			\begin{itemize}
				\item
					$\mu \cdot (\lambda \cdot x) = (\lambda \mu) \cdot x$,
				\item
					$\lambda \cdot (x \dis y) = (\lambda \cdot x) \dis (\lambda \cdot y)$,
				\item
					$1 \cdot x = x$,
				\item
					$0 \cdot x = 0$.
			\end{itemize}
			Furthermore, if the disring $\od$ is ordered, we define that the module $(X, +, \dis, 0, \cdot, \norm)$ is \df{normed} when the operation $\norm\colon X \to \od$ has the properties
			\begin{itemize}
				\item
					$\norm[x] = 0 \implies x = 0$,
				\item
					$\norm[\lambda \cdot x] = \lambda \norm[x]$ (in particular $\norm[0] = 0$),
				\item
					$\norm[a] \dis \norm[b] \leq \norm[a \dis b]$.
			\end{itemize}
		\end{definition}
		
		\begin{theorem}
			$\Ury$ is a normed module over the disring $\od$ (by taking addition to be equal to $\dis$, and the zero element to $[\et]$).
		\end{theorem}
		\begin{proof}
			By the discussion above.
		\end{proof}
		
		Our definition of a module is very reminiscent to the usual one over rings; basically we just replace $+$ with $\dis$. The difference is in the last axiom though which is usually stated $(\lambda + \mu) \cdot x = \lambda \cdot x + \mu \cdot x$ which in our context would read $(\lambda \dis \mu) \cdot x = (\lambda \cdot x) \dis (\mu \cdot x)$. This however does not hold which has to do with the fact that $\dis$ is associative on $\Ury$, but not on $\od$. For example, taking $\od = \nnr$, we have
		$$((2 \dis 1) \dis 1) \cdot x = 0 \cdot x, \qquad (2 \dis (1 \dis 1)) \cdot x = 2 \cdot x$$
		while
		$$(2 \cdot x) \dis (1 \cdot x) \dis (1 \cdot x) = 2 \cdot x$$
		regardless of how we associate $\dis$. We therefore do not require this version of distributivity in the definition of a module, but replace it with the weaker condition $0 \cdot x = 0$.\footnote{Compare this with the theory of semirings where it is likewise explicitly required that $0$ annihilates all elements.}
		
		We wish to extend the module operations to the completion of $\Ury$, that is, to $\U$. First we prepare a lemma that ensures that extensions of operations still satisfy the required algebraic conditions.
		\begin{lemma}\label{Lemma: extension_shares_algebraic_properties}
			Let $\mtr{X} = (X, d_\mtr{X})$ be a metric space and $A \subseteq X$ a dense subspace of $\mtr{X}$. Suppose $X$ has operations which satisfy the equation
			$$f(x_0, x_1, \ldots, x_{n-1}) = g(x_0, x_1, \ldots, x_{n-1})$$
			for all $x_0, x_1, \ldots, x_{n-1} \in A$ where $f, g\colon X^n \to X$ are continuous maps and $n \in \NN$. Then this equation is satisfied for all $x_0, x_1, \ldots, x_{n-1} \in X$.
		\end{lemma}
		\begin{proof}
			If $A$ is dense in $X$, then $A^n$ is dense in $X^n$ (whatever product metric we choose). Now use Lemma~\ref{Lemma: uniqueness_of_continuous_extensions}.
		\end{proof}
		
		Next, we show that the operations on $\Ury$ are continuous, and that they extend to $\U$.
		
		\begin{proposition}\label{Proposition: Urysohn_dis_nonexpansive}
			Under product $1$-metric, the operation $\dis\colon \Ury \times \Ury \to \Ury$ is a non-expansive map and thus extends to a non-expansive map $\dis\colon \U \times \U \to \U$, making $(\U, \dis, [\et])$ an associative disgroup.
		\end{proposition}
		\begin{proof}
			For $(a, b), (c, d) \in \Ury \times \Ury$ we have
			$$d(a \dis b, c \dis d) = \norm[a \dis b \dis c \dis d] \leq \norm[a \dis c] + \norm[b \dis d] = d(a, c) + d(b, d).$$
			Thus $\dis$ extends to a non-expansive operation on $\U$ by Proposition~\ref{Proposition: extension_of_maps_into_complete_space}. Since all conditions for a group of order (at most) $2$ are given as equations, $(\U, \dis, [\et])$ is an associative disgroup by Theorem~\ref{Theorem: noncompleted_Urysohn_associative_disgroup}, Lemma~\ref{Lemma: extension_shares_algebraic_properties} and Proposition~\ref{Proposition: order_two_group_is_a_disgroup}.
		\end{proof}
		
		The extension of the scalar multiplication is trickier, though. First of all, it is not non-expansive. This is actually usual for products, thought they are normally still Lipschitz on bounded subsets (and thus in turn area Lipschitz which would enable us to use Proposition~\ref{Proposition: extension_of_maps_into_complete_space}). This follows from distributivity; here is a model calculation for $a, b, c, d \in \RR$:
		$$|a \cdot b - c \cdot d| = |a \cdot b - c \cdot b + c \cdot b - c \cdot d| \leq$$
		$$\leq |a \cdot b - c \cdot b| + |c \cdot b - c \cdot d| = |(a - c) \cdot b| + |c \cdot (b - d)| \leq$$
		$$\leq \sup\{|b|, |c|\} \cdot \big(|a - c| + |b - d|\big).$$
		However, this method does not work in our case, as the scalar multiplication on $\Ury$ is distributive only in one factor. Here is a trick how to get around it.
		\begin{lemma}
			For every $\lambda \in \od$ the unary operation $m_\od^\lambda\colon \Ury \to \Ury$, given by
			$$m_\od^\lambda(x) \dfeq \lambda \cdot x,$$
			is a Lipschitz map with the coefficient $\lambda$, and thus extends to the Lipschitz map $m^\lambda\colon \U \to \U$ (with the same coefficient).\footnote{To make a very fine point, in our definition of a Lipschitz map (Definition~\ref{Definition: more_maps}), we required the Lipschitz coefficient to be positive, so we would actually have to take something like $\sup\{\lambda, 1\}$ for it, but it makes absolutely no difference.}
		\end{lemma}
		\begin{proof}
			We have
			$$d(\lambda \cdot x, \lambda \cdot y) = \norm[(\lambda \cdot x) \dis (\lambda \cdot y)] = \norm[\lambda \cdot (x \dis y)] = \lambda \norm[x \dis y] = \lambda d(x, y).$$
			As usual, we get the extension from Proposition~\ref{Proposition: extension_of_maps_into_complete_space}.
		\end{proof}
		
		Since we can take $\od = \nnr$, we can actually define the scalar multiplication on $\nnr \times \U$ by $\lambda \cdot x \dfeq m^\lambda(x)$.
		
		\begin{proposition}\label{Proposition: scalar_multiplication}
			The map $\cdot\colon \nnr \times \U \to \U$ is continuous and thus the unique continuous extension of $\cdot\colon \od \times \Ury \to \Ury$ (for any $\od$). Moreover, it satisfies the conditions needed to make $\U$ into a module over $\nnr$.
		\end{proposition}
		\begin{proof}
			Continuity of $\cdot$ is the one proof which I do not know how to do classically, which thus remains as a challenge for the readers. Here is how it can be done with ``heavy artillery''.
			
			Since the entire development so far has been done fully constructively (including without using countable choice), we can interpret it in any topos~\cite{Johnstone_PT_2002:_sketches_of_an_elephant_a_topos_theory_compendium} with natural numbers object, in particularly in such models of synthetic topology (see Remark~\ref{Remark: synthetic_topology} in the next section). In these models, any map that we can construct is automatically continuous~\cite{Escardo_M_2004:_notes_on_synthetic_topology, Escardo_M_2004:_synthetic_topology_of_data_types_and_classical_spaces, Lesnik_D_2010:_synthetic_topology_and_constructive_metric_spaces}, and this property transfers to classical mathematics. More precisely, we may choose to interpret this theory in the gros topos over separable metric spaces and continuous maps between them, in which the real numbers and the Urysohn space are representable by their classical counterparts (see~\cite[Section~5.4]{Lesnik_D_2010:_synthetic_topology_and_constructive_metric_spaces}), thus (since separable metric spaces fully and faithfully embed into the gros topos) the scalar multiplication is representable by a continuous map.
			
			Continuity of $\cdot$ is sufficient for its uniqueness (Lemma~\ref{Lemma: uniqueness_of_continuous_extensions}) and algebraic properties (Lemma~\ref{Lemma: extension_shares_algebraic_properties}).
		\end{proof}
		
		The last thing to extend to $\U$ is the norm which is easy enough: just take $\norm[x] = d(x, [\et])$ (of course, being non-expansive, it could also be extended via Proposition~\ref{Proposition: extension_of_maps_into_complete_space}). Clearly the required properties are still satisfied.
		
		\begin{theorem}
			$(\U, \dis, [\et], \cdot, \norm)$ is a complete normed module over $\nnr$.
		\end{theorem}
		\begin{proof}
			By the discussion above.
		\end{proof}

	\section{Applications}\label{Section: applications}
		
		We present here a few simple applications of continuity and algebra results of previous sections.
		
		Let $\aut$ denote the set of automorphisms of $\U$, \ie isometric isomorphisms $\U \to \U$.	This is a group for composition $\circ$.	
		\begin{proposition}
			The transposition of $\dis$ provides an isometric embedding $\U \to \aut$ in the sense of Lemma~\ref{Lemma: sup_metric_induces_uniform_convergence} which is moreover a group homomorphism.
		\end{proposition}
		\begin{proof}
			For all $x, y, z \in \U$ we have
			$$d(x \dis z, y \dis z) = \norm[((x \dis z) \dis (y \dis z)] = \norm[x \dis y \dis z \dis z] = \norm[x \dis y] = d(x, y),$$
			so the map $i(z) \dfeq \insarg \dis z$ indeed maps as $i\colon \U \to \aut$. We claim that its image is metrizable by the $\sup$ metric, and that $i$ is an isometry in this sense. Since
			$$d(i(z)(x), i(w)(x)) = d(z \dis x, w \dis x) = d(z, w)$$
			for all $z, w, x \in \U$, not only do we see that the required supremum exists, it is in fact the supremum of a singleton set $\{d(w, z)\}$.
			
			The group homomorphism condition $i(z \dis w) = i(z) \circ i(w)$ is the associativity of $\dis$.
		\end{proof}
		
		\begin{proposition}\label{Proposition: Urysohn_strongly_homogeneous}
			The Urysohn space $\U$ is homogeneous in the following strong sense:
			\begin{itemize}
				\item
					there is a non-expansive (hence continuous) group homomorphism $\U \times \U \to \aut$ which maps points $a, b \in \U$ to an automorphism of $\U$ which swaps them.
			\end{itemize}
			Here `non-expansive' is meant in the sense of Lemma~\ref{Lemma: sup_metric_induces_uniform_convergence}, and of $\U \times \U$ being equipped with the product $1$-metric.
		\end{proposition}
		\begin{proof}
			By the previous proposition mapping $(a, b) \in \U \times \U$ to $x \mapsto x \dis a \dis b$ works. It is non-expansive since
			$$d(x \dis a \dis b, x \dis a' \dis b') = \norm[x \dis a \dis b \dis x \dis a' \dis b'] =$$
			$$= \norm[a \dis a' \dis b \dis b'] \leq \norm[a \dis a'] + \norm[b \dis b'] = d(a, a') + d(b, b').$$
			Also, this is a group homomorphism since $x \dis (a \dis a') \dis (b \dis b') = (x \dis a \dis b) \dis a' \dis b'$.
		\end{proof}
		
		\begin{proposition}
			The Urysohn space $\U$ is contractible (in particular, path-connected).
		\end{proposition}
		\begin{proof}
			We construct a contraction from $\U$ to an arbitrary point $z \in \U$ in two ways: once using results from Section~\ref{Section: continuity}, and once from Section~\ref{Section: algebra_of_Urysohn}.
			
			First, let $H\colon \intcc{0}{1} \times \U \to \U$ be given as
			$$H(t, x) \dfeq \ext\Big(\big(x, t d(x, z)\big), \big(z, (1-t) d(x, z)\big)\Big).$$
			Since the metric $d$ and the map $\ext$ are continuous (Proposition~\ref{Proposition: continuity_of_one_point_extension}), so is $H$. We have $d(x, H(0, x)) = 0$ and $d(z, H(1, x)) = 0$ by Theorem~\ref{Theorem: Urysohn_space}(\ref{Theorem: Urysohn_space: extension_map}), so $H(0, x) = x$ and $H(1, x) = z$.
			
			For the second, let $G\colon \intcc{0}{1} \times \U \to \U$ be
			$$G(t, x) \dfeq (1 - t) \cdot x \dis t \cdot z.$$
			Clearly $G(0, x) = x$ and $G(1, x) = z$. Since the operations are continuous (Propositions~\ref{Proposition: Urysohn_dis_nonexpansive} and~\ref{Proposition: scalar_multiplication}), so is $G$.
		\end{proof}
		
		Because the operations in $\U$ are continuous, they induce the $\nnr$-normed module structure on the set of continuous maps $\C(X, \U)$ (by defining operations pointwise) for any topological space $X$. This might be useful to study such function sets, in particular hierarchies of the Urysohn space (see~\cite{Normann_D_2009:_a_rich_hierarchy_of_functionals_of_finite_types}).
		
		\medskip
		
		As proven by Bogaty\u{\i}~\cite{Bogatyi_SA_2000:_compact_homogeneity_of_urysohns_universal_metric_space:_russian, Bogatyi_SA_2000:_compact_homogeneity_of_urysohns_universal_metric_space}, we may extend partial isometries into $\U$ not only from finite, but more generally from compact subsets. We reprove this in our restricted setting.
		
		Classicaly, a metric space $\mtr{X} = (X, d_\mtr{X})$ is compact when any of the following equivalent conditions hold:
		\begin{itemize}
			\item
				every open cover of $X$ has a finite subcover,
			\item
				every sequence in $X$ has an accumulation point,
			\item
				every continuous map $X \to \RR$ is bounded,
			\item
				$\mtr{X}$ is a complete totally bounded metric space.
		\end{itemize}
		Constructively these conditions are not equivalent, so we need to pick the right one. In the context of metric spaces practice (not to mention Bogaty\u{\i}'s proof) shows~\cite{Troelstra_AS_Dalen_D_1988:_constructivism_in_mathematics_volume_1, Troelstra_AS_Dalen_D_1988:_constructivism_in_mathematics_volume_2}, that we want the last condition, and our case is no exception.
		
		Classically, and in at least some forms of constructivism, we say that a metric space is totally bounded when for every $\epsilon \in \RR_{> 0}$ it can be covered by finitely many balls of radius $\epsilon$. To make this work however, we need countable choice since otherwise we cannot even prove that a totally bounded metric space is separable. We therefore adjust the definition to our setting.
		\begin{definition}
			We say that $\mtr{X} = (X, d_\mtr{X}, s\colon \NN \to X + \one, a\colon \NN \to \NN)$ is a \df{totally bounded metric space} when $(X, d_\mtr{X})$ is a metric space and the condition
			$$\xall{n}{\NN}{\bigcup \st[2]{\ball[\mtr{X}]{s_i}{2^{-n}}}{i \in \NN_{< a_n} \cap s^{-1}(X)} = X}$$
			holds.
		\end{definition}
		It is easy to see that this definition is equivalent to the usual one in the presence of countable choice (in particular, in classical mathematics). Moreover, notice also that $s$ witnesses separability of $\mtr{X}$.
		
		One final observation before the proof. While classically a compact space is \emph{complete} totally bounded, it should be clear, that we do not require completness in our case, as we could always first extend the isometry to the completion of its original domain by Proposition~\ref{Proposition: extension_of_maps_into_complete_space}. Indeed, in the proof below (as well as Bogaty\u{\i}'s original proof) completeness of the original domain $A$ never comes up.
		
		\begin{theorem}
			Let
			\begin{itemize}
				\item
					$\mtr{X} = (X, d_\mtr{X}, s_\mtr{X})$ be a separable metric space,
				\item
					$A \subseteq X$,
				\item
					$d_\mtr{A}$ the restriction of $d_\mtr{X}$ to $X$,
				\item
					$\mtr{A} = (A, d_\mtr{A}, s_\mtr{A}, a)$ a totally bounded subspace of $\mtr{X}$, and
				\item
					$f\colon A \to \U$ an isometry.
			\end{itemize}
			Then there exists (a canonical choice of) an isometry $g\colon X \to \U$ which extends $f$.
		\end{theorem}
		\begin{proof}
			We can make several assumptions without loss of generality to simplify the technical part of the proof.
			\begin{itemize}
				\item
					Contrary to the case of separable metric spaces, for totally bounded ones it is decidable whether they are inhabited (consider whether there are any elements of $A$ in $\st{s_\mtr{A}(i)}{i \in \NN_{< a_0}}$). Thus we may consider $A = \emptyset$ a trivial special case (something what classically we might have done anyway) for which the theorem holds because $\emptyset$ is finite. In the remainder assume that $A$, and therefore $X$, is inhabited, so we may also assume that $s_\mtr{X}$ and $s_\mtr{A}$ are given as maps $\NN \to X$ and $\NN \to A$, respectively.
				\item
					The image od $s_\mtr{A}$ can be assumed to be contained in the image of $s_\mtr{X}$; we have a bijection $\NN \ism \NN + \NN$, and we can replace $s_\mtr{X}$ by $s_\mtr{X}$ on the first $\NN$ and $s_\mtr{A}$ on the second.
				\item
					Obviously the values of $a$ can be increased and the condition for total boundedness still holds; assume therefore that $a$ is an increasing sequence (\ie $a_n \leq a_{n+1}$ for all $n \in \NN$), and also that $a_n \geq n$ holds (we want $a$ to go to infinity which is not necessarily the case, as $A$ could be finite).
			\end{itemize}
			
			Take now an arbitrary $x \in X$ and define the sequence $b^x\colon \NN \to \U$ by
			$$b^x_n \dfeq \extm\big((s_\mtr{A}(i), f(s_\mtr{A}(i))_{i \in \NN_{< a(n)}}\big)(x).$$
			\begin{itemize}
				\item\proven{$\xall{n}{\NN}{d(b^x_n, b^x_{n+1}) \leq 2^{-n + 1}}$}
					For each of the finitely many $j \in \intco[\NN]{a_n}{a_{n+1}}$ choose $k_j \in \NN_{< a_n}$ such that $d_\mtr{A}(s_\mtr{A}(j), s_\mtr{A}(k_j)) \leq 2^{-n}$ (we can do that by total boundedness of $\mtr{A}$). Then
					$$\st{(s_\mtr{A}(i), f(s_\mtr{A}(i)))}{i \in \NN_{< a(n)}} =$$
					$$= \st{(s_\mtr{A}(i), f(s_\mtr{A}(i)))}{i \in \NN_{< a(n)}} \cup \st{(s_\mtr{A}(k_j), f(s_\mtr{A}(k_j)))}{j \in \intco[\NN]{a_n}{a_{n+1}}},$$
					and so
					$$\extm\big((s_\mtr{A}(i), f(s_\mtr{A}(i)))_{i \in \NN_{< a(n)}}\big)(x) =$$
					$$= \extm\big((s_\mtr{A}(i), f(s_\mtr{A}(i)))_{i \in \NN_{< a(n)}} \cnct (s_\mtr{A}(k_j), f(s_\mtr{A}(k_j)))_{j \in \intco[\NN]{a_n}{a_{n+1}}}\big)(x),$$
					thus by Lemma~\ref{Lemma: estimation_of_supremum_metric_on_isometries}(\ref{Lemma: estimation_of_supremum_metric_on_isometries: dP_bound})
					$$d(b^x_n, b^x_{n+1}) \leq$$
					$$\leq \sup\Big(\st{d_\mtr{X}(s_\mtr{A}(i), s_\mtr{A}(i))}{i \in \NN_{< a_n}} \cup \st{d_\mtr{X}(s_\mtr{A}(k_j), s_\mtr{A}(j))}{i \in \intco[\NN]{a_n}{a_{n+1}}}\Big) +$$
					$$+ \sup\Big(\st{d(f(s_\mtr{A}(i)), f(s_\mtr{A}(i)))}{i \in \NN_{< a_n}} \cup \st{d(f(s_\mtr{A}(k_j)), f(s_\mtr{A}(j)))}{i \in \intco[\NN]{a_n}{a_{n+1}}}\Big).$$
					We have
					$$d(f(s_\mtr{A}(i)), f(s_\mtr{A}(i))) = d_\mtr{X}(s_\mtr{A}(i), s_\mtr{A}(i)) = 0$$
					and
					$$d(f(s_\mtr{A}(k_j)), f(s_\mtr{A}(j))) = d_\mtr{X}(s_\mtr{A}(k_j), s_\mtr{A}(j)) \leq 2^{-n},$$
					so $d(b^x_n, b^x_{n+1}) \leq 2^{-n} + 2^{-n} = 2^{-n +1}$.
			\end{itemize}
			We see that $b^x$ is a Cauchy sequence (even rapid Cauchy if we drop the first term), and so has a limit in (complete, therefore Cauchy complete by Proposition~\ref{Proposition: complete_implies_Cauchy_complete}) $\U$. Define $g(x) \dfeq \lim b^x$.
			\begin{itemize}
				\item\proven{$\xall{x, y}{X}{d(g(x), g(y)) = d_\mtr{X}(x, y)}$}
					Consider the sequence $n \mapsto d(b^x(n), b^y(n))$. Since $\extm\big((s_\mtr{A}(i), f(s_\mtr{A}(i)))_{i \in \NN_{< a(n)}}\big)$ is an isometry, this sequence must be constant with value $d_\mtr{X}(x, y)$ which is then also its limit. We can swap the limit and $d$ (every metric is continuous), thus obtaining the result.
				\item\proven{$\rstr{g}_A = f$}
					Note that for every $n$ the sequence $b^{s_\mtr{A}(n)}$ is eventually constant (since $a$ goes to infinity), its terms equal to $f(s_\mtr{A}(n))$ which then must be the limit of the sequence as well. Thus $f$ and $g$ match on the image of $s_\mtr{A}$, and therefore, being isometries, on the entire $A$ by Lemma~\ref{Lemma: uniqueness_of_continuous_extensions}.
			\end{itemize}
		\end{proof}

	\section{Concluding Remarks}\label{Section: concluding_remarks}
	
		We end the paper with some remarks and questions. Shorter remarks are given directly below, while longer ones with some propositions to prove have their separate subsections.
		
		\begin{remark}
			Discussion of algebraic structure of the Urysohn space puts into our minds the Uspenskij's result~\cite{Uspenskij2004145} that the Urysohn space is homeomorphic to the Hilbert space $\ell^2$. The proof for this is classical, and it is unclear whether it holds constructively. In any case, I feel that the vector space structure is not the one we should be looking for on the Urysohn space. Suppose we had one, such that the norm would satisfy $\norm[-x] = \norm[x]$; if $x$ was the result of extending the isometry $\{0\} \hookrightarrow \U$, then $-x$ would be an equally valid alternative. This suggests that we would need to make choices when extending isometries, which does not mesh well with the result that there is a canonical choice of extensions. This is a flimsy argument (the challenge for readers is to find a better one, such as a negative mathematical result about existence of certain algebraic structures on $\U$), based on constructive intuition, but it was this observation that led me to the notion of disgroups, one point of which is that they have only ``positive direction''.
		\end{remark}
		
		\begin{remark}\label{Remark: unique_choice}
			We mentioned that we did not use any choice principles in the paper, not even countable choice which many constructivists accept. Strictly speaking however, we did use the so-called \df{unique choice} which states that any relation, which is total and single-valued, is the graph of some (necessarily unique) map:
			$$\xall{x}{X}\xexactlyone{y}{Y}{R(x, y)} \implies \xsome{f}{Y^X}\xall{x}{X}{R(x, f(x))}.$$
			Specifically, we used it whenever we invoked Postulate~\ref{Postulate: reals_Dedekind_complete}. However, unique choice is very rarely considered in question, and this goes especially in our case since in practice we actually have the map which realizes Postulate~\ref{Postulate: reals_Dedekind_complete}, but I avoided its formulation in order not the refer to powersets (the existence of which is considered a lot more problematic than the validity of unique choice). If we do use them though (and besides, one can still speak about powerclasses in predicative mathematics), and if we choose two-sided Dedekind cuts of rationals as our model of the reals, then the map we need is given by
			$$(L, U) \mapsto \big(\st{q \in \QQ}{\xsome{x}{L}{q < x}}, \st{q \in \QQ}{\xsome{x}{U}{x < q}}\big).$$
		\end{remark}
		
		\begin{remark}\label{Remark: synthetic_topology}
			I am not an expert on the subject of Urysohn space; my interest in it is rather tangential --- I needed a constructive version of it to prove some results~\cite{Lesnik_D_2010:_synthetic_topology_and_constructive_metric_spaces}(Theorems~{4.56} and~{5.14}) in \df{synthetic topology}, and this paper eventually grew from that.
			
			A synthetic approach to mathematics is to study a structure by creating an axiomatic system which makes that structure an intrisic property of objects (as opposed to the classical approach where basic objects are sets, on which additional structures are added as an afterthought)~\cite{Kock_A_2006:_synthetic_differential_geometry, Phoa_W_1990:_domain_theory_in_realizability_toposes, Escardo_M_2004:_synthetic_topology_of_data_types_and_classical_spaces, Lesnik_D_2010:_synthetic_topology_and_constructive_metric_spaces}. One of the usefulness of this approach is that theorems involving that structure typically become simpler logical statements if not outright tautologies; for example, a synthetic topological proof that if $X$ and $Y$ are compact, so is their product, amounts to nothing more than to observe the equivalence
			$$\xall{p}{X \times Y}{p \in U} \iff \xall{x}{X}\xall{y}{Y}{(x, y) \in U}$$
			for all open subsets $U \subseteq X \times Y$. If the axioms are chosen well, classical theory will embed into suitable synthetic models in a way, that validity of statements is preserved when suitably interpreted at both ends; for example, the category of topological spaces embeds\footnote{More precisely, due to foundational issues we need to restrict to small subcategories (= categories in which objects and morphisms form a set, as opposed to a proper class) of topological spaces, but that turns out to be good enough.} into sheaf topoi. The corollary is that a statement, proven synthetically (presumably in a simpler way), automatically holds classically as well~\cite{Escardo_M_2004:_synthetic_topology_of_data_types_and_classical_spaces, Taylor_P_2006:_computably_based_locally_compact_spaces}. In particular, every map in a synthetic topological model is continuous --- a fact which is proven by the equivalence $x \in f^{-1}(U) \iff f(x) \in U$ --- and this can often be used to prove continuity of maps, which we can construct in the framework of intuitionistic logic without choice principles. We used this in Proposition~\ref{Proposition: scalar_multiplication}, but more generally, this let me know in advance, that extension maps can be shown to be continuous.
			
			However, this isn't to say that the whole of Section~\ref{Section: continuity} can be conveyed in one sentence synthetically. Countinuity does follow if we equip $\isom$ with the topology of pointwise convergence, or of uniform convergence on compact subsets, but we did it more generally, for uniform convergence. Synthetic topological interpretation would be, that the intrinsic topology of $\isom$ is not the subspace topology in $\U^X$ (if $\U^X$ has the exponential topology), but a stronger one (at least as strong as the topology of uniform convergence).
		\end{remark}
		
		\begin{remark}
			What is the merit of the halving map on $\od$? Strictly speaking, we do not actually need a way to produce exactly half of quantity to salvage our results. However, a halving map is a convenient way to ensure that three of the assumptions, that we do need, hold. First, it ensures that $\od$ is a lattice. Second, when $\od$ is a non-trivial subdisgroup of $\nnr$, it is dense in $\nnr$. Third, it makes $\od$-metric spaces into uniform spaces\footnote{Actually, there is another requirement for that: we need to be able to say which elements are positive. See Subsection~\ref{Subsection: completion_of_disgroups} below.} (for the usual fundamental system of entourages $U_a \dfeq \st{(x, y) \in X \times X}{d_\mtr{X}(x, y) \leq a}$ where $a \in \od_{> 0}$). The halving map should in this context be seen as the analogue of the requirement that for every entourage $U$ there exists an entourage $V$ such that if $(x, y) \in V$ and $(y, z) \in V$, then $(x, z) \in U$. To prove this in our case, find a $U_a \subseteq U$, then take $V = U_{\frac{a}{2}}$.
		\end{remark}
		
		\begin{remark}\label{Remark: graded_Urysohn}
			Let $\Ury[n]$ be the Kolmogorov quotient of $\pseudo[n]$ for all $n \in \NN$. Observe that $\Ury[1]$ is isometrically isomorphic to $\od$ --- the inverse isometries being $\od \to \Ury[1]$, $\lambda \mapsto [\strat[1]{(\et, \lambda)}{}]$ and $\Ury[1] \to \od$, $x \mapsto \norm[x]$. This means that we have a retraction $r_1\colon \Ury \to \Ury[1]$, given by
			$$r_1(x) \dfeq [\strat[1]{(\et, \norm[x])}{}].$$
			Moreover, $r_1$ is non-expansive since $d(r_1(x), r_1(y)) = \norm[x] \dis \norm[y] \leq d(x, y)$, and therefore induces a non-expansive retraction (that we'll denote by the same symbol) $r_1\colon \U \to \strat[1]{\U}{\nnr}$. Now $\nnr$ is a disgroup, so it induces ${\dis}_1$ on $\strat[1]{\U}{\nnr} \ism \nnr$, and we have
			$$r_1(x \dis y) = r_1(x) {\dis}_1 r_1(y).$$
			Obviously similar results hold for the unique map $r_0\colon \U \to \strat[0]{\U}{\nnr}$ and the trivial ${\dis}_0$ on the singleton $\strat[0]{\U}{\nnr}$. The question is, can we find a sequence of (non-expansive) retractions $r_n\colon \U \to \strat[n]{\U}{\nnr}$? If so, what properties would the operation ${\dis}_n\colon \strat[n]{\U}{\nnr} \times \strat[n]{\U}{\nnr} \to \strat[n]{\U}{\nnr}$, given by $x {\dis}_n y \dfeq r_n(x \dis y)$, have? One might view ${\dis}_n$s as better and better approximations to an associative disgroup operation. Can we formalize this notion, presumably in the sense, that given a disgroup, it produces another disgroup of which $\dis$ is closer to being associative (along with a map from the original disgroup to the new one)? If so, it would be interesting to see if it leads to another construction of a Urysohn space: start with a disgroup, then apply this ``associativing'' to produce a sequence, of which (the completion of) the colimit should be the Urysohn space.
			
			Another application would be the representation of all elements in $\U$ as tuples, albeit infinite ones. We could identify each $x \in \U$ with the sequence $(r_n(x))_{n \in \NN}$ (or $(r_n(x), d(x, r_n(x)))_{n \in \NN}$, if we wanted tuples like in $\pseudo$) where $x = \lim r_n(x)$. Thus an element would be in $\U_\nnr$ if and only if this sequence was eventually constant while a sequence which is not would represent an element which we genuinely acquired anew when completing $\U_\nnr$. It is in fact this idea that was used for the proof of Proposition~\ref{Proposition: uncompleted_Urysohn_not_complete}.
			
			Note also that this would mean that for every point $x \in \U$ there exists a canonical choice of a sequence in $\U_\nnr$ which converges to $x$. The consequence is that $\U$ would in fact be a Cauchy completion (not just a completion) of $\U_\nnr$ even in the absence of countable choice (of course, that doesn't mean that it would be a Cauchy completion of $\Ury$ for some smaller $\od$, such as $\QQ_{\geq 0}$).
		\end{remark}
		
		\subsection{Multiplication on the Urysohn Space}
		
			We have seen that the Urysohn space has the structure of a module over the disring $\nnr$. Does it hold even more, can we make $\U$ an ``algebra over the disring $\nnr$''? That is, is $\U$ not only a disgroup, but a disring?
			
			There seems to be a very good candidate for multiplication on $\U$. Let us start by defining it on $\pseudo$. If $\cdot$ is a multiplication which preserves the norm, that is $\norm[a \cdot b] = \norm[a] \norm[b]$, then we get for $a = \atuple{a}{\alpha}{i}$, $b = \atuple{b}{\beta}{j} \in \pseudo$
			$$\alpha_i \beta_j = d(a, a_i) d(b, b_j) = \norm[a \dis a_i] \norm[b \dis b_j] = \norm[(a \dis a_i) \cdot (b \dis b_j)] =$$
			$$= \norm[a \cdot b \dis a_i \cdot b \dis a \cdot b_j \dis a_i \cdot b_j] = d(a \cdot b, a_i \cdot b \dis a \cdot b_j \dis a_i \cdot b_j)$$
			which suggests that the reasonable definition of $\cdot$ (inductively on $\age{a} + \age{b}$) is
			$$a \cdot b = \big(a_i \cdot b \dis a \cdot b_j \dis a_i \cdot b_j,\ \alpha_i \beta_j\big)_{(i, j) \in \NN_{< \lnth{a}} \times \NN_{< \lnth{b}}}$$
			(we'll stop writing the ages of the tuples here; just imagine them to be, say, the smallest possible ones). Observe that $()$ acts as zero, $() \cdot a = a \cdot () = ()$, as it should, being the unit for $\dis$. The tuple $((), 1)$ is the unit for multiplication, and more generally, for every $\lambda \in \od$ we have $((), \lambda) \cdot a = \lambda \cdot a$. Thus $\cdot$ can be seen as an extension of the scalar multiplication (recall from Remark~\ref{Remark: graded_Urysohn} that $\Ury[1] \ism \od$).
			
			Let $\equ$ be the Kolmogorov equivalence relation. Then $x \dis y \equ y \dis x$, so clearly $\cdot$ is commutative in the sense $a \cdot b \equ b \cdot a$ (use induction on $\age{a} + \age{b}$). As for distributivity, by induction on $\age{a} + \age{b} + \age{c}$ we have
			$$(a \dis b) \cdot c = \big((a_i \dis b,  \alpha_i)_{i \in \NN_{< \lnth{a}}} \cnct (a \dis b_j, \beta_j)_{j \in \NN_{< \lnth{b}}}\big) \cdot c =$$
			$$= \Big(\big((a_i \dis b) \cdot c \dis (a \dis b) \cdot c_k \dis (a_i \dis b) \cdot c_k,\ \alpha_i \gamma_k)_{(i, k) \in \NN_{< \lnth{a}} \times \NN_{< \lnth{c}}}\big) \cnct$$
			$$\cnct \big((a \dis b_j) \cdot c \dis (a \dis b) \cdot c_k \dis (a \dis b_j) \cdot c_k,\ \beta_j \gamma_k)_{(j, k) \in \NN_{< \lnth{b}} \times \NN_{< \lnth{c}}}\big)\Big) \equ$$
			$$\equ \Big(\big(a_i \cdot c \dis b \cdot c \dis a \cdot c_k \dis b \cdot c_k \dis a_i \cdot c_k \dis b \cdot c_k,\ \alpha_i \gamma_k)_{(i, k) \in \NN_{< \lnth{a}} \times \NN_{< \lnth{c}}}\big) \cnct$$
			$$\cnct \big(a \cdot c \dis b_j \cdot c \dis a \cdot c_k \dis b \cdot c_k \dis a \cdot c_k \dis b_j \cdot c_k,\ \beta_j \gamma_k)_{(j, k) \in \NN_{< \lnth{b}} \times \NN_{< \lnth{c}}}\big)\Big) \equ$$
			$$\equ \Big(\big((a_i \cdot c \dis a \cdot c_k \dis a_i \cdot c_k) \dis b \cdot c,\ \alpha_i \gamma_k)_{(i, k) \in \NN_{< \lnth{a}} \times \NN_{< \lnth{c}}}\big) \cnct$$
			$$\cnct \big(a \cdot c \dis (b_j \cdot c \dis b \cdot c_k \dis b_j \cdot c_k),\ \beta_j \gamma_k)_{(j, k) \in \NN_{< \lnth{b}} \times \NN_{< \lnth{c}}}\big)\Big) =$$
			$$= a \cdot c \dis b \cdot c.$$
			So $\cdot$ would induce a disring structure on $\Ury$\ldots\ if it induced a map at all. Not only does $\cdot$ not actually preserve the norm, it does not even respect $\equ$ (that is, replacing factors by equivalent ones does not always yield an equivalent product). We can prove a more general negative result.
			
			\begin{proposition}
				There is no operation $\cdot$ on $\pseudo$ (where $\od$ is a non-trivial halved disring) which satisfies $((), \lambda) \cdot a \equ \lambda \cdot a$, respects $\equ$ and induces a disring structure on $\Ury$.
			\end{proposition}
			\begin{proof}
				First observe a general fact, that $\big((), \lambda\big) \equ \big(\big((), \lambda + \mu\big), \mu\big)$ for all $\lambda, \mu \in \od$. Using this, we obtain
				$$\big((), 1\big) \equ \big((), 1\big) \cdot \big((), 1\big) \equ \big(\big((), 2\big), 1\big) \cdot \big(\big((), 2\big), 1\big) \equ$$
				$$\equ \Big(\big((), 2\big) \cdot \big(\big((), 2\big), 1\big) \dis \big(\big((), 2\big), 1\big) \cdot \big((), 2\big) \dis \big((), 2\big) \cdot \big((), 2\big),\ 1 \cdot 1\Big) \equ$$
				$$\equ \Big(\big((), 2\big) \cdot \big((), 2\big),\ 1\Big) \equ \big(\big((), 4\big), 1\big) \equ \big((), 3\big),$$
				a contradiction.
			\end{proof}
			
			That said, a couple of questions remains. Is the above defined $\cdot$ of any use, even if it does not induce an operation on $\U$? Is there a reasonable multiplicative structure on $\U$?
			
			If there happens to be one, we can go one step further, defining division with the help of \df{Neumann series}. For $x \in \U$ define inductively $a_0(x) \dfeq 1$, $a_{n+1}(x) \dfeq a_n(x) \dis x^{n+1}$. Then $(1 \dis x) \cdot a_n(x) = 1 \dis x^{n+1}$ which is a Cauchy sequence (in the complete metric space $\U$) when $\norm[x] < 1$, and in this case $(1 \dis x)^{-1} = \lim a_n(x)$.
			
			Let us extend this beyond the unit ball around $1$. Let $x = [a]$ where $a \in \pseudo$. First we need an upper bound on the terms in $a$; define inductively on $\age{a}$
			$$u(a) \dfeq \sup\big(\st{u(a_i)}{i \in \NN_{< \lnth{a}}} \cup \st{\alpha_i}{i \in \NN_{< \lnth{a}}}\big).$$
			It is easy to see that $\norm[a] \leq u(a)$, though unlike the norm, $u$ does not respect the Kolmogorov equivalence relation (but we won't need it to).
			
			In the following lemma suppose that the multiplicative unit $1$ is actually $((), 1)$.
			\begin{lemma}
				Let $n \in \NN$ be large enough so that $2 u(a) \leq 2^n$. Then
				$$d(1, 2^{-n} \cdot a) \leq 1 - 2^{-n} \norm[a].$$
			\end{lemma}
			\begin{proof}
				By induction on $\age{a}$. Note that the condition $2 u(a) \leq 2^n$ implies $2^{-n} \alpha_i \leq \frac{1}{2}$ and $2^{-n} \norm[a] \leq \frac{1}{2}$, the latter of which implies $d(1, 2^{-n} \cdot a_i) \geq \norm[1] \dis \norm[2^{-n} \cdot a_i] \geq \frac{1}{2}$.
				$$d((), 2^{-n} \cdot a) \dis 1 = \norm[2^{-n} \cdot a] \dis 1 = 1 - 2^{-n} \norm[a]$$
				$$d(1, 2^{-n} \cdot a_i) \dis 2^{-n} \cdot \alpha_i = d(1, 2^{-n} \cdot a_i) - 2^{-n} \cdot \alpha_i \leq$$
				$$\leq 1 - 2^{-n} \norm[a_i] - 2^{-n} \alpha_i \leq 1 - 2^{-n} \norm[a]$$
			\end{proof}
			
			We can write $2^{-n} \cdot [a] = 1 \dis 1 \dis 2^{-n} \cdot [a]$ and if $\norm[a] > 0$, then by this lemma we obtain $d(1, 2^{-n} \cdot a) = \norm[1 \dis 2^{-n} \cdot a] < 1$, so we can calculate the inverse of $1 \dis 1 \dis 2^{-n} \cdot [a]$ via Neumann series, obtaining the inverse of $x = [a]$ itself:
			$$x^{-1} = 2^{-n} \cdot (2^{-n} \cdot [a])^{-1} = 2^{-n} \cdot (1 \dis 1 \dis 2^{-n} \cdot [a])^{-1}.$$
			This would make $\U$ a ``disfield'', and since $\dis$ is associative, in fact a field of characteristic $2$.
			
			A possible alternative argument for this result might be to construct a field of fractions over $\U$, then verify, that it also has the Urysohn extension property, thus making it isometrically isomorphic to $\U$.
			
			If we had this structure on $\U$, it would offer us a simple way to generalize the homogeneity result from Proposition~\ref{Proposition: Urysohn_strongly_homogeneous}, allowing us to swap not just two points, but two tuples of points (having the same interdistances). Take a simple case, suppose we want an automorphism of $\U$ which fixes $0$ but maps $1$ to a point $z$ on the unit sphere. Then the solution is simply the map $x \mapsto z \cdot x$, with the inverse $x \mapsto z^{-1} \cdot x$. More general swaping of pairs could be achieved by linear maps, and for swaping general tuples one could use analogues to Lagrange polynomials.
			
			Assuming that there is a reasonable multiplicative structure on $\U$, how much of this discussion could be salvaged?
		
		\subsection{Implication in Disgroups}\label{Subsection: disgroup_implication}
		
			Let $(X, +, 0, \dis)$ be a disgroup and $\leq$ \emph{any} partial order on $X$ (not necessarily the one induced by $+$ and $\dis$). Suppose futher that $X$ has abitrary finite suprema in this order, and that $0$ is its nullary supremum --- the smallest element.
			
			Under these conditions we may define a binary operation $\fdis\colon X \times X \to X$ for $a, b \in X$ by
			$$a \fdis b \dfeq \sup\{a, b\} \dis a.$$
			Interpreting this in $\nnr$ (under the usual order), $a \fdis b$ is the positive part of the difference $b - a$, or to put it differently, it tells us, how far forward from $a$ must we go to exceed $b$.
			
			One can verify that the formulas
			$$a \fdis a = 0 \qquad a \fdis 0 = 0 \qquad 0 \fdis a = a \qquad a \fdis b \leq b \qquad a \fdis b = 0 \iff b \leq a$$
			hold, and when $\leq$ is the usual order on a disgroup, so do the following ones.
			$$a + a \fdis b = \sup\{a, b\} \qquad a \fdis b + b \fdis c \geq a \fdis c$$
			$$a \dis b = \sup\{a \fdis b, b \fdis a\} = a \fdis b + b \fdis a$$
			$$a \fdis \sup\{b, c\} = \sup\{a \fdis b, a \fdis c\}$$
			
			These formulas are reminiscent of some from the propositional calculus (interpret $\dis$ as equivalence, $\fdis$ as implication, $\sup$ and $+$ as conjunction, $0$ as truth). An aspect of this is the following proposition.
			
			\begin{proposition}\label{Proposition: Boolean_lattices_disrings}
				A Boolean lattice is an associative disring in which $a + b = a \dis b = a \Leftrightarrow b$, $0 = \top$, $a \cdot b = a \lor b$, $1 = \bot$. Furthermore, if $\leq$ is the order, opposite to the usual one in a Boolean lattice, then $\sup\{a, b\} = a \land b$ and $a \fdis b = a \Rightarrow b$.
			\end{proposition}
			\begin{proof}
				Exercise.
			\end{proof}
			
			Of course, since a Boolean lattice is also a Boolean ring, it comes as no surprise to be an associative disring as well, but note that the disring structure in the above proposition differs from the Boolean ring structure where $a + b = a \veebar b = \lnot(a \Leftrightarrow b)$, $0 = \bot$, $a \cdot b = a \land b$, $1 = \top$ (it is actually opposite to it).
			
			Boolean lattices are models of classical propositional calculus. Their constructive analogue are Heyting lattices~\cite{Gierz_G_Hoffmann_KH_Keimel_K_Lawson_JD_Mislove_MW_Scott_DS_2003:_continuous_lattices_and_domains}. Does the above proposition hold for them as well? No; in general, the equivalence is not associative. Still, one can observe, that the stable part of a Heyting lattice is an associative disgroup for operations given as in Proposition~\ref{Proposition: Boolean_lattices_disrings} (but not a disring in general, since it is usually not closed for disjunctions).
		
		\subsection{Completion of Disgroups}\label{Subsection: completion_of_disgroups}
		
			When discussing metric completions, we restricted ourselves to $\od$-(pseudo)metric spaces, where $\od \subseteq \nnr$. The main reason for this is that we need the strict order $<$ (or at least the comparison $0 <$) to even define density. General disgroups do not have a suitable such relation; in this remark we propose what ``suitable'' means in this case, and then consider in what way this enables us to generalize the notion of metric completion.
			
			\begin{definition}\label{Definition: strictly_ordered_disgroups}
				Call $(X, +, 0, \dis, 0 <)$ a \df{strictly ordered disgroup} when $(X, +, 0, \dis)$ is a disgroup and the unary relation $0 <$ on $X$ satisfies the following conditions for all $a, b, x \in X$.
				\begin{itemize}
					\item
						$\lnot(0 < 0)$ \quad (ireflexivity)
					\item
						$0 < a \implies 0 < x \lor x \leq a$ \quad (cotransitivity)
					\item
						$0 < a \land a \leq b \implies 0 < b$ \quad (transitivity)
					\item
						$0 < a + a \implies 0 < a$
					\item
						$\lnot(0 < a) \implies a = 0$ \quad (tightness)
				\end{itemize}
			\end{definition}
			Note that ireflexivity and tightness together can be given as $\lnot(0 < a) \iff a = 0$.
			
			Examples from Remark~\ref{Remark: examples_of_disgroups} are also examples of strictly ordered disgroups: for $\nnr$ and its subdisgroups, take the usual strict order, for positive semidefinite matrices declare $0 < A$ when there exists a vector $x$ in the domain of $A$ such that $\langle A x, x\rangle > 0$ (equivalently, when $A$ has a positive eigenvalue), and for function disgroups declare $0 < f$ when $f$ attains a positive value.
			
			\begin{proposition}
			Let  $(X, +, 0, \dis, 0 <)$ be a strictly ordered disgroup. Then the following holds.
			\begin{enumerate}
				\item
					$0 < a \implies 0 < a + b$ for all $a, b \in X$.
				\item
					$\leq$ is a partial order.
			\end{enumerate}
			\end{proposition}
			\begin{proof}
				\begin{enumerate}
					\item
						By transitivity since $0 \leq b$, so $a \leq a + b$.
					\item
						By the previous item, if $0 < a$, then $0 < a + a$. The contrapositive $a + a = 0 \implies a = 0$ proves antisymmetry of $\leq$ by Proposition~\ref{Proposition: when_disgroup_is_partially_ordered}.
				\end{enumerate}
			\end{proof}
			
			Hereafter we restrict our attention to halved disgroups.\footnote{To be honest, I don't think the conditions in Definition~\ref{Definition: strictly_ordered_disgroups} work well outside the scope of halved disgroups; some other conditions might be required for more general theory.}
			
			\begin{lemma}\label{Lemma: descent_in_strictly_ordered_disgroup}
				Let  $(X, +, 0, \dis, 0 <)$ be a strictly ordered halved disgroup and $x \in X$. Suppose $\xall{\epsilon}{X_{> 0}}{x \leq \epsilon}$; then $x = 0$.
			\end{lemma}
			\begin{proof}
				Suppose $0 < x$; then also $0 < \frac{x}{2}$, so by assumption $\frac{x}{2} \leq x$, and therefore $x \leq x + x$. Cancelling $x$ on both sides and taking antisymmetry into account, we obtain $x = 0$, a contradiction to ireflexivity. Thus $x = 0$ by tightness.
			\end{proof}
			
			Classically ireflexivity and tightness determine $0 <$ uniquely, namely $0 < a \iff a \neq 0$, and if $\leq$ is a partial order, they imply the other conditions from Definition~\ref{Definition: strictly_ordered_disgroups}. Constructively it is not so simple; the situation is similar to that of apartness relation which intuitively states the difference of elements in a positive way. More formally, recall~\cite{Troelstra_AS_Dalen_D_1988:_constructivism_in_mathematics_volume_2} that a binary relation $\apart$ on $X$ is called \df{apartness} when the following conditions are satisfied.
			\begin{itemize}
				\item
					$\lnot(a \apart a)$ \quad (ireflexivity)
				\item
					$a \apart b \iff b \apart a$ \quad (symmetry)
				\item
					$a \apart b \implies a \apart x \lor x \apart b$ \quad (cotransitivity)
			\end{itemize}
			If furthermore
			\begin{itemize}
				\item
					$\lnot(a \apart b) \implies a = b$ \quad (tightness)
			\end{itemize}
			is satisfied, then $\apart$ is called a \df{tight} apartness.
			
			An example of a tight apartness is $a \apart b \iff a < b \lor a > b \iff |b - a| > 0$ on the reals $\RR$. A real is invertible if and only if it is apart from $0$.
			
			\begin{proposition}
				The relation $0 <$ on a disgroup $\od$ induces a tight apartness on $\od$ by $a \apart b \dfeq 0 < a \dis b$, and more generally on any $\od$-metric space $(X, d)$ by $a \apart b \dfeq 0 < d(a, b)$.
			\end{proposition}
			\begin{proof}
				Irefelexivity, symmetry and tightness are immediate. For cotransitivity assume $a \apart b$, that is $0 < a \dis b$, hence $0 < \frac{a \dis b}{2}$. Then by cotransitivity of $0 <$ we have
				$$0 < a \dis x \lor a \dis x \leq \frac{a \dis b}{2}.$$
				If the first condition holds, we are done since $a \apart x$. If the second one does, then
				$$a \dis b \leq a \dis x + b \dis x \leq \frac{a \dis b}{2} + b \dis x$$
				whence $0 < \frac{a \dis b}{2} \leq b \dis x$, so $b \apart x$.
			\end{proof}
			
			We can construct a model of a completion for a strictly ordered disgroup $\od$. At this point we assume the existence of the powerset $\pst(\od \times \od)$.
			
			Let
			$$\widetilde{\od} \dfeq \st{A \in \pst(\od \times \od)}{\xall{\epsilon}{\od_{> 0}}\xsome{(u, v)}{A}{u \leq \epsilon}}.$$
			Define the relation $\preceq$ on $\widetilde{\od}$ by
			$$A \preceq B \dfeq \xall{(u, v)}{A}\xall{(w, z)}{B}{u \dis v \leq w + z}$$
			and $\equ$ by $A \equ B \dfeq A \preceq B \land B \preceq A$.
			\begin{proposition}
				The relation $\preceq$ (and therefore also $\equ$) is transitive.
			\end{proposition}
			\begin{proof}
				Suppose $A \preceq B$, $B \preceq C$, and $(u, v) \in A$, $(x, y) \in C$, $\epsilon \in \od_{> 0}$ arbitrary. Then there exists $(w, z) \in B$ such that $w \leq \frac{\epsilon}{2}$, so
				$$u \dis v \leq w + z \leq w + w + w \dis z \leq w \dis z + \epsilon \leq x + y + \epsilon.$$
				By Lemma~\ref{Lemma: descent_in_strictly_ordered_disgroup} $u \dis v \leq x + y$.
				
				The transitivity of $\equ$ follows easily from the transitivity of $\preceq$.
			\end{proof}
			Obviously $\equ$ is also symmetric, so a partial equivalence relation, and thus an equivalence relation on its domain $\st[1]{A \in \widetilde{\od}}{A \equ A}$. Define $\cmpl{\od}$ to be the set of equivalence classes (the quotient set) of the domain.
			
			The idea is that $[A]$ represents a point such that $\st{u \dis v}{(u, v) \in A}$ are (some of) its lower bounds and $\st{u + v}{(u, v) \in A}$ its upper bounds. The defining property of $\widetilde{\od}$ ensures that these bounds are arbitrarily good approximations.
			
			Every equivalence class in $\cmpl{\od}$ has a canonical representative, namely its maximal one:
			$$A \equ \st[1]{(u, v) \in \od \times \od}{\all{(a, b)}{A}{u \dis v \leq a + b \land a \dis b \leq u + v}}.$$
			
			Note that for any $[A] \in \cmpl{\od}$ the set $A$ contains at most one element of the form $(0, a)$; if it contains also $(0, b)$, then $0 \dis a \leq 0 + b$ and vice versa, so $a = b$. As such, the map $i\colon \od \to \cmpl{\od}$, given by
			$$i(x) \dfeq \big[\{(0, x)\}\big] = \big[\st{(u, v) \in \od \times \od}{u \dis v \leq x \leq u + v}\big],$$
			is injective; in fact, its image can be identified with those equivalence classes which contain a representative containing an element of the form $(0, x)$. Via this embedding we will consider $\od$ to be a subset of $\cmpl{\od}$.
			
			For $[A], [B] \in \cmpl{\od}$ declare $[A] \leq [B] \dfeq A \preceq B$. It is clear that this is a partial order on $\cmpl{\od}$; note that it extends the partial order on $\od$.
			\begin{lemma}
				The following is equivalent for every $[A], [B] \in \cmpl{\od}$.
				\begin{enumerate}
					\item
						$[A] \leq [B]$
					\item
						$\all{x}{\od}{i(x) \leq [A] \implies i(x) \leq [B]}$
					\item
						$\all{x}{\od}{B \leq i(x) \implies [A] \leq i(x)}$
				\end{enumerate}
			\end{lemma}
			\begin{proof}
				$(1)$ implies $(2)$ by transitivity. Conversely, assume $(2)$ and take arbitrary $(u, v) \in A$, $(w, z) \in B$. Then $i(u \dis v) \leq [A]$, so $i(u \dis v) \leq [B]$ which means $0 \dis (u \dis v) \leq w + z$.
				
				The equivalence $(1) \iff (3)$ is proved analogously.
			\end{proof}
			
			\begin{lemma}\label{Lemma: extension_of_disgroup_operations_to_the_completion}
				Let $f\colon \od^n \to \od$ be a monotone operation, \ie for $(x_k)_{k \in \NN_{< n}}, (y_k)_{k \in \NN_{< n}} \in \od^n$, if $x_k \leq y_k$ for all $k \in \NN_{< n}$, then $f((x_k)_{k \in \NN_{< n}}) \leq f((y_k)_{k \in \NN_{< n}})$. Then there exists a unique monotone extension $\cmpl{f}\colon \cmpl{\od}^n \to \cmpl{\od}$ of $f$,
				\begin{enumerate}
					\item\label{Lemma: extension_of_disgroup_operations_to_the_completion: monotone}
						given by
						$$\cmpl{f}(([A_k])_{k \in \NN_{< n}}) =$$
						\mlst{= \Big[}{(x, y) \in \od \times \od}{\xall{k}{\NN_{< n}}\xall{(u_k, v_k)}{A_k}{\big(x \dis y \leq f((u_k + v_k)_{k \in \NN_{< n}})} \nl f((u_k \dis v_k)_{k \in \NN_{< n}}) \leq x + y\big)}{\Big];}
					\item
						if $f$ is moreover a subadditive operation, \ie it satisfies
						$$f((x_k + y_k)_{k \in \NN_{< n}}) \leq f((x_k)_{k \in \NN_{< n}}) + f((y_k)_{k \in \NN_{< n}})$$
						for all $(x_k)_{k \in \NN_{< n}}, (y_k)_{k \in \NN_{< n}} \in \od^n$, then $\cmpl{f}$ can be given simplier as
						$$\cmpl{f}(([A_k])_{k \in \NN_{< n}}) = \big[\st{(f((u_k)_{k \in \NN_{< n}}), f((v_k)_{k \in \NN_{< n}}))}{\xall{k}{\NN_{< n}}{(u_k, v_k) \in A_k}}\big].$$
				\end{enumerate}
			\end{lemma}
			\begin{proof}
				Exercise (for uniqueness use the previous lemma).
			\end{proof}
			
			Since $+$, $\frac{\insarg}{2}$ and $\sup$ are monotone and subadditive, they extend to $\cmpl{\od}$ as
			$$[A] + [B] = \big[\st{(u + w, v + z)}{(u, v) \in A \land (w, z) \in B}\big]$$
			$$\tfrac{[A]}{2} = \big[\st{(\tfrac{u}{2}, \tfrac{v}{2})}{(u, v) \in A}\big]$$
			$$\sup\{[A], [B]\} = \big[\st{(\sup\{u, w\}, \sup\{v, z\}}{(u, v) \in A \land (w, z) \in B}\big]$$
			(the fact that $\sup$ is indeed the supremum on $\cmpl{\od}$ follows from uniqueness). Clearly then the unit for addition is
			$$0 = i(0) = [\{(0, 0)\}] = [\st{(x, x)}{x \in X}],$$
			and we can define the extension of $0 <$ by
			$$0 < [A] \iff \xsome{(u, v)}{A}{u \apart v}.$$
			The operation $\dis$ is more difficult because it is not monotone, but we can sidestep that problem by recalling the operation $\fdis$ from Subsection~\ref{Subsection: disgroup_implication} which is monotone in the second argument and antitone in the first, so by a similar reasoning as in Lemma~\ref{Lemma: extension_of_disgroup_operations_to_the_completion}(\ref{Lemma: extension_of_disgroup_operations_to_the_completion: monotone})
			\mlst{[A] \fdis [B] = \Big[}{(x, y) \in \od \times \od}{\xall{(u, v)}{A}\xall{(w, z)}{B}{\big(x \dis y \leq (u \dis v) \fdis (w + z)} \nl (u + v) \fdis (w \dis z) \leq x + y\big)}{\Big],}
			$$[A] \dis [B] = \sup\{[A] \fdis [B], [B] \fdis [A]\}.$$
			This formula is rather complicated; the question is, can it be simplified? Regardless, we'll skip the technical verification that these operations make $\cmpl{\od}$ into a strictly ordered halved disgroup.
			
			It is reasonable to call $\cmpl{\od}$ the completion of $\od$ due to the following proposition (in the proof of which we again skip the technical verification).
			\begin{proposition}
				$\cmpl{\od}$ is the largest disgroup of which $\od$ is a dense subdisgroup.
			\end{proposition}
			\begin{proof}
				First we show that $i$ has a dense image. Take any $[A] \in \cmpl{\od}$ and $[E] \in \cmpl{\od}_{> 0}$. We thus have $(x, y) \in E$ such that $x \apart y$, and we may find $(u, v) \in A$ such that $u \leq x \dis y$. Then $[A] \dis i(v) \leq i(u) \leq i(x \dis y) \leq [E]$.
				
				Let $\od'$ be another disgroup and $j\colon \od \to \od'$ an injective map which preserves disgroup structure and has a dense image. Then $k\colon \od' \to \cmpl{\od}$, given by
				$$k(x) \dfeq \big[\st{(u, v) \in \od \times \od}{j(u \dis v) \leq x \leq j(u + v)}\big]$$
				has the property $k \circ j = i$.
			\end{proof}
			Consequently, the completion of $\cmpl{\od}$ is (isomorphic to) $\cmpl{\od}$.
			
			Having the completion of the base disgroup $\od$, we can construct completions of general $\od$-(proto)metric spaces via locations (recall Section~\ref{Section: reals_and_completions}). Hereafter, assume that $\od$ is complete.
			\begin{definition}
				Let $\mtr{X} = (X, d_\mtr{X})$ be a $\od$-(pseudo)metric space. A map $f\colon X \to \od$ is a \df{location} on $\mtr{X}$ when
				\begin{itemize}
					\item
						$d_\mtr{X}(x, y) \dis f(x) \leq f(y)$ for all $x, y \in X$, and
					\item
						$\xall{\epsilon}{\od_{> 0}}\xsome{x}{X}{f(x) \leq \epsilon}$.
				\end{itemize}
				We denote the set of locations on $\mtr{X}$ by $\locd(\mtr{X})$.
			\end{definition}
			$\locd(\mtr{X})$ is a $\od$-metric space, with the metric $d_{\locd(\mtr{X})}\colon \locd(\mtr{X}) \times \locd(\mtr{X}) \to \cmpl{\od} \ism \od$ given as
			$$d_{\locd(\mtr{X})}(f, g) \dfeq \big[\st{(f(x), g(x))}{x \in X}\big].$$
			From here we can proceed much as in Section~\ref{Section: reals_and_completions}.
			
			The point of this exercise is the following question. As we have seen in Section~\ref{Section: countable_Urysohn}, we can make a ``countable version of a Urysohn space'' over any partially ordered disgroup $\od$ with finite suprema. Here we've seen that if $\od$ is halved and strictly ordered, we can complete this countable version. An adaptation of arguments from Section~\ref{Section: complete_Urysohn} should show that the completion would satisfy the Urysohn extension property. So, are there strictly ordered halved disgroups, other than such subdisgroups of $\nnr$, over which the Urysohn space would be of interest?

	\bibliographystyle{plain}
	\bibliography{Ury_Bibliography}

\end{document}